\documentclass[a4paper,12pt,reqno,english]{amsart}
\usepackage[utf8]{inputenc}
\usepackage[T1]{fontenc}
\usepackage{babel}

\usepackage[bookmarks = true, colorlinks, citecolor = blue, urlcolor = blue]{hyperref}
\usepackage[capitalize, nameinlink]{cleveref}
\usepackage[titletoc, toc, title]{appendix}

\hypersetup{pdfstartview=}

\usepackage{amssymb}
\usepackage{tikz}
\usepackage{tikz-cd}

\numberwithin{equation}{section}
\numberwithin{figure}{section}

\theoremstyle{plain}
\newtheorem{theorem}{Theorem}[section]

\theoremstyle{plain}
\newtheorem*{theorem*}{Theorem}

\theoremstyle{plain}
\newtheorem{proposition}[theorem]{Proposition}

\theoremstyle{plain}
\newtheorem{lemma}[theorem]{Lemma}

\theoremstyle{plain}
\newtheorem{corollary}[theorem]{Corollary}

\theoremstyle{definition}
\newtheorem{definition}[theorem]{Definition}

\theoremstyle{definition}

\theoremstyle{definition}

\theoremstyle{remark}
\newtheorem{remark}[theorem]{Remark}

\addto\extrasenglish{
  
}

\newcommand{\Uqg}{U_q(\mathfrak{g})}
\newcommand{\Uqu}{U_q(\mathfrak{u})}
\newcommand{\UqlS}{U_q(\mathfrak{l}_S)}

\newcommand{\UqkS}{U_q(\mathfrak{k}_S)}

\newcommand{\CqG}{\mathbb{C}_q[G]}
\newcommand{\CqU}{\mathbb{C}_q[U]}
\newcommand{\Cqflag}{\mathbb{C}_q[U / K_S]}

\newcommand{\id}{\mathrm{id}}

\newcommand{\bbC}{\mathbb{C}}
\newcommand{\bbZ}{\mathbb{Z}}

\newcommand{\diff}{\mathrm{d}}
\newcommand{\del}{\partial}
\newcommand{\delbar}{\bar{\partial}}

\newcommand{\ev}{\mathsf{ev}}
\newcommand{\coev}{\mathsf{coev}}

\newcommand{\EV}{\mathsf{E}}
\newcommand{\EVp}{\mathsf{E}^\prime}
\newcommand{\CV}{\mathsf{C}}
\newcommand{\CVp}{\mathsf{C}^\prime}

\newcommand{\braid}{\hat{\mathsf{R}}}
\newcommand{\braidP}{\mathsf{P}}
\newcommand{\braidQ}{\mathsf{Q}}

\newcommand{\met}{g}
\newcommand{\metMP}{g_{- +}}
\newcommand{\metPM}{g_{+ -}}

\newcommand{\Tnabla}{T_\nabla}
\newcommand{\coTnabla}{\mathrm{coT}_\nabla}

\newcommand{\calc}{\Omega}
\newcommand{\calcM}{\Omega_-}
\newcommand{\calcP}{\Omega_+}
\newcommand{\overcalcM}{\Gamma_-}
\newcommand{\overcalcP}{\Gamma_+}
\newcommand{\FovercalcM}{\tilde{\Gamma}_-}
\newcommand{\FovercalcP}{\tilde{\Gamma}_+}

\newcommand{\alg}{\mathcal{A}}
\newcommand{\subalg}{\mathcal{B}}

\newcommand{\conn}{\nabla}
\newcommand{\connM}{\nabla_-}
\newcommand{\connP}{\nabla_+}

\newcommand{\opS}{\mathsf{S}}
\newcommand{\opSt}{\tilde{\mathsf{S}}}
\newcommand{\opT}{\mathsf{T}}

\newcommand{\PhiMP}{\Phi_{- +}}
\newcommand{\PhiPM}{\Phi_{+ -}}
\newcommand{\PsiMP}{\Psi_{- +}}
\newcommand{\PsiPM}{\Psi_{+ -}}

\newcommand{\lieg}{\mathfrak{g}}

\newcommand{\liel}{\mathfrak{l}}

\newcommand{\lieu}{\mathfrak{u}}

\newcommand{\sigmaMM}{\sigma_{- -}}
\newcommand{\sigmaMP}{\sigma_{- +}}
\newcommand{\sigmaPM}{\sigma_{+ -}}
\newcommand{\sigmaPP}{\sigma_{+ +}}

\newcommand{\elemMPL}{M_{- +}^L}
\newcommand{\elemMPR}{M_{- +}^R}
\newcommand{\elemPML}{M_{+ -}^L}
\newcommand{\elemPMR}{M_{+ -}^R}

\begin{document}

\title[Metrics and connections on quantum projective spaces]{Fubini-Study metrics and Levi-Civita connections on quantum projective spaces}

\author{Marco Matassa}

\address{OsloMet – Oslo Metropolitan University}

\email{marco.matassa@oslomet.no}

\begin{abstract}
We introduce analogues of the Fubini-Study metrics and the corresponding Levi-Civita connections on quantum projective spaces.
We define the quantum metrics as two-tensors, symmetric in the appropriate sense, in terms of the differential calculi introduced by Heckenberger and Kolb.
We define connections on these calculi and show that they are torsion free and cotorsion free, where the latter condition uses the quantum metric and is a weaker notion of metric compatibility.
Finally we show that these connections are bimodule connections and that the metric compatibility also holds in a stronger sense.
\end{abstract}

\maketitle

\section*{Introduction}

Metrics and connections are two of the cornerstones upon which our description of differential geometry is built, hence it is desirable to extend these notions to the realm of quantum spaces.
By quantum spaces, we mean a class of appropriately defined non-commutative algebras, which we interpret as quantizations of functions on the underlying classical spaces.
There are various possible perspectives on this problem and we recall some of them below.
The goal of this paper is to introduce certain appropriate analogues of the Fubini-Study metrics and the corresponding Levi-Civita connections for the \emph{quantum projective spaces}. This generalizes certain results of \cite{majid-sphere} obtained in the case of the quantum two-sphere.

Given a (unital) non-commutative algebra $A$, one possible approach to introduce a metric is the theory of \emph{compact quantum metric spaces} \cite{rieffel}, developed by Rieffel following the ideas of Connes.
In this theory one introduces a metric on the state space of $A$ in terms of an appropriately defined Dirac operator, which should satisfy some properties.
Such Dirac operators are readily available for quantum projective spaces, see \cite{dirac-projective}.
Roughly speaking, what is being quantized in this approach is the distance between points, since in the commutative situation the points can be identified with the pure states.
Instead we are looking for a quantization of the metric tensor, since we want to have some notion of compatibility between a connection and a metric.
For this reason we adopt a more algebraic approach, which is explained for instance in the recent book \cite{quantum-book} by Beggs and Majid.

Let us recall some of the ideas of this approach, which we refer to as \emph{quantum Riemannian geometry}.
Given an algebra $A$, we begin by introducing a differential calculus $\Omega^\bullet$ over $A$, with its degree-one part denoted by $\Omega^1$.
Then a quantum metric can be defined as an element $g \in \Omega^1 \otimes_A \Omega^1$ satisfying an appropriate invertibility condition.
Using the differential calculus, we can also define connections in the standard algebraic sense.
In particular, given a connection $\nabla$ on $\Omega^1$, there is a standard notion of torsion as well.
To formulate an analogue of the compatibility of $\nabla$ with the metric $g$ there are two possibilities: 1) a weak version which uses the notion of cotorsion, due to Majid; 2) a strong version that requires $\nabla$ to be a bimodule connection.
In the classical case the second version coincides with the usual metric compatibility, while the first version is a weaker property (to be recalled later).

This setup can be applied to the quantum projective spaces, which we regard as a family within the class of \emph{quantum irreducible flag manifolds}.
It turns out that all the quantum spaces in this class admit canonical differential calculi $\Omega^\bullet$, introduced by Heckenberger and Kolb in \cite{locally-finite, heko}.
We refer to these calculi as canonical since, as soon as some natural conditions are imposed, they are uniquely defined.
These quantum spaces and their differential calculi admit a uniform description, which we adopt in this paper, making simplifications relative to the quantum projective spaces only when needed.
We expect that the results obtained in this paper will hold more generally for all quantum irreducible flag manifolds, with those obtained here providing important steps in this direction.

Having the calculi $\calc^\bullet$ at our disposal, we can discuss quantum metrics and connections on them.
We denote by $\subalg$ the algebra of a generic quantum projective space and write $\calc = \calc^1$.
Our first main result is the existence of quantum metrics in the sense of \cref{def:quantum-metric}, which also requires the existence of appropriate inverse metrics.

\begin{theorem*}[\cref{thm:quantum-metric}]
Any quantum projective space $\subalg$ admits a quantum metric $\met \in \calc \otimes_\subalg \calc$.
Moreover, in the classical limit it reduces to the Fubini-Study metric.
\end{theorem*}

Next, we look at connections on the first-order differential calculi $\calc$.
We show the existence of some particular connections and investigate the properties of torsion and cotorsion.
The latter involves the quantum metric $\met$ introduced above.
In particular, the condition of cotorsion freeness should be seen as a weaker notion of compatibility with the metric (see \cref{def:cotorsion} and the remarks after that).
Our second main result is the following.

\begin{theorem*}[\cref{thm:levi-civita}]
Any quantum projective space $\subalg$ admits a connection $\conn: \calc \to \calc \otimes_\subalg \calc$ which is torsion free and cotorsion free.
Moreover, in the classical limit it reduces to the Levi-Civita connection for the Fubini-Study metric on the cotangent bundle.
\end{theorem*}

A connection which is torsion and cotorsion free is called a weak quantum Levi-Civita connection in \cite{quantum-book}, since the ordinary Levi-Civita connection (on the tangent bundle) can be characterized as the unique connection which is torsion free and compatible with the metric.
It is natural to ask whether $\conn$ is a bimodule connection and if the condition of metric compatibility holds in the strong form.
Indeed, this turns out to be the case.

\begin{theorem*}[\cref{thm:strong-levi-civita}]
The connection $\nabla: \calc \to \calc \otimes_\subalg \calc$ is a bimodule connection and is compatible with the quantum metric, in the sense that $\conn \met = 0$.
\end{theorem*}

In this case we say that $\nabla$ is a quantum Levi-Civita connection, in perfect agreement with the classical description.
Hence we find that, in the case of quantum projective spaces, the classical theory can be lifted to the quantum realm in a fairly satisfactory way.

Our results generalize those of \cite{majid-sphere} for the quantum two-sphere (the simplest case of a quantum projective space), with the notable difference that the conditions of being a bimodule connection and metric compatibility in the strong form were not investigated.

A quantum metric and a connection are the main ingredients needed to study further aspects of quantum Riemannian geometry, as discussed in \cite{quantum-book}.
This program is carried out further in \cite{majid-sphere}, where it is shown for instance that the quantum two-sphere satisfies an analogue of the Einstein condition: this means that the quantum metric is proportional to an appropriately defined Ricci tensor, defined using the curvature of the connection.
We conjecture that this will hold for all quantum projective spaces, and more generally for all quantum irreducible flag manifolds.
We plan to tackle this problem in future research.

Let us also discuss how our results compare to the existing literature on connections for quantum projective spaces.
A lot of attention has been reserved to the case of line bundles, for instance we mention \cite{linebundles-cp2, linebundles-cpn} for their focus on complex geometry.
More relevant for us is the paper \cite{projective-bundle}, where the theory of quantum principal bundles is used to introduce a connection on the cotangent bundle, using a non-canonical calculus on the total space (the quantum special unitary group, in this case).
We point out that no further properties of these connections are explored, and extensive use is made of the explicit algebraic relations, making it hard to generalize to arbitrary quantum flag manifolds.

We should mention that the connections introduced in \cite{projective-bundle} turn out to coincide with those we describe here.
This follows from the recent results of \cite{connections-irreducible}, where representation-theoretic methods are used to prove the following result: there exists a unique covariant connection on the Heckenberger-Kolb calculus $\calc$ over a quantum irreducible flag manifold.
Moreover they show that this connection is torsion free.
It should be possible to extend these techniques to study some further aspects, a plan which is currently under investigation.

However, one notable drawback of the representation-theoretic approach is that it does not give explicit formulae for the connections.
On the other hand, in this paper we provide explicit formulae, which for instance allow us to straightforwardly check the classical limit.
Another bonus is that our approach is essentially self-contained, since we only use the relations in the Heckenberger-Kolb calculus $\calc$, plus general identities of categorical nature.

Finally let us say something about uniqueness of the structures presented in this paper, since in the classical case the Levi-Civita connection is the unique connection which is torsion free and compatible with the metric.
This is also the case here, as long as we insist that everything should be \emph{covariant}, that is compatible with the quantum group (co)action.
In this case uniqueness follows from representation theory, essentially as in the classical case: as already mentioned, the results from \cite{connections-irreducible} show that there is a unique covariant connection on $\calc$; by similar arguments, one shows that coinvariant quantum metrics are unique up to a scalar.
However, what is not clear from this point of view is why the connection and the quantum metric should be compatible, which is one of the goals we achieve in this paper.

Let us now discuss the organization of this paper.
The first four sections contain various background material, presented in a form suitable for our needs.
In \cref{sec:quantum-groups} we recall some basic facts about compact quantum groups, while in \cref{sec:categorical} we recall various identities holding in the setting of rigid braided monoidal categories, which we use throughout the text.
In \cref{sec:calculi-etc} we give the precise definitions involving differential calculi, quantum metrics and connections.
In \cref{sec:quantum-flag} we describe the quantum irreducible flag manifolds following \cite{heko}, with some small changes.
\cref{sec:hk-calculi} is also largely explanatory, as we recall the description of the Heckenberger-Kolb calculi for quantum irreducible flag manifolds, but we also prove various alternative expressions for some of the relations of the calculi.

The next three sections contain the proofs of our main results.
In \cref{sec:quantum-metrics} we introduce the quantum metrics, discuss some of their properties and finally prove the existence of appropriate inverse metrics.
In \cref{sec:connections} we introduce two connections on the holomorphic and antiholomorphic part of the calculi.
Their direct sum gives a connection which we show to be torsion free and cotorsion free.
In \cref{sec:bimodule-connections} we show that this is a bimodule connection and verify the property of metric compatibility in the strong form.

Many technical computations are relegated to the appendices, to make the main text more readable.
In \cref{sec:projective-spaces} we recall various results about projective spaces, to facilitate the comparison with the quantum case.
In \cref{sec:properties-S} we prove various properties satisfied by the maps $\opS$ and $\opSt$, which we use to rewrite some of the relations of the Heckenberger-Kolb calculus.
In \cref{sec:identities} we prove many of the technical identities that are used in the main text.
Finally in \cref{sec:bimodule-maps} we introduce various bimodule maps, some used to define the inverse metrics and some to check the bimodule property of the connections.

\medskip
\textbf{Acknowledgements.} I would like to thank Réamonn Ó Buachalla for various discussions and his comments on a preliminary version of this paper.

\section{Quantum groups}
\label{sec:quantum-groups}

In this section we review some background material on compact quantum groups.

\subsection{Quantized enveloping algebras}

We use the conventions of the book \cite{klsc}, since they are used in our main reference \cite{heko}.
Let $\lieg$ be a complex simple Lie algebra.
Given a real number $q$ such that $0 < q < 1$, the \emph{quantized enveloping algebra} $\Uqg$ is a certain Hopf algebra deformation of the enveloping algebra $U(\lieg)$, defined as follows.
It has generators $\{ K_i, \ E_i, \ F_i \}_{i = 1}^r$ with $r := \mathrm{rank}(\lieg)$ and relations as in \cite[Section 6.1.2]{klsc}.
In particular, the comultiplication, antipode and counit are given by
\[
\begin{gathered}
\Delta(K_i) = K_i \otimes K_i, \quad
\Delta(E_i) = E_i \otimes K_i + 1 \otimes E_i, \quad
\Delta(F_i) = F_i \otimes 1 + K_i^{-1} \otimes F_i, \\
S(K_i) = K_i^{-1}, \quad
S(E_i) = - E_i K_i^{-1}, \quad
S(F_i) = - K_i F_i, \\
\varepsilon(K_i) = 1, \quad
\varepsilon(E_i) = 0, \quad
\varepsilon(F_i) = 0.
\end{gathered}
\]
Given $\lambda = \sum_{i = 1}^r n_i \alpha_i$ we write $K_\lambda := K_1^{n_1} \cdots K_r^{n_r}$.
Let $\rho := \frac{1}{2} \sum_{\alpha > 0} \alpha$ be the half-sum of the positive roots of $\lieg$.
Then we have $S^2(X) = K_{2 \rho} X K_{2 \rho}^{-1}$ for any $X \in \Uqg$.

We also consider a $*$-structure on $\Uqg$, which in the classical case corresponds to the compact real form $\lieu$ of $\lieg$.
We can take for instance
\[
K_i^* = K_i, \quad
E_i^* = K_i F_i, \quad
F_i^* = E_i K_i^{-1}.
\]
The precise formulae are not very important here, as any equivalent $*$-structure works equally well for our purposes.
We write $\Uqu := (\Uqg, *)$ when we consider $\Uqg$ endowed with the $*$-structure corresponding to the compact real form.

\subsection{Quantized coordinate rings}

The \emph{quantized coordinate ring} $\CqG$ is defined as a subspace of the linear dual $\Uqg^*$.
We take the span of all the matrix coefficients of the finite-dimensional irreducible representations $V(\lambda)$ (see below).
It becomes a Hopf algebra by duality in the following manner: given $X, Y \in \Uqg$ and $a, b \in \CqG$ we define
\[
\begin{gathered}
(a b)(X) := (a \otimes b) \Delta(X), \quad
\Delta(a) (X \otimes Y) := a(X Y), \\
S(a)(X) := a(S(X)), \quad
1(X) := \varepsilon(X), \quad
\varepsilon(a) := a(1).
\end{gathered}
\]
Moreover it becomes a Hopf $*$-algebra by setting
\[
a^*(X) := \overline{a(S(X)^*)}.
\]
We write $\CqU := (\CqG, *)$ for $\CqG$ endowed with this $*$-structure.

We have a left action $\triangleright$ and a right action $\triangleleft$ of $\Uqg$ on $\CqG$ given by
\[
(X \triangleright a)(Y) := a(Y X), \quad
(a \triangleleft X)(Y) := a(X Y).
\]
Using the action of $\Uqg$ on $\CqG$ we can define quantum homogeneous spaces.

\subsection{Matrix coefficients}
\label{sec:mat-coeff}

The representation theory of $\Uqg$ is essentially the same as that of $U(\lieg)$, hence of $\lieg$.
In particular we have analogues of the highest weight modules $V(\lambda)$ for any dominant weight $\lambda$, which we denote by the same symbol.
Given a finite-dimensional representation $V$, we define its \emph{matrix coefficients} by
\[
(c^V_{f, v})(X) := f(X v), \quad f \in V^*, \ v \in V, \ X \in \Uqg.
\]
These elements span $\CqG$, according to the description given above.

We say that an inner product $(\cdot, \cdot)$ on $V$ is \emph{$\Uqu$-invariant} if it satisfies
\[
(X v, w) = (v, X^* w), \quad \forall v, w \in V, \ \forall X \in \Uqu.
\]
Here we use the $*$-structure of $\Uqu$.
It is well-known that an $\Uqu$-invariant inner product exists on every representation $V(\lambda)$, and it is unique up to a constant.
We typically write $\{v_i\}_i$ for an orthonormal weight basis of $V(\lambda)$ with respect to $(\cdot, \cdot)$, and write $\lambda_i$ for the \emph{weight} of $v_i$.
We also denote by $\{f^i\}_i$ the corresponding dual basis of $V(\lambda)^*$.

\section{Categorical preliminaries}
\label{sec:categorical}

The category of finite-dimensional $\Uqg$-modules is braided monoidal, that is we have a tensor product and an analogue of the flip map.
We use some of the language of tensor categories to make our computations more natural, with \cite{egno} as our main reference.

\subsection{Braiding}

A \emph{braiding} on a monoidal category is the choice of a natural isomorphism $X \otimes Y \cong Y \otimes X$ for each pair of objects $X$ and $Y$, satisfying the hexagon relations \cite[Definition 8.1.1]{egno}.
It is a generalization of the flip map in the category of vector spaces.

For the category of finite-dimensional $\Uqg$-modules we write the braiding as
\[
\braid_{V, W}: V \otimes W \to W \otimes V.
\]
An important relation satisfied by the braiding is the \emph{braid equation}, which is
\[
(\braid_{W, Z} \otimes \id_V) (\id_W \otimes \braid_{V, Z}) (\braid_{V, W} \otimes \id_Z)
= (\id_Z \otimes \braid_{V, W}) (\braid_{V, Z} \otimes \id_W) (\id_V \otimes \braid_{W, Z}),
\]
acting on $V \otimes W \otimes Z$ for any modules $V, W, Z$.
In the following we employ a leg-notation for the action on tensor products, in terms of which the braid equation reads
\begin{equation}
\label{eq:braid-equation}
(\braid_{W, Z})_{1 2} (\braid_{V, Z})_{2 3} (\braid_{V, W})_{1 2}
= (\braid_{V, W})_{2 3} (\braid_{V, Z})_{1 2} (\braid_{W, Z})_{2 3}.
\end{equation}

A braiding on the category of finite-dimensional $\Uqg$-modules is not quite unique.
We adopt the same choice as \cite{heko}, which is described as follows.
Consider two simple modules $V(\lambda)$ and $V(\mu)$ and choose a highest weight vector $v_\lambda$ for the first and a lowest weight vector $v_{w_0 \mu}$ for the second.
Then the braiding is completely determined by
\[
\braid_{V(\lambda), V(\mu)} (v_\lambda \otimes v_{w_0 \mu}) = q^{(\lambda, w_0 \mu)} v_{w_0 \mu} \otimes v_\lambda.
\]
Here $(\cdot, \cdot)$ denotes the usual non-degenerate symmetric bilinear form on the dual of the Cartan subalgebra of $\lieg$ (rescaled so that $(\alpha, \alpha) = 2$ for short roots $\alpha$, for definiteness).
Indeed, $v_\lambda \otimes v_{w_0 \mu}$ is a cyclic vector for $V(\lambda) \otimes V(\mu)$, hence $\braid_{V(\lambda), V(\mu)}$ is completely determined by the action on this vector and the fact that it is a $\Uqg$-module map.

\subsection{Duality}

The notion of duality in a monoidal category is captured by the existence of \emph{evaluation} and \emph{coevaluation} morphisms.
In our setting these are maps
\[
\begin{gathered}
\ev_V: V^* \otimes V \to \bbC, \quad
\coev_V: \bbC \to V \otimes V^*, \\
\ev^\prime_V: V \otimes V^* \to \bbC, \quad
\coev^\prime_V: \bbC \to V^* \otimes V,
\end{gathered}
\]
satisfying certain duality relations to be recalled below.
Here $V$ is a finite-dimensional $\Uqg$-module and $V^*$ is its linear dual.
The maps $\ev_V$ and $\coev_V$ are related to the existence of a \emph{left dual}, while $\ev_V^\prime$ and $\coev_V^\prime$ to the existence of a \emph{right dual}.
In the case of $\Uqg$, the property $S^2(X) = K_{2 \rho} X K_{2 \rho}^{-1}$ guarantees that the two duals can be identified.

Let us now discuss the explicit formulae for the category of finite-dimensional $\Uqg$-modules.
Take a weight basis $\{v_i\}_i$ of $V$, with $\lambda_i$ the weight of $v_i$, and a dual basis $\{f^i\}_i$ of $V^*$.
Then the evaluation and coevaluation maps are given by
\[
\begin{gathered}
\ev_V(f^i \otimes v_j) = \delta^i_j, \quad
\coev_V = \sum_i v_i \otimes f^i, \\
\ev^\prime_V(v_i \otimes f^j) = q^{(2 \rho, \lambda_i)} \delta^j_i, \quad
\coev^\prime_V = \sum_i q^{-(2 \rho, \lambda_i)} f^i \otimes v_i.
\end{gathered}
\]
The factor $q^{(2 \rho, \lambda_i)}$ comes from the action of $K_{2 \rho}$ and is related to the square of the antipode.

In the following we are going to fix a simple module $V$ and write
\[
\EV := \ev_V, \quad
\EVp := \ev^\prime_V, \quad
\CV := \coev_V, \quad
\CVp := \coev^\prime_V.
\]
We use the leg-notation for the action of these morphisms on tensor products. We write
\[
\begin{split}
& \EV_{i, i + 1} (w_1 \otimes \cdots \otimes w_{i - 1} \otimes f \otimes v \otimes w_{i + 2} \otimes \cdots \otimes w_n) \\
& := \EV(f \otimes v) w_1 \otimes \cdots \otimes w_{i - 1} \otimes w_{i + 2} \otimes \cdots w_n,
\end{split}
\]
with $v \in V$ and $f \in V^*$, while for the coevaluation we write
\[
\CV_i (w_1 \otimes \cdots \otimes w_n) := w_1 \otimes \cdots \otimes w_{i - 1} \otimes \sum_j v_j \otimes f^j \otimes w_i \otimes \cdots \otimes w_n.
\]
Similarly in the case of $\EVp$ and $\CVp$.
Using the leg-notation, the duality relations of \cite[Section 2.10]{egno} for the evaluation and coevaluation morphisms can be written as
\begin{equation}
\label{eq:duality}
\begin{gathered}
\EV_{2 3} \CV_1 = \id, \quad
\EV_{1 2} \CV_2 = \id, \\
\EVp_{2 3} \CVp_1 = \id, \quad
\EVp_{1 2} \CVp_2 = \id.
\end{gathered}
\end{equation}

We also have various compatibility relations with the braiding $\braid_{V, W}$, since the latter is a natural isomorphism in both entries.
For the evaluation morphisms we have
\begin{equation}
\label{eq:evaluations}
\begin{gathered}
\EV_{1 2} = \EV_{2 3} (\braid_{V^*, W})_{1 2} (\braid_{V, W})_{2 3}, \quad
\EV_{2 3} = \EV_{1 2} (\braid_{W, V})_{2 3} (\braid_{W, V^*})_{1 2}, \\
\EVp_{1 2} = \EVp_{2 3} (\braid_{V, W})_{1 2} (\braid_{V^*, W})_{2 3}, \quad
\EVp_{2 3} = \EVp_{1 2} (\braid_{W, V^*})_{2 3} (\braid_{W, V})_{1 2}.
\end{gathered}
\end{equation}
Similarly, for the coevaluations morphisms we have
\begin{equation}
\label{eq:coevaluations}
\begin{gathered}
\CV_1 = (\braid_{W, V^*})_{2 3} (\braid_{W, V})_{1 2} \CV_2, \quad
\CV_2 = (\braid_{V, W})_{1 2} (\braid_{V^*, W})_{2 3} \CV_1, \\
\CVp_1 = (\braid_{W, V})_{2 3} (\braid_{W, V^*})_{1 2} \CVp_2, \quad
\CVp_2 = (\braid_{V^*, W})_{1 2} (\braid_{V, W})_{2 3} \CVp_1.
\end{gathered}
\end{equation}

Finally we need the following identity, valid for a simple module $V$.

\begin{lemma}
\label{lem:relation-evaluations}
Let $V = V(\lambda)$ be a simple module. Then we have
\begin{equation}
\label{eq:identity-EEp}
\EV \circ \braid_{V, V^*} = q^{-(\lambda, \lambda + 2 \rho)} \EVp.
\end{equation}
\end{lemma}

\begin{proof}
Observe that both $\EV \circ \braid_{V, V^*}$ and $\EVp$ are morphisms from $V \otimes V^*$ to $\bbC$.
Since $V$ is a simple module, we must have $\EV \circ \braid_{V, V^*} = c \EVp$ for some $c \in \bbC$.
To find the constant we evaluate both sides at $v_1 \otimes f^1$, where $v_1$ is a highest weight vector of $V = V(\lambda)$ and $f^1$ is its dual.
In our conventions for the braiding we have $\braid_{V, V^*}(v_1 \otimes f^1) = q^{-(\lambda, \lambda)} f^1 \otimes v_1$. Then
\[
\EV \circ \braid_{V, V^*} (v_1 \otimes f^1) = q^{-(\lambda, \lambda)} \EV (f^1 \otimes v_1) = q^{-(\lambda, \lambda)}.
\]
On the other hand we have $\EVp (v_1 \otimes f^1) = q^{(2 \rho, \lambda_1)} = q^{(2 \rho, \lambda)}$.
Hence $c = q^{-(\lambda, \lambda + 2 \rho)}$.
\end{proof}

\section{Differential calculi, metric and connections}
\label{sec:calculi-etc}

In this section we collect various definitions about differential calculi and connections.

\subsection{Differential calculi}

In this section $A$ denotes an arbitrary algebra.
The definitions recalled here are fairly standard and one possible reference is \cite{klsc}.

\begin{definition}
A \emph{differential calculus} over $A$ is a differential graded algebra $(\Omega^\bullet, \diff)$ such that $\Omega^0 = A$ and which is generated by the elements $a, \diff b$ with $a, b \in A$.

If $A$ is a $*$-algebra, we say that $(\Omega^\bullet, \diff)$ is a \emph{$*$-differential calculus} if in addition the $*$-structure of $A$ extends to an involutive conjugate-linear map on $\Omega^\bullet$, such that $\diff a^* = (\diff a)^*$ and $(\omega \wedge \chi)^* = (-1)^{p q} \chi^* \wedge \omega^*$ for all $\omega \in \Omega^p$ and $\chi \in \Omega^q$.
\end{definition}

The concrete definition of a differential calculus usually begins with the description of its degree-one part. This leads to the following definition.

\begin{definition}
A \emph{first order differential calculus} (FODC) over $A$ is an $A$-bimodule $\Omega$ with a linear map $\diff: A \to \Omega$ which obeys the Leibniz rule
\[
\diff(a b) = \diff a b + a \diff b, \quad a, b \in A,
\]
and such that $\Omega$ is generated as a left $A$-module by the elements $\diff a$ with $a \in A$.
\end{definition}

Given any FODC $(\Omega, \diff)$, there exists a universal differential calculus such that its degree-one part is $\Omega$.
The universal property in this case is the following.

\begin{definition}
The \emph{universal differential calculus} associated to a FODC $(\Omega, \diff)$ over $A$ is the unique differential calculus $(\Omega^\bullet_u, \diff_u)$ over $A$ with $\Omega^1_u = \Omega$, $\diff_u |_A = \diff$ and such that the following property is satisfied: for any differential calculus $(\Gamma^\bullet, \diff^\prime)$ with $\Gamma^1 = \Omega$ and $\diff^\prime |_A = \diff$, there exists a map of differential graded algebras $\phi: \Omega^\bullet_u \to \Gamma^\bullet$ such that $\phi |_{A \oplus \Omega} = \id$.
\end{definition}

The universal differential calculus can be constructed as a quotient of the tensor algebra of the $A$-bimodule $\Omega^1$, with differential $\diff_u(a_0 \diff a_1 \wedge \cdots \wedge \diff a_n) = \diff a_0 \diff a_1 \wedge \cdots \wedge \diff a_n$.
Any differential calculus can be obtained as a quotient of the universal differential calculus.

Finally we recall the notion of induced calculus over a subalgebra.

\begin{definition}
Let $B \subset A$ be a subalgebra and $(\Omega, \diff)$ a FODC over $A$.
Then the \emph{induced} FODC over $B$ is defined by $\Omega|_B = \mathrm{span} \{ b_1 \diff b_2: b_1, b_2 \in B \}$ and with differential $\diff|_B$.
\end{definition}

\subsection{Metrics}

We now recall the notion of quantum metric as stated in \cite[Definition 1.15]{quantum-book}.
Notice that invertibility is part of the definition.

\begin{definition}
\label{def:quantum-metric}
A \emph{(generalized) quantum metric} is an element $g \in \Omega^1 \otimes_A \Omega^1$ which is invertible, in the sense that there exists a bimodule map $(\cdot, \cdot): \Omega^1 \otimes_A \Omega^1 \to A$ such that
\[
(\omega, g^{(1)}) g^{(2)} = \omega = g^{(1)} (g^{(2)}, \omega)
\]
for all $\omega \in \Omega^1$, where we write $g = g^{(1)} \otimes g^{(2)}$.
\end{definition}

\begin{remark}
From the categorical point of view, a quantum metric makes $\Omega^1$ into a self-dual object in the monoidal category of $A$-bimodules.
\end{remark}

Notice that the definition of a quantum metric only uses $\Omega^1$, the degree-one part of $\Omega^\bullet$.
To impose an analogue of the symmetry condition we also use $\Omega^2$.

\begin{definition}
\label{def:metric-symmetric}
A quantum metric $g \in \Omega^1 \otimes_A \Omega^1$ is \emph{symmetric} if we have $\wedge(g) = 0$, where $\wedge: \Omega^1 \otimes \Omega^1 \to \Omega^2$ denotes the wedge product of one-forms.
\end{definition}

Finally, in the case when $A$ is a $*$-algebra and $\Omega^\bullet$ is a $*$-differential calculus, we can require the metric to be real in the following sense.

\begin{definition}
\label{def:metric-real}
A quantum metric $g \in \Omega^1 \otimes_A \Omega^1$ is \emph{real} if we have $g^\dagger = g$, where $\dagger := \mathrm{flip} \circ (* \otimes *)$ is given by the $*$-structure composed with the flip map.
\end{definition}

\subsection{Connections}

The notion of connection on a module is quite standard.
We are only going to consider left connections, so we omit "left" after the definition.

\begin{definition}
A (left) \emph{connection} on a (left) $A$-module $E$ is a linear map $\nabla_E: E \to \Omega^1 \otimes_A E$ which obeys the (left) Leibniz rule, that is
\[
\nabla_E(a e) = \diff a \otimes e + a \nabla_E(e), \quad
a \in A, \ e \in E.
\]
\end{definition}

For the left $A$-module $E = \Omega^1$ we can define additional properties.

\begin{definition}
\label{def:torsion}
The \emph{torsion} of a connection $\nabla: \Omega^1 \to \Omega^1 \otimes_A \Omega^1$ is the left $A$-module map $T_\nabla: \Omega^1 \to \Omega^2$ defined by
\[
T_\nabla := \wedge \circ \nabla - \diff.
\]
A connection is called \emph{torsion free} if $T_\nabla = 0$.
\end{definition}

Now suppose that $\Omega^1$ admits a quantum metric $g \in \Omega^1 \otimes_A \Omega^1$ as in \cref{def:quantum-metric}.
Then we can consider the cotorsion of the connection $\nabla$ with respect to $g$, a notion introduced by Majid in \cite{majid-braided} as a weaker version of metric compatibility.

\begin{definition}
\label{def:cotorsion}
The \emph{cotorsion} of a connection $\nabla: \Omega^1 \to \Omega^1 \otimes_A \Omega^1$ with quantum metric $g \in \Omega^1 \otimes_A \Omega^1$ is the element $\mathrm{co} T_\nabla \in \Omega^2 \otimes_A \Omega^1$ defined by
\[
\mathrm{co} T_\nabla := (\diff \otimes \id - (\wedge \otimes \id) \circ (\id \otimes \nabla)) g.
\]
A connection is called \emph{cotorsion free} if $\mathrm{co} T_\nabla = 0$.
\end{definition}

\begin{remark}
Let $M$ be a smooth manifold with metric $g$.
Consider a connection $\nabla$ on $M$, defined in the usual sense as acting on vector fields.
Define a connection $\nabla^*$ by
\[
X(g(Y, Z)) = g(\nabla^*_X Y, Z) + g(Y, \nabla_X Z).
\]
Here $X, Y, Z$ are vector fields on $M$.
As discussed in \cite[Corollary 5.70]{quantum-book}, the cotorsion of the connection $\nabla$ can be identified with the torsion of the connection $\nabla^*$ defined above.
In particular, if $\nabla$ is torsion free then the cotorsion free condition gives
\[
(\nabla_X g)(Y, Z) = (\nabla_Y g)(X, Z)
\]
for all vector fields $X, Y, Z$.
This is a weaker condition than $(\nabla_X g)(Y, Z) = 0$, which is the standard metric compatibility condition with respect to the metric $g$.
\end{remark}

In the classical case, the Levi-Civita connection is the unique connection on the tangent bundle of a smooth manifold which is torsion free and compatible with the metric.
This motivates the following definition, see \cite[Definition 8.2]{quantum-book}.

\begin{definition}
\label{def:levi-civita}
A \emph{weak quantum Levi-Civita connection} is a connection $\nabla: \Omega^1 \to \Omega^1 \otimes_A \Omega^1$ which is torsion free and cotorsion free.
\end{definition}

To introduce a strong version we need bimodule connections, which we now recall.

\subsection{Bimodule connections}

Classically a connection on $\Omega^1$ naturally extends to a connection on the tensor product $\Omega^1 \otimes_A \Omega^1$.
In the quantum case this lifting requires the connection to be a bimodule connection, defined as in \cite[Definition 3.66]{quantum-book}.

\begin{definition}
\label{def:bimodule-connection}
A (left) \emph{bimodule connection} on an $A$-bimodule $E$ is a (left) connection $\nabla_E: E \to \Omega^1 \otimes_A E$ together with an $A$-bimodule map $\sigma_E: E \otimes_A \Omega^1 \to \Omega^1 \otimes_A E$ such that
\[
\nabla_E(e a) = \sigma_E(e \otimes \diff a) + \nabla_E(e) a, \quad e \in E, \ a \in A.
\]
\end{definition}

The bimodule map $\sigma_E$, called the \emph{generalized braiding}, is not additional data for the connection.
Indeed, if it exists it is uniquely determined by the condition above.

A bimodule connection $\nabla: \Omega^1 \to \Omega^1 \otimes_A \Omega^1$ can be extended to a connection on $\Omega^1 \otimes_A \Omega^1$ by the Leibniz rule and the generalized braiding, see \cite[Theorem 3.78]{quantum-book}.
In particular, given a quantum metric $g \in \Omega^1 \otimes_A \Omega^1$ we can consider
\[
\nabla g = (\nabla \otimes \id) g + (\sigma \otimes \id) (\id \otimes \nabla) g.
\]
This naturally leads to the following definition.

\begin{definition}
Let $\nabla$ be a bimodule connection on $\Omega^1$.
We say that it is \emph{quantum metric compatible} with a quantum metric $g \in \Omega^1 \otimes_A \Omega^1$ if $\nabla g = 0$.
\end{definition}

\begin{remark}
As discussed before, in the classical case cotorsion freeness is a weaker property than metric compatibility.
In the quantum case we need an extra assumption to compare the two notions, namely the condition $\wedge \circ (\sigma + \id) = 0$ for the generalized braiding.
Under this condition and torsion freeness of $\nabla$ we obtain $\mathrm{co} T_\nabla = (\wedge \otimes \id) \nabla g$, see \cite[Section 8.1]{quantum-book}.
This shows that cotorsion freeness is a weaker property, in this case.
\end{remark}

Finally we can formulate the notion of Levi-Civita connection, as in the classical case.

\begin{definition}
Let $\nabla: \Omega^1 \to \Omega^1 \otimes_A \Omega^1$ be a bimodule connection and let $g \in \Omega^1 \otimes_A \Omega^1$ be a quantum metric.
Then we say that $\nabla$ is a \emph{quantum Levi-Civita connection} if it is torsion free and quantum metric compatible with $g$.
\end{definition}

\section{Quantum flag manifolds}
\label{sec:quantum-flag}

The quantum projective spaces can be regarded as the easiest family to describe within the class of quantum (irreducible) flag manifolds.
All these quantum spaces admit a uniform description, which we recall here.
Even though the focus of this paper is on the projective spaces, many of our computations also work for general irreducible flag manifolds.
We take some care in explaining the index-free notation we are going to employ in the following, as it simplifies the computations tremendously (once one gets the hang of it).

\subsection{Geometrical description}

We start by quickly recalling the definition of a flag manifold in the classical case, for precise definitions see for instance \cite{cap-book}.
Let $G$ be a complex simple Lie group, with compact real form $U$.
Corresponding to any subset of simple roots, denoted by $S$, we can define a \emph{parabolic subgroup} $P_S \subset G$ and a \emph{Levi subgroup} $L_S \subset P_S$.
A \emph{(generalized) flag manifold} is a homogeneous space of the form $G / P_S$.
In terms of the compact real form, we have the subgroup $K_S := P_S \cap U = L_S \cap U$ and the isomorphism $G / P_S \cong U / K_S$.

In the quantum case, we begin by introducing an analogue of the Levi factor $\liel_S$ (the Lie algebra of $L_S$), following \cite{quantum-flag}.
The \emph{quantized Levi factor} $\UqlS$ is defined by
\[
\UqlS := \langle K_i, \ E_j, \ F_j : i \in I, \ j \in S \rangle \subset \Uqg.
\]
Here $\langle \cdot \rangle$ denotes the subalgebra generated by the given elements in $\Uqg$.
It is easily verified that $\UqlS$ is a Hopf subalgebra.
Moreover it is a Hopf $*$-subalgebra with $*$ corresponding to the compact real form.
Taking the $*$-structure into account we write $\UqkS := (\UqlS, *)$.
The \emph{quantum flag manifold} $\Cqflag$ is then defined as
\[
\Cqflag := \CqU^{\UqkS} = \{ a \in \CqU: X \triangleright a = \varepsilon(X) a, \ \forall X \in \UqkS \}.
\]

In the following we restrict to the case of \emph{irreducible} flag manifolds.
At the Lie algebra level these can be characterized as follows: the set $S$ consists of all the simple roots except for $\alpha_s$, where $\alpha_s$ is a simple root appearing with multiplicity one in the highest root of $\lieg$.

\subsection{Generators and relations}

The quantum flag manifolds $\Cqflag$ admit a uniform description in terms of generators and relations.
We follow the presentation in \cite{heko}.

Consider the simple $\Uqg$-module $V := V(\omega_s)$, where $\omega_s$ is the fundamental weight corresponding to the simple root $\alpha_s$ described above, and write $N := \dim V$.

We define the algebra $\alg$ with generators $\{ v^i, f^i \}_{i = 1}^N$ and relations
\begin{equation}
\label{eq:relationsA-indices}
\begin{gathered}
f^i f^j = q^{-(\omega_s, \omega_s)} \sum_{k, l} (\braid_{V, V})^{i j}_{k l} f^k f^l, \quad
v^i v^j = q^{-(\omega_s, \omega_s)} \sum_{k, l} (\braid_{V^*, V^*})^{i j}_{k l} v^k v^l, \\
v^i f^j = q^{(\omega_s, \omega_s)} \sum_{k, l} (\braid_{V, V^*})^{i j}_{k l} f^k v^l, \quad
\sum_i v^i f^i = 1.
\end{gathered}
\end{equation}
These generators should be interpreted as follows: after fixing a weight basis $\{v_i\}_{i = 1}^N$ of $V = V(\omega_s)$, we have the \emph{dual basis} $\{f^i\}_{i = 1}^N$ of $V^* \cong V(- w_0 \omega_s)$ and the \emph{double dual basis} $\{v^i\}_{i = 1}^N$ of $V^{**} \cong V(\omega_s)$ (defined by $v^i(f^j) = \delta_{i j}$).
One can check that this identification makes $\alg$ into a $\Uqg$-module algebra.
Then we have the following result.

\begin{lemma}
\label{lem:flag-isomorphism}
The $\Uqg$-module algebra $\alg$ is isomorphic to the $\Uqg$-module subalgebra of $\CqG$ generated by the matrix coefficients $c^{\omega_s}_{f^i, v_1}$ and $c^{-w_0 \omega_s}_{v^i, f^1}$, where $v_1$ is a fixed highest weight vector of $V(\omega_s)$.
The isomorphism is given by
\[
f^i \mapsto c^{\omega_s}_{f^i, v_1}, \quad
v^i \mapsto c^{-w_0 \omega_s}_{v^i, f^1}.
\]
\end{lemma}

The algebra $\alg$ is $\bbZ$-graded by $\deg f^i := 1$ and $\deg v^i := -1$.
We write $\subalg := \alg^0$ for its degree-zero subalgebra, which is generated by the elements $p^{i j} := f^i v^j$.

\begin{proposition}
The algebra $\subalg$ is isomorphic to the quantum irreducible flag manifold $\Cqflag$ as a $\Uqg$-module under the isomorphism above.
\end{proposition}

The relations for the generators $p^{i j}$ of $\subalg$ are given in \cite[Section 3.1.3]{heko}. Write
\[
\braidP_{V, V} := \braid_{V, V} - q^{(\omega_s, \omega_s)}, \quad
\braidP_{V^*, V^*} := \braid_{V^*, V^*} - q^{(\omega_s, \omega_s)}.
\]
Then the relations can be written as
\begin{equation}
\label{eq:relationsB-indices}
\begin{gathered}
\sum_{a, b, c, d} (\braidP_{V, V})^{i j}_{a b} (\braid^{-1}_{V, V^*})^{b k}_{c d} p^{a c} p^{d l} = 0, \quad
\sum_{a, b, c, d} (\braidP_{V^*, V^*})^{k l}_{a b} (\braid^{-1}_{V, V^*})^{j a}_{c d} p^{i c} p^{d b} = 0, \\
\sum_i q^{(2 \rho, \lambda_i)} p^{i i} = q^{(\omega_s, 2 \rho)}.
\end{gathered}
\end{equation}

\begin{remark}
The last relation appears as $q^{(\omega_s, \omega_s)} \sum_{i, j, k} (\braid_{V, V^*})^{k k}_{i j} p^{i j} = 1$ in \cite{heko}.
We rewrite it using the identity $\EV \circ \braid_{V, V^*} = q^{-(\omega_s, \omega_s + 2 \rho)} \EVp$ from \cref{lem:relation-evaluations}, which leads to $\sum_k (\braid_{V, V^*})^{k k}_{i j} = q^{-(\omega_s, \omega_s + 2 \rho)} q^{(2 \rho, \lambda_i)} \delta_{i j}$.
Using this we obtain $\sum_i q^{(2 \rho, \lambda_i)} p^{i i} = q^{(\omega_s, 2 \rho)}$.
\end{remark}

As shown in \cite[Proposition 3.3]{kahler}, the algebras $\alg$ and $\subalg$ can be made into $*$-algebras as follows.
Choosing an orthonormal basis for $V = V(\omega_s)$ with respect to a $\Uqu$-invariant inner product, the $*$-structure is given by $(f^i)^* = v^i$.
In this case, the isomorphism from \cref{lem:flag-isomorphism} becomes a $*$-isomorphism.
For the generators $p^{i j}$ of $\subalg$ we have $(p^{i j})^* = p^{j i}$.

\subsection{Index-free notation}

In the following we adopt an index-free notation, as done in \cite{heko}, since it makes computations significantly clearer.
The basic idea is very simple: for instance, with $\{v_i\}_i$ the basis of $V$ (note the lower index) we write
\[
\sum_{a, b} (\braid_{V, V})_{j k}^{a b} v_i v_a v_b v_l
\ \longleftrightarrow \
(\braid_{V, V})_{2 3} v v v v.
\]
We want to use a similar notation for the generators $\{ f^i, v^i \}_{i = 1}^N$ of $\alg$.
What complicates matters here is that we want to consider $f^i$ as a linear functional on $V$ and $v^i$ as a linear functional on $V^*$, since this is how we have defined the $\Uqg$-module structure on $\alg$ (and the reason why we use upper indices for the generators).
The bottom line is that we need to consider the action of $\Uqg$-module maps on these elements via the transpose.

To give a concrete example, in this notation the first relation of \eqref{eq:relationsA-indices} becomes
\[
f^i f^j = q^{-(\omega_s, \omega_s)} \sum_{k, l} (\braid_{V, V})^{i j}_{k l} f^k f^l
\ \longleftrightarrow \
f f = q^{-(\omega_s, \omega_s)} (\braid_{V, V})_{1 2} f f.
\]
To give a different example, the last relation $\sum_i v^i f^i = 1$ of \eqref{eq:relationsA-indices} becomes $\EV_{1 2} v f = 1$, since the LHS is $\sum_{i, j} \EV_{i j} v^i f^j$ and we have $\EV_{i j} = \delta_{i j}$.
As a final example, the expression $\CV_1 f$ carries three indices and corresponds to $(\CV_1 f)^{i j k} = \delta_{i j} f^k$.

With this notation, the relations of $\alg$ can be rewritten in the condensed form
\begin{equation}
\label{eq:relationsA}
\begin{gathered}
f f = q^{-(\omega_s, \omega_s)} (\braid_{V, V})_{1 2} f f, \quad
v v = q^{-(\omega_s, \omega_s)} (\braid_{V^*, V^*})_{1 2} v v, \\
v f = q^{(\omega_s, \omega_s)} (\braid_{V, V^*})_{1 2} f v, \quad
\EV_{1 2} v f = 1.
\end{gathered}
\end{equation}
The situation is similar for the flag manifold $\subalg \subset \alg$.
In this case we have the generators $p^{i j} = f^i v^j$, which carry two indices, and the relations can be written as
\begin{equation}
\label{eq:relationsB}
(\braidP_{V, V})_{1 2} (\braid^{-1}_{V, V^*})_{2 3} p p = 0, \quad
(\braidP_{V^*, V^*})_{3 4} (\braid^{-1}_{V, V^*})_{2 3} p p = 0, \quad
\EVp_{1 2} p = q^{(\omega_s, 2 \rho)}.
\end{equation}

\section{Heckenberger-Kolb calculus}
\label{sec:hk-calculi}

In this section we describe the Heckenberger-Kolb calculus associated to an irreducible quantum flag manifold, as introduced in \cite{heko}.
We also give a slightly different presentation of some of the relations, which turns out to be more convenient for our purposes.
Finally we focus on a particular situation, which we refer to as the quadratic case, which geometrically corresponds to the quantum projective spaces.

\subsection{Definitions}

We start by describing the FODC $(\calc, \diff)$ associated to the Heckenberger-Kolb calculus.
We have $\calc := \calcP \oplus \calcM$ and $\diff := \del + \delbar$, where the two FODCs $\calcP$ and $\calcM$ are generated as left $\subalg$-modules by $\del p$ and $\delbar p$ respectively (we use the index-free notation from now on).
To describe the relations we need some additional notation.
Recall that $\braidP_{V, V} = \braid_{V, V} - q^{(\omega_s, \omega_s)}$ and $\braidP_{V^*, V^*} = \braid_{V^*, V^*} - q^{(\omega_s, \omega_s)}$. We also write
\[
\braidQ_{V, V} := \braid_{V, V} + q^{(\omega_s, \omega_s) - (\alpha_s, \alpha_s)}, \quad
\braidQ_{V^*, V^*} := \braid_{V^*, V^*} + q^{(\omega_s, \omega_s) - (\alpha_s, \alpha_s)}.
\]
Then $\calcP$ is generated by $\del p$, as a left $\subalg$-module, with relations
\begin{equation}
\label{eq:calcP-relations}
(\braidP_{V, V})_{1 2} (\braidQ_{V, V})_{1 2} (\braid^{-1}_{V, V^*})_{2 3} p \del p = 0, \quad
(\braidP_{V^*, V^*})_{3 4} (\braid^{-1}_{V, V^*})_{2 3} p \del p = 0, \quad
\EVp_{1 2} \del p = 0.
\end{equation}
Similarly, $\calcM$ is generated by $\delbar p$ with relations
\begin{equation}
\label{eq:calcM-relations}
(\braidP_{V^*, V^*})_{3 4} (Q_{V^*, V^*})_{3 4} (\braid^{-1}_{V, V^*})_{2 3} p \delbar p = 0, \quad
(\braidP_{V, V})_{1 2} (\braid^{-1}_{V, V^*})_{2 3} p \delbar p = 0, \quad
\EVp_{1 2} \delbar p = 0.
\end{equation}

To define the right $\subalg$-module structure, let us introduce the notation
\begin{equation}
\label{eq:T-map}
\opT_{1 2 3 4} := (\braid_{V, V^*})_{2 3} (\braid_{V, V})_{1 2} (\braid^{-1}_{V^*, V^*})_{3 4} (\braid^{-1}_{V, V^*})_{2 3}.
\end{equation}
Then the right $\subalg$-module structure of $\calcP$ and $\calcM$ is defined by
\begin{equation}
\label{eq:right-module}
\del p p = q^{(\alpha_s, \alpha_s)} \opT_{1 2 3 4} p \del p, \quad
\delbar p p = q^{-(\alpha_s, \alpha_s)} \opT_{1 2 3 4} p \delbar p.
\end{equation}

Finally, the Heckenberger-Kolb calculus $(\Omega^\bullet, \diff)$ is the universal differential calculus associated to $(\Omega, \diff)$.
It turns out that we have the decomposition $\diff = \del + \delbar$ also in higher degrees.
In particular, this implies that $\del^2 = \delbar^2 = 0$ and $\del \delbar = - \delbar \del$.
Moreover, as shown in \cite[Theorem 4.2]{kahler}, the calculus $(\Omega^\bullet, \diff)$ becomes a $*$-calculus upon setting $(\del p^{i j})^* = \delbar p^{j i}$.

\subsection{Induced calculi}
\label{sec:overcalc}

In \cite{heko} the FODCs $\calcP$ and $\calcM$ over $\subalg$ are constructed as induced calculi from some auxiliary FODCs $\overcalcP$ and $\overcalcM$ over the larger algebra $\alg$.
This description is also useful for us, so we recall the details below.

Consider the algebra $\alg$ introduced before, generated by $f$ and $v$.
We define the left $\alg$-modules $\overcalcP$ and $\overcalcM$, generated respectively by $\del f$ and $\delbar v$, with relations
\begin{equation}
\label{eq:overcalc-relations}
\begin{gathered}
(\braidP_{V, V})_{1 2} (\braidQ_{V, V})_{1 2} f \del f = 0, \quad
(\braidP_{V^*, V^*})_{1 2} (\braidQ_{V^*, V^*})_{1 2} v \delbar v = 0, \\
\EV_{1 2} v \del f = 0, \quad
\EVp_{1 2} f \delbar v = 0.
\end{gathered}
\end{equation}
The last two relations come from $\EV_{1 2} v f = 1$ and $\EVp_{1 2} f v = q^{(\omega_s, 2 \rho)}$.

The right $\alg$-module relations are as follows. For $\overcalcP$ we have
\begin{equation}
\label{eq:overcalcP-right}
\del f f = q^{(\alpha_s, \alpha_s) - (\omega_s, \omega_s)} (\braid_{V, V})_{1 2} f \del f, \quad
\del f v = q^{- (\omega_s, \omega_s)} (\braid^{-1}_{V, V^*})_{1 2} v \del f.
\end{equation}
The relations for $\overcalcM$ are
\begin{equation}
\label{eq:overcalcM-right}
\delbar v f = q^{(\omega_s, \omega_s)} (\braid_{V, V^*})_{1 2} f \delbar v, \quad
\delbar v v = q^{(\omega_s, \omega_s) - (\alpha_s, \alpha_s)} (\braid^{-1}_{V^*, V^*})_{1 2} v \delbar v.
\end{equation}
The differentials of the two FODCs are specified by requiring that $\del v = 0$ and $\delbar f = 0$.

The FODCs $\calcP$ and $\calcM$ are the induced calculi over $\subalg$ obtained from $\overcalcP$ and $\overcalcM$.
This description, together with the relation $\EV_{1 2} v f = 1$ and $p = f v$, leads to the relations
\begin{equation}
\label{eq:evaluations-calculi}
\begin{gathered}
\EV_{2 3} p \del p = 0, \quad
\EV_{2 3} \del p p = \del p, \\
\EV_{2 3} p \delbar p = \delbar p, \quad
\EV_{2 3} \delbar p p = 0.
\end{gathered}
\end{equation}
For more details see for instance \cite[Lemma 5.2]{kahler} (with different notation).

\subsection{Different presentation}

It is convenient to work with the relations of the FODC $\calc$ in a slightly different form, which we now derive.
We begin by defining the maps
\begin{equation}
\label{eq:S-maps}
\begin{split}
\opS_{1 2 3} & := (\braid_{V, V^*})_{2 3} (\braid_{V, V})_{1 2} (\braid^{-1}_{V, V^*})_{2 3}, \\
\opSt_{2 3 4} & := (\braid_{V, V^*})_{2 3} (\braid^{-1}_{V^*, V^*})_{3 4} (\braid^{-1}_{V, V^*})_{2 3}.
\end{split}
\end{equation}
Observe that $\opT_{1 2 3 4} = \opS_{1 2 3} \opSt_{2 3 4}$, where $\opT$ is the map defined in \eqref{eq:T-map} and related to the right $\subalg$-module structure.
The proof of the following result is given in \cref{prop:S-properties-proof}.

\begin{proposition}
\label{prop:S-properties}
The maps $\opS$ and $\opSt$ satisfy the following properties.

\begin{enumerate}
\item We have the commutation relations
\begin{equation}
\label{eq:S-commutation}
\opS_{1 2 3} \opSt_{2 3 4} = \opSt_{2 3 4} \opS_{1 2 3}, \quad
\opSt_{2 3 4} \opS_{3 4 5} = \opS_{3 4 5} \opSt_{2 3 4}.
\end{equation}
\item We have the "braid equations"
\begin{equation}
\label{eq:S-braid-equation}
\opS_{1 2 3} \opS_{3 4 5} \opS_{1 2 3} = \opS_{3 4 5} \opS_{1 2 3} \opS_{3 4 5}, \quad
\opSt_{2 3 4} \opSt_{4 5 6} \opSt_{2 3 4} = \opSt_{4 5 6} \opSt_{2 3 4} \opSt_{4 5 6}.
\end{equation}
\end{enumerate}
\end{proposition}

In particular, observe that we can also write $\opT_{1 2 3 4} = \opSt_{2 3 4} \opS_{1 2 3}$.

We now use the maps $\opS$ and $\opSt$ to rewrite some of the relations of the FODC $\Omega$.
We begin with the relations that involve $\braidP_{V, V}$ and $\braidP_{V^*, V^*}$.

\begin{lemma}
\label{lem:P-relations}
The relations
\[
(\braidP_{V^*, V^*})_{3 4} (\braid^{-1}_{V, V^*})_{2 3} p \del p = 0, \quad
(\braidP_{V, V})_{1 2} (\braid^{-1}_{V, V^*})_{2 3} p \delbar p = 0,
\]
of the FODC $\Omega$ are equivalent to
\begin{equation}
\label{eq:S-action-left}
(\opSt_{2 3 4} - q^{-(\omega_s, \omega_s)}) p \del p = 0, \quad
(\opS_{1 2 3} - q^{(\omega_s, \omega_s)}) p \delbar p = 0.
\end{equation}
Using the right $\subalg$-module structure, they are also equivalent to
\begin{equation}
\label{eq:S-action-right}
(\opSt_{2 3 4} - q^{-(\omega_s, \omega_s)}) \del p p = 0, \quad
(\opS_{1 2 3} - q^{(\omega_s, \omega_s)}) \delbar p p = 0.
\end{equation}
\end{lemma}

\begin{proof}
We consider the second relation $(\braidP_{V, V})_{1 2} (\braid^{-1}_{V, V^*})_{2 3} p \delbar p = 0$, the other one is treated in a similar way.
Applying $(\braid_{V, V^*})_{2 3}$ and using $\braidP_{V, V} = \braid_{V, V} - q^{(\omega_s, \omega_s)}$ we get
\[
(\braid_{V, V^*})_{2 3} (\braid_{V, V})_{1 2} (\braid^{-1}_{V, V^*})_{2 3} p \delbar p - q^{(\omega_s, \omega_s)} p \delbar p = 0.
\]
Since $\opS_{1 2 3} = (\braid_{V, V^*})_{2 3} (\braid_{V, V})_{1 2} (\braid^{-1}_{V, V^*})_{2 3}$ we obtain $(\opS_{1 2 3} - q^{(\omega_s, \omega_s)}) p \delbar p = 0$.

Next, using the right $\subalg$-module structure from \eqref{eq:right-module}, we have $p \delbar p = q^{(\alpha_s, \alpha_s)} \opT^{-1}_{1 2 3 4} \delbar p p$.
Then the identity $(\opS_{1 2 3} - q^{(\omega_s, \omega_s)}) p \delbar p = 0$ can be rewritten in the form
\[
(\opS_{1 2 3} - q^{(\omega_s, \omega_s)}) \opT^{-1}_{1 2 3 4} \delbar p p = 0.
\]
As $\opS_{1 2 3}$ commutes with $\opT^{-1}_{1 2 3 4}$, the identity $(\opS_{1 2 3} - q^{(\omega_s, \omega_s)}) \delbar p p = 0$ follows by applying $\opT_{1 2 3 4}$.
\end{proof}

Note that the previous identities also hold with $p p$ instead of $p \del p$ or $p \delbar p$, since we have the relations \eqref{eq:relationsB}.
Next, we rewrite the right $\subalg$-module relations.

\begin{lemma}
\label{lem:right-module}
The right $\subalg$-module relations \eqref{eq:right-module} of $\calc$ are equivalent to
\begin{equation}
\label{eq:S-right-module}
\del p p = q^{(\alpha_s, \alpha_s) - (\omega_s, \omega_s)} \opS_{1 2 3} p \del p, \quad
\delbar p p = q^{(\omega_s, \omega_s) - (\alpha_s, \alpha_s)} \opSt_{2 3 4} p \delbar p.
\end{equation}
\end{lemma}

\begin{proof}
By \cref{lem:P-relations}, the relation $(\braidP_{V^*, V^*})_{3 4} (\braid^{-1}_{V, V^*})_{2 3} p \del p = 0$ is equivalent to $\opSt_{2 3 4} p \del p = q^{-(\omega_s, \omega_s)} p \del p$. Using this and $\opT_{1 2 3 4} = \opS_{1 2 3} \opSt_{2 3 4}$ we compute
\[
\del p p = q^{(\alpha_s, \alpha_s)} \opT_{1 2 3 4} p \del p
= q^{(\alpha_s, \alpha_s)} \opS_{1 2 3} \opSt_{2 3 4} p \del p
= q^{(\alpha_s, \alpha_s) - (\omega_s, \omega_s)} \opS_{1 2 3} p \del p.
\]
Similarly, using $\opS_{1 2 3} p \delbar p = q^{(\omega_s, \omega_s)} p \delbar p$ we obtain
\[
\delbar p p = q^{-(\alpha_s, \alpha_s)} \opT_{1 2 3 4} p \delbar p
= q^{- (\alpha_s, \alpha_s)} \opSt_{2 3 4} \opS_{1 2 3} p \delbar p
= q^{(\omega_s, \omega_s) - (\alpha_s, \alpha_s)} \opSt_{2 3 4} p \delbar p. \qedhere
\]
\end{proof}

Finally, we rewrite the relations involving $\braidP_{V, V} \braidQ_{V, V}$ and $\braidP_{V^*, V^*} \braidQ_{V^*, V^*}$ in terms of $\opS$ and $\opSt$.

\begin{lemma}
\label{lem:relations-PQ}
The relations
\[
\begin{gathered}
(\braidP_{V, V})_{1 2} (\braidQ_{V, V})_{1 2} (\braid^{-1}_{V, V^*})_{2 3} p \del p = 0, \\
(\braidP_{V^*, V^*})_{3 4} (\braidQ_{V^*, V^*})_{3 4} (\braid^{-1}_{V, V^*})_{2 3} p \delbar p = 0,
\end{gathered}
\]
of the FODC $\Omega$ are equivalent to
\[
\begin{gathered}
(\opS_{1 2 3} - q^{(\omega_s, \omega_s)}) (\opS_{1 2 3} + q^{(\omega_s, \omega_s) - (\alpha_s, \alpha_s)}) p \del p = 0, \\
(\opSt^{-1}_{2 3 4} - q^{(\omega_s, \omega_s)}) (\opSt^{-1}_{2 3 4} + q^{(\omega_s, \omega_s) - (\alpha_s, \alpha_s)}) p \delbar p = 0.
\end{gathered}
\]
\end{lemma}

\begin{proof}
First, as in \cref{lem:P-relations}, we observe the following easily proven identities
\[
\begin{split}
(\braid_{V, V^*})_{2 3} (\braidP_{V, V})_{1 2} (\braid^{-1}_{V, V^*})_{2 3} & = \opS_{1 2 3} - q^{(\omega_s, \omega_s)}, \\
(\braid_{V, V^*})_{2 3} (\braidQ_{V, V})_{1 2} (\braid^{-1}_{V, V^*})_{2 3} & = \opS_{1 2 3} + q^{(\omega_s, \omega_s) - (\alpha_s, \alpha_s)}.
\end{split}
\]
Next, we note that $(\braidP_{V, V})_{1 2} (\braidQ_{V, V})_{1 2} (\braid^{-1}_{V, V^*})_{2 3} p \del p = 0$ is equivalent to
\[
(\braid_{V, V^*})_{2 3} (\braidP_{V, V})_{1 2} (\braid^{-1}_{V, V^*})_{2 3} (\braid_{V, V^*})_{2 3} (\braidQ_{V, V})_{1 2} (\braid^{-1}_{V, V^*})_{2 3} p \del p = 0,
\]
which leads to the first identity.

For the second identity, using $\opSt^{-1}_{2 3 4} = (\braid_{V, V^*})_{2 3} (\braid_{V^*, V^*})_{3 4} (\braid^{-1}_{V, V^*})_{2 3}$ we obtain
\[
\begin{split}
(\braid_{V, V^*})_{2 3} (\braidP_{V^*, V^*})_{3 4} (\braid^{-1}_{V, V^*})_{2 3} & = \opSt^{-1}_{2 3 4} - q^{(\omega_s, \omega_s)}, \\
(\braid_{V, V^*})_{2 3} (\braidQ_{V^*, V^*})_{3 4} (\braid^{-1}_{V, V^*})_{2 3} & = \opSt^{-1}_{2 3 4} + q^{(\omega_s, \omega_s) - (\alpha_s, \alpha_s)}.
\end{split}
\]
The result then easily follows.
\end{proof}

\subsection{The quadratic case}
\label{sec:quadratic-case}

In this paper we consider the situation when $\braid_{V, V}$ satisfies a quadratic relation, and refer to this as the \emph{quadratic case}.
When $V = V(\omega_s)$, this corresponds to a tensor product decomposition with only two simple factors, that is
\[
V(\omega_s) \otimes V(\omega_s) \cong V(2 \omega_s) \oplus V(2 \omega_s - \alpha_s).
\]
The eigenvalues of the braiding in this case are $q^{(\omega_s, \omega_s)}$ and $- q^{(\omega_s, \omega_s) - (\alpha_s, \alpha_s)}$, corresponding to $V(2 \omega_s)$ and $V(2 \omega_s - \alpha_s)$ respectively. The quadratic relation satisfied by the braiding $\braid_{V, V}$, also known as the Hecke relation in this context, is given by
\[
\braidP_{V, V} \braidQ_{V, V} = (\braid_{V, V} - q^{(\omega_s, \omega_s)}) (\braid_{V, V} + q^{(\omega_s, \omega_s) - (\alpha_s, \alpha_s)}) = 0.
\]
The situation is completely analogous for $\braid_{V^*, V^*}$.
Geometrically, the quadratic case of the Heckenberger-Kolb calculus corresponds to the quantum projective spaces.
Indeed this holds for $\Uqg = U_q(\mathfrak{sl}_{r + 1})$ and the choice $\omega_s = \omega_1$ or $\omega_s = \omega_r$, corresponding to the fundamental representation or its dual, which satisfies the quadratic decomposition above.

In the quadratic case the relations for $\calcP$ and $\calcM$ can be simplified.
Indeed, the first relation of \eqref{eq:calcP-relations} is automatically satisfied, due to the quadratic relation for $\braid_{V, V}$.
Similarly for the first relation of \eqref{eq:calcM-relations}, due to the quadratic relation for $\braid_{V^*, V^*}$.
Taking into account the presentation in terms of $\opS$ and $\opSt$ discussed above, we obtain the following description.

In the quadratic case, $\calcP$ is generated as a left $\subalg$-module by $\del p$ with relations
\begin{equation}
\label{eq:calcP-quadratic}
(\opSt_{2 3 4} - q^{-(\omega_s, \omega_s)}) p \del p = 0, \quad
\EVp_{1 2} \del p = 0.
\end{equation}
Similarly, $\calcM$ is generated by $\delbar p$ with relations
\begin{equation}
\label{eq:calcM-quadratic}
(\opS_{1 2 3} - q^{(\omega_s, \omega_s)}) p \delbar p = 0, \quad
\EVp_{1 2} \delbar p = 0.
\end{equation}

Let us also consider the FODCs $\overcalcP$ and $\overcalcM$ over the larger algebra $\alg$ in the quadratic case.
The first two relations in \eqref{eq:overcalc-relations} are identically satisfied and we are left with
\begin{equation}
\label{eq:overcalc-quadratic}
\EV_{1 2} v \del f = 0, \quad
\EVp_{1 2} f \delbar v = 0.
\end{equation}

Finally, as a consequence of the quadratic relation for the braiding, we have the following relations for the maps $\opS$ and $\opSt$, as shown in \cref{lem:S-quadratic}.

\begin{lemma}
In the quadratic case we have the relations
\begin{align}
\label{eq:S-quadratic1}
\opS_{1 2 3} & = q^{2 (\omega_s, \omega_s) - (\alpha_s, \alpha_s)} \opS_{1 2 3}^{-1} + q^{(\omega_s, \omega_s)} (1 - q^{- (\alpha_s, \alpha_s)}), \\
\label{eq:S-quadratic2}
\opSt_{2 3 4} & = q^{(\alpha_s, \alpha_s) - 2 (\omega_s, \omega_s)} \opSt_{2 3 4}^{-1} + q^{- (\omega_s, \omega_s)} (1 - q^{(\alpha_s, \alpha_s)}).
\end{align}
\end{lemma}

Let us also record that, due to the quadratic condition, we have the identities
\begin{equation}
\label{eq:S-bimodule}
(\opS_{1 2 3} - q^{(\omega_s, \omega_s)}) (p \del p + \del p p) = 0, \quad
(\opSt_{2 3 4} - q^{-(\omega_s, \omega_s)}) (p \delbar p + \delbar p p) = 0.
\end{equation}
To see this, we use the right $\subalg$-module relations from \cref{lem:right-module} to write
\[
p \del p + \del p p = (1 + q^{(\alpha_s, \alpha_s) - (\omega_s, \omega_s)} \opS_{1 2 3}) p \del p, \quad
p \delbar p + \delbar p p = (1 + q^{(\omega_s, \omega_s) - (\alpha_s, \alpha_s)} \opSt_{2 3 4}) p \delbar p,
\]
from which \eqref{eq:S-bimodule} follows upon applying the identities \eqref{eq:S-quadratic1} and \eqref{eq:S-quadratic2}.

\section{Quantum metrics}
\label{sec:quantum-metrics}

In this section we define quantum metrics for the quantum projective spaces, reducing to the Fubini-Study metrics in the classical case.
Our main result here is \cref{thm:quantum-metric}, which shows that these are quantum metrics according to \cref{def:quantum-metric}, that is they are invertible in a suitable sense.
We also discuss various properties they satisfy.

\subsection{Definition and properties}

For this first part there is no particular need to restrict to the case of quantum projective spaces, hence $\Omega$ denotes the Heckenberger-Kolb FODC corresponding to a generic quantum irreducible flag manifold.

We define $\met := \metPM + \metMP$ where we write
\begin{equation}
\label{eq:metric}
\begin{split}
\metPM & := \EVp_{1 2} \EV_{2 3} \del p \otimes \delbar p \in \calcP \otimes_\subalg \calcM, \\
\metMP & := \EVp_{1 2} \EV_{2 3} \delbar p \otimes \del p \in \calcM \otimes_\subalg \calcP.
\end{split}
\end{equation}

\begin{remark}
For the quantum projective spaces, $\met$ reduces to the Fubini-Study metric in the classical limit.
This can be seen from the formula \eqref{eq:classical-metric} in terms of the projection $p$.
\end{remark}

Before tackling the issue of invertibility, we show some properties satisfied by $\met$.
We begin by showing that $\met$ is symmetric (\cref{def:metric-symmetric}) and real (\cref{def:metric-real}).

\begin{proposition}
\label{prop:g-symmetric}
We have that $\met$ is symmetric and real.
\end{proposition}

\begin{proof}
To show that $\met$ is symmetric consider the identity \eqref{eq:deldelbarp}, that is
\[
\del \delbar p = \EV_{2 3} \del p \wedge \delbar p + \EV_{2 3} \delbar p \wedge \del p.
\]
Applying $\EVp_{1 2}$ and using $\EVp_{1 2} \delbar p = 0$ from \eqref{eq:calcM-relations} we obtain
\[
0 = \EVp_{1 2} \EV_{2 3} \del p \wedge \delbar p + \EVp_{1 2} \EV_{2 3} \delbar p \wedge \del p.
\]
Now observe that the right-hand side is equal to $\wedge(\met) = \wedge(\metPM) + \wedge(\metMP)$, where here we consider the wedge product as a map $\wedge: \calc^1 \otimes_\subalg \calc^1 \to \calc^2$.

To show that $\met$ is real it is convenient to employ the usual index notation.
We have $\metPM = \sum_{i, j} q^{(2 \rho, \lambda_i)} \del p^{i j} \otimes \delbar p^{j i}$ and $\metMP = \sum_{i, j} q^{(2 \rho, \lambda_i)} \delbar p^{i j} \otimes \del p^{j i}$.
Using $(\del p^{i j})^* = \delbar p^{j i}$ and $(\delbar p^{i j})^* = \del p^{j i}$ we easily check that $\metPM^\dagger = \metMP$ and hence $\met^\dagger = \met$.
\end{proof}

\begin{remark}
It is also possible to show that $\met$ is left $\CqG$-coinvariant, which amounts to a computation similar to that of \cite[Lemma 5.4]{kahler}.
\end{remark}

The next property we want to discuss is related to Kähler metrics.
First, we need the following technical result on the vanishing of certain terms of degree $3$.
This is essentialy proven in \cite[Lemma 5.3]{kahler}, but we revisit it here using the index-free notation.

From now on all tensor products are over $\subalg$ (omitted), except where specified.

\begin{lemma}
\label{lem:vanishing-third}
We have
\[
\begin{gathered}
\EVp_{1 2} \EV_{2 3} \EV_{2 3} \del p \otimes \delbar p \otimes \del p = 0, \\
\EVp_{1 2} \EV_{2 3} \EV_{2 3} \delbar p \otimes \del p \otimes \delbar p = 0.
\end{gathered}
\]
\end{lemma}

\begin{proof}
Write $A = \EVp_{1 2} \EV_{2 3} \EV_{2 3} \del p \otimes \delbar p \otimes \del p$ for the first term.
Using the identity $\del p = \EV_{2 3} \del p p$ from \eqref{eq:evaluations-calculi} and the right $\subalg$-module relations \eqref{eq:right-module} we compute
\[
\begin{split}
A & = \EVp_{1 2} \EV_{2 3} \EV_{2 3} \EV_{6 7} \del p \otimes \delbar p \otimes \del p p \\
& = q^{(\alpha_s, \alpha_s)} \EVp_{1 2} \EV_{2 3} \EV_{2 3} \EV_{6 7} \opT_{5 6 7 8} \opT_{3 4 5 6} \opT_{1 2 3 4} p \del p \otimes \delbar p \otimes \del p \\
& = q^{(\alpha_s, \alpha_s)} \EVp_{1 2} \EV_{2 3} \EV_{4 5} \opT_{3 4 5 6} \EV_{2 3} \opT_{3 4 5 6} \opT_{1 2 3 4} p \del p \otimes \delbar p \otimes \del p.
\end{split}
\]
In the last step we have used $\EV_{2 3} \EV_{6 7} \opT_{5 6 7 8} = \EV_{4 5} \opT_{3 4 5 6} \EV_{2 3}$, as $\EV$ is an evaluation.
Now consider the identity $\EV_{2 3} \opT_{3 4 5 6} \opT_{1 2 3 4} = \opT_{1 2 3 4} \EV_{4 5}$ from \eqref{eq:identity-ETT}.
Using it twice we get
\[
\begin{split}
A & = q^{(\alpha_s, \alpha_s)} \EVp_{1 2} \EV_{2 3} \EV_{4 5} \opT_{3 4 5 6} \opT_{1 2 3 4} \EV_{4 5} p \del p \otimes \delbar p \otimes \del p \\
& = q^{(\alpha_s, \alpha_s)} \EVp_{1 2} \EV_{2 3} \EV_{2 3} \opT_{3 4 5 6} \opT_{1 2 3 4} \EV_{4 5} p \del p \otimes \delbar p \otimes \del p \\
& = q^{(\alpha_s, \alpha_s)} \EVp_{1 2} \EV_{2 3} \opT_{1 2 3 4} \EV_{4 5} \EV_{4 5} p \del p \otimes \delbar p \otimes \del p.
\end{split}
\]
Next, using $\EVp_{1 2} \EV_{2 3} \opT_{1 2 3 4} = \EVp_{1 2} \EV_{2 3}$ from \eqref{eq:S-evaluation2} we have
\[
A = q^{(\alpha_s, \alpha_s)} \EVp_{1 2} \EV_{2 3} \EV_{4 5} \EV_{4 5} p \del p \otimes \delbar p \otimes \del p.
\]
Finally using $\EV_{2 3} \EV_{4 5} \EV_{4 5} = \EV_{2 3} \EV_{2 3} \EV_{2 3}$ and $\EV_{2 3} p \del p = 0$ we obtain
\[
A = q^{(\alpha_s, \alpha_s)} \EVp_{1 2} \EV_{2 3} \EV_{2 3} \EV_{2 3} p \del p \otimes \delbar p \otimes \del p = 0.
\]

The second identity is similar. Write $B = \EVp_{1 2} \EV_{2 3} \EV_{2 3} \delbar p \otimes \del p \otimes \delbar p$.
Using the identity $\delbar p = \EV_{2 3} p \delbar p$ from \eqref{eq:evaluations-calculi} and the right $\subalg$-module relations \eqref{eq:right-module} we compute
\[
\begin{split}
B & = \EVp_{1 2} \EV_{2 3} \EV_{2 3} \EV_{2 3} p \delbar p \otimes \del p \otimes \delbar p \\
& = q^{(\alpha_s, \alpha_s)} \EVp_{1 2} \EV_{2 3} \EV_{2 3} \EV_{2 3} \opT^{-1}_{1 2 3 4} \opT^{-1}_{3 4 5 6} \opT^{-1}_{5 6 7 8} \delbar p \otimes \del p \otimes \delbar p p \\
& = q^{(\alpha_s, \alpha_s)} \EVp_{1 2} \EV_{2 3} \EV_{2 3} \EV_{4 5} \opT^{-1}_{1 2 3 4} \opT^{-1}_{3 4 5 6} \opT^{-1}_{5 6 7 8} \delbar p \otimes \del p \otimes \delbar p p.
\end{split}
\]
We have $\EV_{4 5} \opT^{-1}_{1 2 3 4} \opT^{-1}_{3 4 5 6} = \opT^{-1}_{1 2 3 4} \EV_{2 3}$, again from \eqref{eq:identity-ETT}. Then
\[
\begin{split}
B & = q^{(\alpha_s, \alpha_s)} \EVp_{1 2} \EV_{2 3} \EV_{2 3} \opT^{-1}_{1 2 3 4} \EV_{2 3} \opT^{-1}_{5 6 7 8} \delbar p \otimes \del p \otimes \delbar p p \\
& = q^{(\alpha_s, \alpha_s)} \EVp_{1 2} \EV_{2 3} \EV_{2 3} \opT^{-1}_{1 2 3 4} \opT^{-1}_{3 4 5 6} \EV_{2 3} \delbar p \otimes \del p \otimes \delbar p p \\
& = q^{(\alpha_s, \alpha_s)} \EVp_{1 2} \EV_{2 3} \EV_{4 5} \opT^{-1}_{1 2 3 4} \opT^{-1}_{3 4 5 6} \EV_{2 3} \delbar p \otimes \del p \otimes \delbar p p \\
& = q^{(\alpha_s, \alpha_s)} \EVp_{1 2} \EV_{2 3} \opT^{-1}_{1 2 3 4} \EV_{2 3} \EV_{2 3} \delbar p \otimes \del p \otimes \delbar p p.
\end{split}
\]
Using $\EVp_{1 2} \EV_{2 3} \opT^{-1}_{1 2 3 4} = \EVp_{1 2} \EV_{2 3}$ from \eqref{eq:S-evaluation2} we rewrite
\[
B = q^{(\alpha_s, \alpha_s)} \EVp_{1 2} \EV_{2 3} \EV_{2 3} \EV_{2 3} \delbar p \otimes \del p \otimes \delbar p p.
\]
Finally using $\EV_{2 3} \EV_{2 3} \EV_{2 3} = \EV_{2 3} \EV_{2 3} \EV_{6 7}$ and $\EV_{2 3} \delbar p p = 0$ we obtain
\[
B = q^{(\alpha_s, \alpha_s)} \EVp_{1 2} \EV_{2 3} \EV_{2 3} \EV_{6 7} \delbar p \otimes \del p \otimes \delbar p p = 0. \qedhere
\]
\end{proof}

We are now ready to prove the Kähler property of the metric.

\begin{proposition}
\label{prop:metric-kahler}
The metric $\met$ satisfies
\[
(\diff \otimes \id) g = (\id \otimes \diff) g = 0.
\]
\end{proposition}

\begin{proof}
The metric can be written in the form
\[
\met = \metPM + \metMP = \EVp_{1 2} \EV_{2 3} (\del p \otimes \delbar p + \delbar p \otimes \del p).
\]
We apply $\diff$ to the first leg. Using $\diff = \del + \delbar$ and $\del^2 = \delbar^2 = 0$ we obtain
\[
(\diff \otimes \id) \met = \EVp_{1 2} \EV_{2 3} (\delbar \del p \otimes \delbar p + \del \delbar p \otimes \del p).
\]
We have $\del \delbar p = \EV_{2 3} (\del p \wedge \delbar p + \delbar p \wedge \del p)$ by \eqref{eq:deldelbarp}.
Using this and $\del \delbar = - \delbar \del$ we get
\[
\begin{split}
(\diff \otimes \id) \met & = - \EVp_{1 2} \EV_{2 3} \EV_{2 3} (\del p \wedge \delbar p \otimes \delbar p + \delbar p \wedge \del p \otimes \delbar p) \\
& + \EVp_{1 2} \EV_{2 3} \EV_{2 3} (\del p \wedge \delbar p \otimes \del p + \delbar p \wedge \del p \otimes \del p).
\end{split}
\]
The second and third term of the previous expression vanish by \cref{lem:vanishing-third}.
The other two terms also vanish, since we have the relations $\EV_{2 3} \del p \otimes \del p = \EV_{2 3} \delbar p \otimes \delbar p = 0$, which can be easily proven using \eqref{eq:evaluations-calculi} and keeping in mind that the tensor product is over $\subalg$.

Similarly, applying $\diff$ to the second leg we get
\[
(\id \otimes \diff) \met = \EVp_{1 2} \EV_{2 3} (\del p \otimes \del \delbar p + \delbar p \otimes \delbar \del p).
\]
Using again the expression for $\del \delbar p$ we obtain
\[
\begin{split}
(\id \otimes \diff) \met & = \EVp_{1 2} \EV_{2 3} \EV_{4 5} (\del p \otimes \del p \wedge \delbar p + \del p \otimes \delbar p \wedge \del p) \\
& - \EVp_{1 2} \EV_{2 3} \EV_{4 5} (\delbar p \otimes \del p \wedge \delbar p + \delbar p \otimes \delbar p \wedge \del p).
\end{split}
\]
This is easily seen to vanish as in the previous case.
\end{proof}

\begin{remark}
The definition of the metric $\met$ and its properties discussed above are valid for any quantum irreducible flag manifold, not just in the quadratic case (we never used this in the proofs).
This is completely analogous to the definition of the Kähler forms discussed in \cite{kahler}.
See also \cite{kahler-hochschild} for a discussion of general quantum flag manifolds, not necessarily irreducibles, from the point of view of (twisted) Hochschild homology.
\end{remark}

\subsection{Inverse metric}

From now on we focus on the case of quantum projective spaces.
This means that we take $\Uqg = U_q(\mathfrak{sl}_{r + 1})$ and choose either $\omega_s = \omega_1$ or $\omega_s = \omega_r$, in such a way that the quadratic condition discussed in \cref{sec:quadratic-case} is satisfied.

Our goal is to prove that $\met$ is a quantum metric according to \cref{def:quantum-metric}.
To show this, we need an "inverse metric", which is an appropriate $\subalg$-bimodule map $(\cdot, \cdot): \calc \otimes_\subalg \calc \to \subalg$.
We begin by defining a certain $\subalg$-bimodule map on $\calcM \otimes_\subalg \calcP$.

\begin{lemma}
We have a $\subalg$-bimodule map $(\cdot, \cdot)_{- +}: \calcM \otimes_\subalg \calcP \to \subalg$ defined by
\[
(\delbar p, \del p)_{- +} = \CVp_2 p - q^{-(\omega_s, 2 \rho)} p p.
\]
\end{lemma}

\begin{proof}
We use the FODCs $\overcalcP$ and $\overcalcM$ over $\alg$ introduced in \cref{sec:overcalc} and some results from \cref{sec:bimodule-maps}.
According to \cref{prop:bimodule-metric} we have an $\alg$-bimodule map
\[
\PhiMP - q^{-(\omega_s, 2 \rho)} \PsiMP: \overcalcM \otimes_\subalg \overcalcP \to \alg,
\]
where $\PhiMP$ and $\PsiMP$ are the $\alg$-bimodule maps given by
\[
\PhiMP(\delbar v \otimes \del f) = \CVp_1, \quad
\PsiMP(\delbar v \otimes \del f) = v f.
\]
Recall that $\calcP$ and $\calcM$ are induced from $\overcalcP$ and $\overcalcM$ over $\alg$.
We now show that $\PhiMP - q^{-(\omega_s, 2 \rho)} \PsiMP$ restricts to a map $\calcM \otimes_\subalg \calcP \to \subalg$, which coincides with $(\cdot, \cdot)_{- +}$.

Starting with $\PhiMP$ we compute
\[
\PhiMP(\delbar p \otimes \del p)
= \PhiMP(f \delbar v \otimes \del f v)
= f \CVp_1 v = \CVp_2 f v = \CVp_2 p.
\]
Similarly for $\PsiMP$ we compute
\[
\PsiMP(\delbar p \otimes \del p)
= \PsiMP(f \delbar v \otimes \del f v)
= f v f v = p p.
\]
Therefore $(\delbar p, \del p)_{- +} = (\PhiMP - q^{-(\omega_s, 2 \rho)} \PsiMP)(\delbar p \otimes \del p)$, which gives the result.
\end{proof}

Similarly we define a $\subalg$-bimodule map on $\calcP \otimes_\subalg \calcM$, which is more involved.

\begin{lemma}
We have a $\subalg$-bimodule map $(\cdot, \cdot)_{+ -}: \calcP \otimes_\subalg \calcM \to \subalg$ defined by
\[
(\del p, \delbar p)_{+ -} = q^{(\alpha_s, \alpha_s)} q^{-(\omega_s, \omega_s + 2 \rho)} \opS_{1 2 3} \CV_3 p - q^{(\alpha_s, \alpha_s)} q^{-(\omega_s, 2 \rho)} p p.
\]
\end{lemma}

\begin{proof}
According to \cref{prop:bimodule-metric} we have an $\alg$-bimodule map
\[
\PhiPM - \PsiPM: \overcalcP \otimes_\subalg \overcalcM \to \alg,
\]
where $\PhiPM$ and $\PsiPM$ are the $\alg$-bimodule maps given by
\[
\PhiPM(\del f \otimes \delbar v) = \CV_1, \quad
\PsiPM(\del f \otimes \delbar v) = f v.
\]
We now show that $\PhiPM - \PsiPM$ restricts to a map $\calcP \otimes_\subalg \calcM \to \subalg$.

Consider first $\PhiPM$.
Using the right $\alg$-module relations \eqref{eq:overcalcP-right} and \eqref{eq:overcalcM-right} we compute
\[
\begin{split}
\PhiPM(\del p \otimes \delbar p)
& = q^{-2 (\omega_s, \omega_s)} (\braid^{-1}_{V, V^*})_{1 2} (\braid^{-1}_{V, V^*})_{3 4} \PhiPM(v \del f \otimes \delbar v f) \\
& = q^{-2 (\omega_s, \omega_s)} (\braid^{-1}_{V, V^*})_{1 2} (\braid^{-1}_{V, V^*})_{3 4} \CV_2 (v f) \\
& = q^{-(\omega_s, \omega_s)} (\braid^{-1}_{V, V^*})_{1 2} (\braid^{-1}_{V, V^*})_{3 4} \CV_2 (\braid_{V, V^*})_{1 2} p.
\end{split}
\]
In the last step we have used $v f = q^{(\omega_s, \omega_s)} (\braid_{V, V^*})_{1 2} f v$ from \eqref{eq:relationsA}.
Next, we use the identity $(\braid^{-1}_{V, V^*})_{3 4} \CV_2 = (\braid_{V, V})_{2 3} \CV_3$ from \eqref{eq:coevaluations}. Then
\[
\begin{split}
\PhiPM(\del p \otimes \delbar p)
& = q^{-(\omega_s, \omega_s)} (\braid^{-1}_{V, V^*})_{1 2} (\braid_{V, V})_{2 3} \CV_3 (\braid_{V, V^*})_{1 2} p \\
& = q^{-(\omega_s, \omega_s)} (\braid^{-1}_{V, V^*})_{1 2} (\braid_{V, V})_{2 3} (\braid_{V, V^*})_{1 2} \CV_3 p \\
& = q^{-(\omega_s, \omega_s)} (\braid_{V, V^*})_{2 3} (\braid_{V, V})_{1 2} (\braid^{-1}_{V, V^*})_{2 3} \CV_3 p.
\end{split}
\]
Finally using the braid equation \eqref{eq:braid-equation} we obtain
\[
\begin{split}
\PhiPM(\del p \otimes \delbar p)
& = q^{-(\omega_s, \omega_s)} (\braid_{V, V^*})_{2 3} (\braid_{V, V})_{1 2} (\braid^{-1}_{V, V^*})_{2 3} \CV_3 p \\
& = q^{-(\omega_s, \omega_s)} \opS_{1 2 3} \CV_3 p.
\end{split}
\]
Similarly, for $\PsiPM$ we compute
\[
\begin{split}
\PsiPM(\del p \otimes \delbar p)
& = q^{-2 (\omega_s, \omega_s)} (\braid^{-1}_{V, V^*})_{1 2} (\braid^{-1}_{V, V^*})_{3 4} \PsiPM(v \del f \otimes \delbar v f) \\
& = q^{-2 (\omega_s, \omega_s)} (\braid^{-1}_{V, V^*})_{1 2} (\braid^{-1}_{V, V^*})_{3 4} (v f v f) \\
& = f v f v = p p.
\end{split}
\]

Using the identities above we find that
\[
q^{-(\alpha_s, \alpha_s)} q^{(\omega_s, 2 \rho)} (\del p, \delbar p)_{+ -} = (\PhiPM - \PsiPM)(\del p \otimes \delbar p).
\]
This proves the claim about $(\cdot, \cdot)_{+ -}$.
\end{proof}

\begin{remark}
The normalization factor in this map is chosen for later convenience.
\end{remark}

Hence we can define a $\subalg$-bimodule map $(\cdot, \cdot): \calc \otimes_\subalg \calc \to \subalg$ by
\[
(\cdot, \cdot) :=
\begin{cases}
(\cdot, \cdot)_{+ -} & \textrm{on } \calcP \otimes \calcM \\
(\cdot, \cdot)_{- +} & \textrm{on } \calcM \otimes \calcP \\
0 & \textrm{otherwise}
\end{cases}.
\]

\begin{remark}
In the classical limit the map $(\cdot, \cdot): \calc \otimes_\subalg \calc \to \subalg$ reduces to the inverse of the Fubini-Study metric, see the explicit formulae in \eqref{eq:inverse-metric}.
\end{remark}

We are now ready to show that $\met$ is a quantum metric according to \cref{def:quantum-metric}.

\begin{theorem}
\label{thm:quantum-metric}
Write $\met = \met^{(1)} \otimes \met^{(2)}$. Then for any $\omega \in \calc$ we have
\[
\met^{(1)} (\met^{(2)}, \omega) = \omega = (\omega, \met^{(1)}) \met^{(2)}.
\]
Hence $\met \in \calc \otimes_\subalg \calc$ is a quantum metric.
\end{theorem}

\begin{proof}
Since $(\cdot, \cdot)$ is a $\subalg$-bimodule map, it suffices to prove the claim for the generators $\del p$ and $\delbar p$ of $\calc$.
We have to consider four different cases.

$\bullet$ \textbf{The case $\met^{(1)} (\met^{(2)}, \del p)$}. Since $(\del p, \del p) = 0$ we have
\[
\met^{(1)} (\met^{(2)}, \del p)
= \EVp_{1 2} \EV_{2 3} \del p (\delbar p, \del p) = \EVp_{1 2} \EV_{2 3} \del p \left( \CVp_2 p - q^{-(\omega_s, 2 \rho)} p p \right).
\]
We have $\EVp_{1 2} \EV_{2 3} \del p p p = \EVp_{1 2} \del p p = 0$. Hence
\[
\met^{(1)} (\met^{(2)}, \del p)
= \EVp_{1 2} \EV_{2 3} \del p \CVp_2 p
= \EVp_{1 2} \EV_{2 3} \CVp_4 \del p p.
\]
Using the duality relation $\EVp_{1 2} \CVp_2 = \id$ from \eqref{eq:duality} we obtain
\[
\met^{(1)} (\met^{(2)}, \del p)
= \EVp_{1 2} \CVp_2 \EV_{2 3} \del p p
= \EV_{2 3} \del p p = \del p.
\]

$\bullet$ \textbf{The case $(\delbar p, \met^{(1)}) \met^{(2)}$}. Since $(\delbar p, \delbar p) = 0$ we have
\[
(\delbar p, \met^{(1)}) \met^{(2)}
= \EVp_{3 4} \EV_{4 5} (\delbar p, \del p) \delbar p = \EVp_{3 4} \EV_{4 5} \left( \CVp_2 p - q^{-(\omega_s, 2 \rho)} p p \right) \delbar p.
\]
We have $\EVp_{3 4} \EV_{4 5} p p \delbar p = \EVp_{3 4} p \delbar p = 0$. Hence
\[
(\delbar p, g^{(1)}) g^{(2)} = \EVp_{3 4} \EV_{4 5} \CVp_2 p \delbar p.
\]
Using the duality relation $\EVp_{3 4} \CVp_2 = \id$ from \eqref{eq:duality} we obtain
\[
(\delbar p, g^{(1)}) g^{(2)} = \EVp_{3 4} \CVp_2 \EV_{2 3} p \delbar p = \delbar p.
\]

$\bullet$ \textbf{The case $(\del p, \met^{(1)}) \met^{(2)}$}. Since $(\del p, \del p) = 0$ we have
\[
\begin{split}
(\del p, \met^{(1)}) \met^{(2)}
& = \EVp_{3 4} \EV_{4 5} (\del p, \delbar p) \del p \\
& = \EVp_{3 4} \EV_{4 5} \left( q^{(\alpha_s, \alpha_s)} q^{-(\omega_s, \omega_s + 2 \rho)} \opS_{1 2 3} \CV_3 p - q^{(\alpha_s, \alpha_s)} q^{-(\omega_s, 2 \rho)} p p \right) \del p \\
& = q^{(\alpha_s, \alpha_s)} q^{-(\omega_s, \omega_s + 2 \rho)} \EVp_{3 4} \EV_{4 5} \opS_{1 2 3} \CV_3 p \del p,
\end{split}
\]
where in the last step we have used $\EVp_{3 4} \EV_{4 5} p p \del p = 0$.
Since $\EV_{4 5}$ commutes with $\opS_{1 2 3}$, using the duality relation $\EV_{4 5} \CV_3 = \id$ from \eqref{eq:duality} we get
\[
(\del p, \met^{(1)}) \met^{(2)} = q^{(\alpha_s, \alpha_s)} q^{-(\omega_s, \omega_s + 2 \rho)} \EVp_{3 4} \opS_{1 2 3} p \del p.
\]
Then using $p \del p = q^{(\omega_s, \omega_s) - (\alpha_s, \alpha_s)} \opS^{-1}_{1 2 3} \del p p$ from \eqref{eq:S-right-module} we conclude that
\[
(\del p, \met^{(1)}) \met^{(2)} = q^{-(\omega_s, 2 \rho)} \EVp_{3 4} \del p p = \del p.
\]

$\bullet$ \textbf{The case $\met^{(1)} (\met^{(2)}, \delbar p)$}. Since $(\delbar p, \delbar p) = 0$ we have
\[
\begin{split}
\met^{(1)} (\met^{(2)}, \delbar p)
& = \EVp_{1 2} \EV_{2 3} \delbar p (\del p, \delbar p) \\
& = \EVp_{1 2} \EV_{2 3} \delbar p \left( q^{(\alpha_s, \alpha_s)} q^{-(\omega_s, \omega_s + 2 \rho)} \opS_{1 2 3} \CV_3 p - q^{(\alpha_s, \alpha_s)} q^{-(\omega_s, 2 \rho)} p p \right).
\end{split}
\]
Using again $\EVp_{1 2} \EV_{2 3} \del p p p = 0$ we rewrite
\[
\met^{(1)} (\met^{(2)}, \delbar p)
= q^{(\alpha_s, \alpha_s)} q^{-(\omega_s, \omega_s + 2 \rho)} \EVp_{1 2} \EV_{2 3} \opS_{3 4 5} \CV_5 \delbar p p.
\]
Next, consider the expression
\[
\begin{split}
\EV_{2 3} \opS_{3 4 5} \CV_5
& = \EV_{2 3} (\braid_{V, V^*})_{4 5} (\braid_{V, V})_{3 4} (\braid^{-1}_{V, V^*})_{4 5} \CV_5 \\
& = (\braid_{V, V^*})_{2 3} \EV_{2 3} (\braid_{V, V})_{3 4} (\braid^{-1}_{V, V^*})_{4 5} \CV_5.
\end{split}
\]
By naturality of the braiding, equations \eqref{eq:evaluations} and \eqref{eq:coevaluations} give
\[
\EV_{2 3} (\braid_{V, V})_{3 4} = \EV_{3 4} (\braid^{-1}_{V, V^*})_{2 3}, \quad
(\braid^{-1}_{V, V^*})_{4 5} \CV_5 = (\braid_{V^*, V^*})_{5 6} \CV_4.
\]
Using these and the duality relation $\EV_{1 2} \CV_2 = \id$ from \eqref{eq:duality} we get
\[
\begin{split}
\EV_{2 3} \opS_{3 4 5} \CV_5
& = (\braid_{V, V^*})_{2 3} \EV_{3 4} (\braid^{-1}_{V, V^*})_{2 3} (\braid_{V^*, V^*})_{5 6} \CV_4 \\
& = (\braid_{V, V^*})_{2 3} (\braid_{V^*, V^*})_{3 4} \EV_{3 4} \CV_4 (\braid^{-1}_{V, V^*})_{2 3} \\
& = (\braid_{V, V^*})_{2 3} (\braid_{V^*, V^*})_{3 4} (\braid^{-1}_{V, V^*})_{2 3} = \opSt^{-1}_{2 3 4}.
\end{split}
\]
Therefore we have
\[
\met^{(1)} (\met^{(2)}, \delbar p) = q^{(\alpha_s, \alpha_s)} q^{-(\omega_s, \omega_s + 2 \rho)} \EVp_{1 2} \opSt^{-1}_{2 3 4} \delbar p p.
\]
Finally using $\delbar p p = q^{(\omega_s, \omega_s) - (\alpha_s, \alpha_s)} \opSt_{2 3 4} p \delbar p$ from \eqref{eq:S-right-module} we conclude that
\[
\met^{(1)} (\met^{(2)}, \delbar p) = q^{-(\omega_s, 2 \rho)} \EVp_{1 2} p \delbar p = \delbar p. \qedhere
\]
\end{proof}

\begin{corollary}
We have that $\met$ is central, in the sense that $b \met = \met b$ for all $b \in \subalg$.
\end{corollary}

\begin{proof}
This is a general property of quantum metrics, see \cite[Lemma 1.16]{quantum-book}.
Of course, this can also be proven directly using the relations of the FODC $\Omega$.
\end{proof}

Since $\met = \metPM + \metMP$ this also implies that $\metPM$ and $\metMP$ are central.

\section{Connections and properties}
\label{sec:connections}

In this section we introduce two connections $\connP$ and $\connM$ on the FODCs $\calcP$ and $\calcM$, for the case of quantum projective spaces.
Their direct sum $\conn = \connP + \connM$ is a connection on the FODC $\calc$, which in the classical limit reduces to the Levi-Civita connection on the cotangent bundle.
As for the quantum case, we show that this connection is torsion free and cotorsion free.
In other words, $\nabla$ is a weak quantum Levi-Civita connection in the sense of \cref{def:levi-civita}, which is our main result from \cref{thm:levi-civita}.

\subsection{Connections}

We begin with the connection on $\calcM$, which is easier.

\begin{proposition}
\label{prop:connection-minus}
We have a connection $\connM: \calcM \to \calc \otimes_\subalg \calcM$ defined by
\[
\connM(\delbar p) = \EV_{2 3} \del p \otimes \delbar p - q^{-(\omega_s, 2 \rho)} p \metPM.
\]
\end{proposition}

\begin{proof}
Recall that, in the quadratic case, $\calcM$ is generated as a left $\subalg$-module by $\delbar p$ with the relations $(\opS_{1 2 3} - q^{(\omega_s, \omega_s)}) p \delbar p = 0$ and $\EVp_{1 2} \delbar p = 0$, as described in \eqref{eq:calcP-quadratic}.
Hence to show that $\connM$ is well-defined we need to check the relations
\[
(\opS_{1 2 3} - q^{(\omega_s, \omega_s)}) \connM(p \delbar p) = 0, \quad
\EVp_{1 2} \connM(\delbar p) = 0.
\]
Let us begin with the second relation.
Recall that $\metPM = \EVp_{1 2} \EV_{2 3} \del p \otimes \delbar p$ and we have the identity $\EVp_{1 2} p = q^{(\omega_s, 2 \rho)}$.
Using these we compute
\[
\EVp_{1 2} \connM(\delbar p) = \metPM - \metPM = 0.
\]

Next, we want to show that $(\opS_{1 2 3} - q^{(\omega_s, \omega_s)}) \connM(p \delbar p) = 0$. Using the Leibniz rule we have
\[
\begin{split}
\connM (p \delbar p) & = \diff p \otimes \delbar p + p \connM(\delbar p) \\
& = \del p \otimes \delbar p + \delbar p \otimes \delbar p + \EV_{4 5} (p \del p \otimes \delbar p) - q^{-(\omega_s, 2 \rho)} p p g_{+ -}.
\end{split}
\]
First we claim that $(\opS_{1 2 3} - q^{(\omega_s, \omega_s)}) \delbar p \otimes \delbar p = 0$.
To see this, we use \eqref{eq:evaluations-calculi} to write
\[
\delbar p \otimes \delbar p = \EV_{4 5} \delbar p \otimes p \delbar p = \EV_{4 5} \delbar p p \otimes \delbar p.
\]
Then using $(\opS_{1 2 3} - q^{(\omega_s, \omega_s)}) \delbar p p = 0$ from \eqref{eq:S-action-right} we obtain the claim.
Similarly one shows that the term $p p g_{+ -}$ vanishes under $\opS_{1 2 3} - q^{(\omega_s, \omega_s)}$.
Hence let us consider
\[
A = \del p \otimes \delbar p + \EV_{4 5} p \del p \otimes \delbar p,
\]
which we want to vanish under $\opS_{1 2 3} - q^{(\omega_s, \omega_s)}$.
Using $\delbar p = \EV_{2 3} p \delbar p$ we rewrite it as
\[
A = \EV_{4 5} \del p \otimes p \delbar p + \EV_{4 5} p \del p \otimes \delbar p = \EV_{4 5} (\del p p + p \del p) \otimes \delbar p.
\]
Clearly $\opS_{1 2 3} - q^{(\omega_s, \omega_s)}$ commutes with $\EV_{4 5}$.
Finally, using $(\opS_{1 2 3} - q^{(\omega_s, \omega_s)}) (\del p p + p \del p) = 0$ from \eqref{eq:S-bimodule}, we conclude that $(\opS_{1 2 3} - q^{(\omega_s, \omega_s)}) A = 0$.
\end{proof}

Next we consider the case of $\calcP$, which is more complicated.

\begin{proposition}
\label{prop:connection-plus}
We have a connection $\connP: \calcP \to \calc \otimes_\subalg \calcP$ defined by
\[
\connP(\del p) = q^{(\alpha_s, \alpha_s)} \EV_{2 3} \opT_{1 2 3 4} \delbar p \otimes \del p - q^{(\alpha_s, \alpha_s)} q^{-(\omega_s, 2 \rho)} p \metMP.
\]
\end{proposition}

\begin{proof}
Recall that, in the quadratic case, $\calcP$ is generated as a left $\subalg$-module by $\del p$ with the relations $(\opSt_{2 3 4} - q^{-(\omega_s, \omega_s)}) p \del p = 0$ and $\EVp_{1 2} \del p = 0$, as described in \eqref{eq:calcM-quadratic}.
Hence to show that $\connP$ is well-defined we need to check the relations
\[
(\opSt_{2 3 4} - q^{-(\omega_s, \omega_s)}) \connP(p \del p) = 0, \quad
\EVp_{1 2} \connP(\del p) = 0.
\]

Let us begin with the second relation.
We can use the identity $\EVp_{1 2} \EV_{2 3} \opT_{1 2 3 4} = \EVp_{1 2} \EV_{2 3}$ from \eqref{eq:identity-EpET}.
Then we obtain the expression
\[
\EVp_{1 2} \EV_{2 3} \opT_{1 2 3 4} \delbar p \otimes \del p
= \EVp_{1 2} \EV_{2 3} \delbar p \otimes \del p = \metMP.
\]
Using this and $\EVp_{1 2} p = q^{(\omega_s, 2 \rho)}$ we conclude that
\[
\EVp_{1 2} \connP(\del p) = q^{(\alpha_s, \alpha_s)} \metMP - q^{(\alpha_s, \alpha_s)} \metMP = 0.
\]

Next we want to show that $(\opSt_{2 3 4} - q^{-(\omega_s, \omega_s)}) \connP(p \del p) = 0$. We have
\[
\begin{split}
\connP(p \del p) & = \diff p \otimes \del p + p \connP(\del p) \\
& = \del p \otimes \del p + \delbar p \otimes \del p + q^{(\alpha_s, \alpha_s)} p \EV_{2 3} \opT_{1 2 3 4} \delbar p \otimes \del p - q^{(\alpha_s, \alpha_s)} q^{-(\omega_s, 2 \rho)} p p \metMP.
\end{split}
\]
The terms $\del p \otimes \del p$ and $p p$ can be shown to vanish under $\opSt_{2 3 4} - q^{-(\omega_s, \omega_s)}$, exactly as in the proof of \cref{prop:connection-minus}.
Hence let us consider the term
\[
A = \delbar p \otimes \del p + q^{(\alpha_s, \alpha_s)} \EV_{4 5} \opT_{3 4 5 6} p \delbar p \otimes \del p.
\]
We want to show that it vanishes under $\opSt_{2 3 4} - q^{-(\omega_s, \omega_s)}$.
We can rewrite it using
\[
\delbar p \otimes \del p
= \EV_{4 5} \delbar p \otimes \del p p
= q^{(\alpha_s, \alpha_s)} \EV_{4 5} \opT_{3 4 5 6} \delbar p p \otimes \del p.
\]
Hence we obtain the expression
\[
A = q^{(\alpha_s, \alpha_s)} \EV_{4 5} \opT_{3 4 5 6} (\delbar p p + p \delbar p) \otimes \del p.
\]
Now consider the identity $\opSt_{2 3 4} \EV_{4 5} \opSt_{4 5 6} = \EV_{4 5} \opSt_{4 5 6} \opSt_{2 3 4}$ from \eqref{eq:identity-StESt}.
Multiplying by $\opS_{3 4 5}$ on the right and using the fact that $\opS$ and $\opSt$ commute we get
\[
\opSt_{2 3 4} \EV_{4 5} \opT_{3 4 5 6} = \EV_{4 5} \opT_{3 4 5 6} \opSt_{2 3 4}.
\]
Then we can commute $\opSt_{2 3 4}$ with $\EV_{4 5} \opT_{3 4 5 6}$ to obtain
\[
(\opSt_{2 3 4} - q^{-(\omega_s, \omega_s)}) A = q^{(\alpha_s, \alpha_s)} \EV_{4 5} \opT_{3 4 5 6} (\opSt_{2 3 4} - q^{-(\omega_s, \omega_s)}) (\delbar p p + p \delbar p) \otimes \del p.
\]
Finally, we can use the identity $(\opSt_{2 3 4} - q^{-(\omega_s, \omega_s)}) (\delbar p p + p \delbar p) = 0$ from \eqref{eq:S-bimodule} to conclude that $(\opSt_{2 3 4} - q^{-(\omega_s, \omega_s)}) A = 0$, which finishes the proof.
\end{proof}

In the following we write
\[
\conn := \connP + \connM: \Omega \to \Omega \otimes_\subalg \Omega
\]
for the direct sum of the two connections on $\calc = \calcP \oplus \calcM$.

\begin{remark}
In the classical limit, the connection $\nabla$ reduces to the Levi-Civita connection on the cotangent bundle, as can be seen from the formulae in \eqref{eq:classical-connection}.
\end{remark}

\begin{remark}
The constructions of this paper can be put in a $\CqG$-covariant setting, but we do not give the details here, as this is not our main goal.
Recall that the algebra $\subalg$ is a left $\CqG$-comodule, while the FODCs $\calcP$ and $\calcM$ are shown to be left covariant in \cite{heko}.
It is possible to show that the connections $\connP$ and $\connM$ introduced here are left $\CqG$-covariant.
This means that $\connP: \calcP \to \calc \otimes_\subalg \calcP$ (and similarly $\connM$) is a map of left $\CqG$-comodules, where $\calc \otimes_\subalg \calcP$ is given the usual structure of tensor product of left comodules.

In \cite[Theorem 4.5]{connections-irreducible} it is shown, using representation-theoretic methods, that in the case of quantum irreducible flag manifolds there exists a unique left \emph{covariant} connection for comodules belonging to a certain class.
In particular this applies to $\calcP$ and $\calcM$, hence giving (non-explicitly) the connections we have introduced here.
Moreover, the argument of \cite{connections-irreducible} can be easily adapted to give a uniqueness result for $\calcP \oplus \calcM$.
This implies that the connection $\conn$ described here must coincide with the one introduced in \cite{projective-bundle}.
\end{remark}

\subsection{Torsion}

The first property of the connection $\conn$ we want to explore is its torsion, which according to \cref{def:torsion} is the left $\subalg$-module map
\[
T_\nabla = \wedge \circ \nabla - \diff: \calc^1 \to \calc^2.
\]

\begin{proposition}
\label{prop:torsion}
We have $\Tnabla = 0$, that is $\nabla$ is torsion free.
\end{proposition}

\begin{proof}
It suffices to show that $\Tnabla$ vanishes on the generators $\del p$ and $\delbar p$, since $\Tnabla$ is a $\subalg$-module map.
For the rest of this proof we write $\kappa_{+ -} = \wedge(\metPM)$ and $\kappa_{- +} = \wedge(\metMP)$.
Moreover observe that $\kappa_{+ -} = - \kappa_{- +}$, since $\met$ is symmetric by \cref{prop:g-symmetric}.

Consider first $\Tnabla(\delbar p)$.
We have $\diff \delbar p = \del \delbar p$, since $\diff = \del + \delbar$.
Moreover we can write $\del \delbar p = \EV_{2 3} (\del p \wedge \delbar p + \delbar p \wedge \del p)$, as in \eqref{eq:deldelbarp}.
Therefore
\[
\begin{split}
\Tnabla(\delbar p) & = \EV_{2 3} \del p \wedge \delbar p - q^{-(\omega_s, 2 \rho)} p \kappa_{+ -} - \EV_{2 3} (\del p \wedge \delbar p + \delbar p \wedge \del p) \\
& = q^{-(\omega_s, 2 \rho)} p \kappa_{- +} - \EV_{2 3} \delbar p \wedge \del p.
\end{split}
\]
Next, consider the identity $\EV_{2 3} \delbar p \otimes \del p = q^{-(\omega_s, 2 \rho)} p \metMP$ from \eqref{eq:identity-metMP}.
Applying $\wedge$ to it gives $\EV_{2 3} \delbar p \wedge \del p = q^{-(\omega_s, 2 \rho)} p \kappa_{- +}$. Hence we conclude that
\[
\Tnabla(\delbar p) = q^{-(\omega_s, 2 \rho)} p \kappa_{- +} - q^{-(\omega_s, 2 \rho)} p \kappa_{- +} = 0.
\]

Now consider $\Tnabla(\del p)$. Using $\diff \del p = \delbar \del p = - \EV_{2 3} (\del p \wedge \delbar p + \delbar p \wedge \del p)$ we have
\[
\begin{split}
\Tnabla(\del p) & = q^{(\alpha_s, \alpha_s)} \EV_{2 3} \opT_{1 2 3 4} \delbar p \wedge \del p - q^{(\alpha_s, \alpha_s)} q^{-(\omega_s, 2 \rho)} p \kappa_{- +} \\
& + \EV_{2 3} (\del p \wedge \delbar p + \delbar p \wedge \del p).
\end{split}
\]
In \cref{lem:identity-torsion} we show that we have the identity
\[
q^{(\alpha_s, \alpha_s)} \EV_{2 3} \opT_{1 2 3 4} \delbar p \wedge \del p = - \EV_{2 3} \del p \wedge \delbar p + (q^{(\alpha_s, \alpha_s)} - 1) \EV_{2 3} \delbar p \wedge \del p.
\]
Plugging this in and simplifying we obtain
\[
\Tnabla(\del p) = q^{(\alpha_s, \alpha_s)} \EV_{2 3} \delbar p \wedge \del p - q^{(\alpha_s, \alpha_s)} q^{-(\omega_s, 2 \rho)} p \kappa_{- +}.
\]
Using again $\EV_{2 3} \delbar p \wedge \del p = q^{-(\omega_s, 2 \rho)} p \kappa_{- +}$, we conclude that
\[
\Tnabla(\del p) = q^{(\alpha_s, \alpha_s)} q^{-(\omega_s, 2 \rho)} p \kappa_{- +} - q^{(\alpha_s, \alpha_s)} q^{-(\omega_s, 2 \rho)} p \kappa_{- +} = 0. \qedhere
\]
\end{proof}

\subsection{Cotorsion}

From \cref{def:cotorsion}, the cotorsion corresponding to a connection $\nabla: \Omega^1 \to \Omega^1 \otimes_A \Omega^1$ and a quantum metric $\met$ is the element $\coTnabla \in \Omega^2 \otimes_A \Omega^1$ given by
\[
\coTnabla = (\diff \otimes \id - (\wedge \otimes \id) \circ (\id \otimes \nabla)) g.
\]

\begin{proposition}
\label{prop:cotorsion}
We have $\coTnabla = 0$, that is $\nabla$ is cotorsion free.
\end{proposition}

\begin{proof}
Using $(\diff \otimes \id) g = 0$ from \cref{prop:metric-kahler}, we can write the cotorsion as
\[
\coTnabla = - (\wedge \otimes \id) \circ (\id \otimes \nabla) \met.
\]
First we compute $(\id \otimes \nabla) \met$. We have
\[
\begin{split}
(\id \otimes \nabla) \met & = \EVp_{1 2} \EV_{2 3} (\del p \otimes \nabla(\delbar p) + \delbar p \otimes \nabla(\del p)) \\
& = \EVp_{1 2} \EV_{2 3} (\del p \otimes \EV_{2 3} (\del p \otimes \delbar p) - q^{-(\omega_s, 2 \rho)} \del p \otimes p \metMP) \\
& + q^{(\alpha_s, \alpha_s)} \EVp_{1 2} \EV_{2 3} (\delbar p \otimes \EV_{2 3} \opT_{1 2 3 4} (\delbar p \otimes \del p) - q^{-(\omega_s, 2 \rho)} \delbar p \otimes p \metPM).
\end{split}
\]
It is easy to show that $\EVp_{1 2} \EV_{2 3} \del p \otimes p \metMP = 0$ and $\EVp_{1 2} \EV_{2 3} \delbar p \otimes p \metPM = 0$ by using the relations in \eqref{eq:evaluations-calculi} (recall that the tensor product is over $\subalg$).
Hence we have
\[
(\id \otimes \nabla) \met = \EVp_{1 2} \EV_{2 3} \EV_{4 5} \del p \otimes \del p \otimes \delbar p + q^{(\alpha_s, \alpha_s)} \EVp_{1 2} \EV_{2 3} \EV_{4 5} \opT_{3 4 5 6} \delbar p \otimes \delbar p \otimes \del p.
\]
The first term vanishes using $\EV_{2 3} \del p \otimes \del p = 0$, since it can be rewritten as
\[
\EVp_{1 2} \EV_{2 3} \EV_{2 3} \del p \otimes \del p \otimes \delbar p = 0.
\]
Then applying $- (\wedge \otimes \id)$ to $(\id \otimes \nabla) \met$ we are left with
\[
\coTnabla = - q^{(\alpha_s, \alpha_s)} \EVp_{1 2} \EV_{2 3} \EV_{4 5} \opT_{3 4 5 6} \delbar p \wedge \delbar p \otimes \del p.
\]

We now show that this term vanishes.
First, apply $\delbar$ to $\delbar p p = q^{-(\alpha_s, \alpha_s)} \opT_{1 2 3 4} p \delbar p$ to get the identity $- \delbar p \wedge \delbar p = q^{-(\alpha_s, \alpha_s)} \opT_{1 2 3 4} \delbar p \wedge \delbar p$. Using this we rewrite
\[
\coTnabla = \EVp_{1 2} \EV_{2 3} \EV_{4 5} \opT_{3 4 5 6} \opT_{1 2 3 4} \delbar p \wedge \delbar p \otimes \del p.
\]
Noting that $\EV_{2 3} \EV_{4 5} = \EV_{2 3} \EV_{2 3}$ allows us to use the identity $\EV_{2 3} \opT_{3 4 5 6} \opT_{1 2 3 4} = \opT_{1 2 3 4} \EV_{4 5}$ from \eqref{eq:identity-ETT}. Then the cotorsion takes the form
\[
\coTnabla = \EVp_{1 2} \EV_{2 3} \opT_{1 2 3 4} \EV_{4 5} \delbar p \wedge \delbar p \otimes \del p.
\]
Next, using $\EVp_{1 2} \EV_{2 3} \opT_{1 2 3 4} = \EVp_{1 2} \EV_{2 3}$ from \eqref{eq:identity-EpET} we have
\[
\coTnabla
= \EVp_{1 2} \EV_{2 3} \EV_{4 5} \delbar p \wedge \delbar p \otimes \del p
= \EVp_{1 2} \EV_{2 3} \EV_{2 3} \delbar p \wedge \delbar p \otimes \del p.
\]
Finally this expression vanishes, since $\EV_{2 3} \delbar p \wedge \delbar p = 0$.
\end{proof}

\subsection{Levi-Civita connection}

Summarizing the results obtained so far, we obtain the following theorem, which is the main result of this section.

\begin{theorem}
\label{thm:levi-civita}
The connection $\nabla: \calc \to \calc \otimes_\subalg \calc$ is a weak quantum Levi-Civita connection with respect to the quantum metric $\met$.
\end{theorem}

\begin{proof}
According to \cref{def:levi-civita}, this means that $\nabla$ is torsion free and cotorsion free (the latter involves $\met$).
This is what we have proven in \cref{prop:torsion} and \cref{prop:cotorsion}.
\end{proof}

In view of these properties, and the fact that it reduces to the Levi-Civita connection on the cotangent bundle in the classical limit, it seems appropriate to consider $\nabla$ as a quantum analogue of the Levi-Civita connection for the quantum projective spaces.

Having the quantum metric $\met$ and the weak Levi-Civita connection $\nabla$, one can investigate further aspects of quantum Riemannian geometry in the sense of \cite{quantum-book}.
These include the computation of the Riemann tensor and an appropriately defined Ricci tensor, for instance.
For the case of the quantum 2-sphere, which corresponds to the easiest case of a quantum projective space, such computations have been performed in \cite{majid-sphere}.
An important result is that an analogue of the Einstein condition holds, that is the Ricci tensor is proportional to the quantum metric.
We conjecture that this will also be the case for general quantum projective spaces, and we plan to investigate this aspect in future research.

We can also ask for the stronger version of the property of compatibility with the metric, as opposed to the cotorsion free condition.
We investigate this in the next section.

\section{Bimodule connections and metric compatibility}
\label{sec:bimodule-connections}

In this section we show that the connections $\connP$ and $\connM$ are bimodule connections.
Then we use this fact to show that the connection $\conn$ satisfies the quantum metric compatibility $\conn \met = 0$, which means that $\nabla$ is a quantum Levi-Civita connection.

\subsection{Bimodule connections}

We investigate whether $\conn: \calc \to \calc \otimes_\subalg \calc$ is a bimodule connection by studying its components $\connP$ and $\connM$.
Recall from \cref{def:bimodule-connection} that this requires the existence of a $\subalg$-bimodule map  $\sigma: \calc \otimes_\subalg \calc \to \calc \otimes_\subalg \calc$ such that
\[
\nabla(\omega b) = \sigma(\omega \otimes \diff b) + \nabla(\omega) b, \quad \omega \in \calc, \ b \in \subalg.
\]

We begin by obtaining a useful expression for $\connM(\delbar p p)$.

\begin{lemma}
\label{lem:connM-bimodule}
We have $\connM(\delbar p p) = \sigmaMM + \sigmaMP + \connM(\delbar p) p$, where
\[
\begin{split}
\sigmaMM & = q^{-(\alpha_s, \alpha_s)} \opT_{1 2 3 4} \delbar p \otimes \delbar p \\
\sigmaMP & = q^{2 (\omega_s, \omega_s) - 2 (\alpha_s, \alpha_s)} \opS_{1 2 3}^{-1} \opSt_{2 3 4} \del p \otimes \delbar p - (q^{-(\alpha_s, \alpha_s)} - 1) q^{-(\omega_s, 2 \rho)} p \metPM p.
\end{split}
\]
\end{lemma}

\begin{proof}
We have $\connM(\delbar p p) = q^{-(\alpha_s, \alpha_s)} \opT_{1 2 3 4} \connM(p \delbar p)$ by \eqref{eq:right-module}.
Then we compute
\[
\begin{split}
\connM(\delbar p p)
& = q^{-(\alpha_s, \alpha_s)} \opT_{1 2 3 4} \diff p \otimes \delbar p + q^{-(\alpha_s, \alpha_s)} \opT_{1 2 3 4} p \connM(\delbar p) \\
& = q^{-(\alpha_s, \alpha_s)} \opT_{1 2 3 4} \delbar p \otimes \delbar p + q^{-(\alpha_s, \alpha_s)} \opT_{1 2 3 4} \del p \otimes \delbar p \\
& + q^{-(\alpha_s, \alpha_s)} \opT_{1 2 3 4} \EV_{4 5} p \del p \otimes \delbar p - q^{-(\alpha_s, \alpha_s)} q^{-(\omega_s, 2 \rho)} \opT_{1 2 3 4} p p \metPM.
\end{split}
\]
We have $\opT_{1 2 3 4} \EV_{4 5} = \EV_{2 3} \opT_{3 4 5 6} \opT_{1 2 3 4}$ by \eqref{eq:identity-ETT}. Then
\[
\opT_{1 2 3 4} \EV_{4 5} p \del p \otimes \delbar p = \EV_{4 5} \opT_{3 4 5 6} \opT_{1 2 3 4} p \del p \otimes \delbar p = \EV_{2 3} \del p \otimes \delbar p p.
\]
Also using the analogue of \cref{lem:P-relations} for $p p$ we have
\[
\opT_{1 2 3 4} p p = \opS_{1 2 3} \opSt_{2 3 4} p p = q^{-(\omega_s, \omega_s)} \opS_{1 2 3} p p = p p.
\]
Using these relations and $p \metPM = \metPM p$ we obtain
\[
\begin{split}
\connM(\delbar p p)
& = q^{-(\alpha_s, \alpha_s)} \opT_{1 2 3 4} \delbar p \otimes \delbar p + q^{-(\alpha_s, \alpha_s)} \opT_{1 2 3 4} \del p \otimes \delbar p \\
& + q^{-(\alpha_s, \alpha_s)} \EV_{2 3} \del p \otimes \delbar p p - q^{-(\alpha_s, \alpha_s)} q^{-(\omega_s, 2 \rho)} p \metPM p.
\end{split}
\]
The second line coincides with $q^{-(\alpha_s, \alpha_s)} \connM(\del p) p$. Then we can write
\[
\connM(\delbar p p) = \sigmaMM + \sigmaMP + \connM(\delbar p) p,
\]
where we define $\sigmaMM = q^{-(\alpha_s, \alpha_s)} \opT_{1 2 3 4} \delbar p \otimes \delbar p$ and
\[
\begin{split}
\sigmaMP & = q^{-(\alpha_s, \alpha_s)} \opT_{1 2 3 4} \del p \otimes \delbar p + (q^{-(\alpha_s, \alpha_s)} - 1) \EV_{2 3} \del p \otimes \delbar p p \\
& - (q^{-(\alpha_s, \alpha_s)} - 1) q^{-(\omega_s, 2 \rho)} p \metPM p.
\end{split}
\]

Now we rewrite $\sigmaMP$ in a more convenient form.
In the quadratic case we have $\opS_{1 2 3} = q^{2 (\omega_s, \omega_s) - (\alpha_s, \alpha_s)} \opS_{1 2 3}^{-1} + q^{(\omega_s, \omega_s)} (1 - q^{-(\alpha_s, \alpha_s)})$ from \eqref{eq:S-quadratic1}. Then we get
\[
\opT_{1 2 3 4} \del p \otimes \delbar p
= q^{2 (\omega_s, \omega_s) - (\alpha_s, \alpha_s)} \opS_{1 2 3}^{-1} \opSt_{2 3 4} \del p \otimes \delbar p + q^{(\omega_s, \omega_s)} (1 - q^{-(\alpha_s, \alpha_s)}) \opSt_{2 3 4} \del p \otimes \delbar p.
\]
Using \eqref{eq:evaluations-calculi} and \eqref{eq:S-right-module} we can rewrite
\[
\begin{split}
\opSt_{2 3 4} \del p \otimes \delbar p
& = \opSt_{2 3 4} \EV_{2 3} \del p p \otimes \delbar p
= \EV_{2 3} \opSt_{4 5 6} \del p \otimes p \delbar p \\
& = q^{(\alpha_s, \alpha_s) - (\omega_s, \omega_s)} \EV_{2 3} \del p \otimes \delbar p p.
\end{split}
\]
Hence we get
\[
\opT_{1 2 3 4} \del p \otimes \delbar p
= q^{2 (\omega_s, \omega_s) - (\alpha_s, \alpha_s)} \opS_{1 2 3}^{-1} \opSt_{2 3 4} \del p \otimes \delbar p + (1 - q^{-(\alpha_s, \alpha_s)}) q^{(\alpha_s, \alpha_s)} \EV_{2 3} \del p \otimes \delbar p p.
\]
Finally plugging back into $\sigmaMP$ gives the expression
\[
\sigmaMP = q^{2 (\omega_s, \omega_s) - 2 (\alpha_s, \alpha_s)} \opS_{1 2 3}^{-1} \opSt_{2 3 4} \del p \otimes \delbar p - (q^{-(\alpha_s, \alpha_s)} - 1) q^{-(\omega_s, 2 \rho)} p p \metPM.
\]
This gives the result as claimed.
\end{proof}

We proceed in the same way for $\connP$, which is a bit more involved.

\begin{lemma}
\label{lem:connP-bimodule}
We have $\connP(\del p p) = \sigmaPP + \sigmaPM + \connP(\del p) p$, where
\[
\begin{split}
\sigmaPP & = q^{(\alpha_s, \alpha_s)} \opT_{1 2 3 4} \del p \otimes \del p, \\
\sigmaPM & = q^{2 (\alpha_s, \alpha_s) - 2 (\omega_s, \omega_s)} \opS_{1 2 3} \opSt_{2 3 4}^{-1} \delbar p \otimes \del p - (q^{(\alpha_s, \alpha_s)} - 1) q^{(\alpha_s, \alpha_s)} q^{-(\omega_s, 2 \rho)} p \metMP p.
\end{split}
\]
\end{lemma}

\begin{proof}
We have $\connP(\del p p) = q^{(\alpha_s, \alpha_s)} \opT_{1 2 3 4} \connP(p \del p)$ by \eqref{eq:right-module}.
Then we compute
\[
\begin{split}
\connP(\del p p)
& = q^{(\alpha_s, \alpha_s)} \opT_{1 2 3 4} \diff p \otimes \del p + q^{(\alpha_s, \alpha_s)} \opT_{1 2 3 4} p \connP(\del p) \\
& = q^{(\alpha_s, \alpha_s)} \opT_{1 2 3 4} \del p \otimes \del p + q^{(\alpha_s, \alpha_s)} \opT_{1 2 3 4} \delbar p \otimes \del p \\
& + q^{2 (\alpha_s, \alpha_s)} \opT_{1 2 3 4} \EV_{4 5} \opT_{3 4 5 6} p \delbar p \otimes \del p - q^{2 (\alpha_s, \alpha_s)} q^{-(\omega_s, 2 \rho)} \opT_{1 2 3 4} p p \metMP.
\end{split}
\]
We have $\opT_{1 2 3 4} \EV_{4 5} = \EV_{2 3} \opT_{3 4 5 6} \opT_{1 2 3 4}$ from \eqref{eq:identity-ETT}. Then consider
\[
\opT_{1 2 3 4} \EV_{4 5} \opT_{3 4 5 6} p \delbar p \otimes \del p = \EV_{2 3} \opT_{3 4 5 6} \opT_{1 2 3 4} \opT_{3 4 5 6} p \delbar p \otimes \del p.
\]
Using \eqref{eq:S-braid-equation} we can derive an analogue of the braid equation for $\opT$, that is
\[
\opT_{1 2 3 4} \opT_{3 4 5 6} \opT_{1 2 3 4} = \opT_{3 4 5 6} \opT_{1 2 3 4} \opT_{3 4 5 6}.
\]
Using this identity we obtain
\[
\begin{split}
\opT_{1 2 3 4} \EV_{4 5} \opT_{3 4 5 6} p \delbar p \otimes \del p
& = \EV_{2 3} \opT_{1 2 3 4} \opT_{3 4 5 6} \opT_{1 2 3 4} p \delbar p \otimes \del p \\
& = \EV_{2 3} \opT_{1 2 3 4} \delbar p \otimes \del p p.
\end{split}
\]
Since we have $\opT_{1 2 3 4} p p \metMP = p p \metMP = p \metMP p$, as in the proof of \cref{lem:connM-bimodule}, we get
\[
\begin{split}
\connP(\del p p)
& = q^{(\alpha_s, \alpha_s)} \opT_{1 2 3 4} \del p \otimes \del p + q^{(\alpha_s, \alpha_s)} \opT_{1 2 3 4} \delbar p \otimes \del p \\
& + q^{2 (\alpha_s, \alpha_s)} \EV_{2 3} \opT_{1 2 3 4} \delbar p \otimes \del p p - q^{2 (\alpha_s, \alpha_s)} q^{-(\omega_s, 2 \rho)} p \metMP p.
\end{split}
\]
The second line coincides with $q^{(\alpha_s, \alpha_s)} \connP(\del p) p$. Then we can write
\[
\connP(\del p p) = \sigmaPP + \sigmaPM + \connP(\del p) p,
\]
where $\sigmaPP = q^{(\alpha_s, \alpha_s)} \opT_{1 2 3 4} \del p \otimes \del p$ and
\[
\begin{split}
\sigmaPM & = q^{(\alpha_s, \alpha_s)} \opT_{1 2 3 4} \delbar p \otimes \del p + (q^{(\alpha_s, \alpha_s)} - 1) q^{(\alpha_s, \alpha_s)} \EV_{2 3} \opT_{1 2 3 4} \delbar p \otimes \del p p \\
& - (q^{(\alpha_s, \alpha_s)} - 1) q^{(\alpha_s, \alpha_s)} q^{-(\omega_s, 2 \rho)} p \metMP p.
\end{split}
\]

We now rewrite $\sigmaPM$ in a more convenient form.
In the quadratic case we have $\opSt_{2 3 4} = q^{(\alpha_s, \alpha_s) - 2 (\omega_s, \omega_s)} \opSt_{2 3 4}^{-1} + q^{-(\omega_s, \omega_s)} (1 - q^{(\alpha_s, \alpha_s)})$ from \eqref{eq:S-quadratic2}. Then
\[
\opT_{1 2 3 4} \delbar p \otimes \del p = q^{(\alpha_s, \alpha_s) - 2 (\omega_s, \omega_s)} \opS_{1 2 3} \opSt_{2 3 4}^{-1} \delbar p \otimes \del p + q^{-(\omega_s, \omega_s)} (1 - q^{(\alpha_s, \alpha_s)}) \opS_{1 2 3} \delbar p \otimes \del p.
\]
Consider the term $\opS_{1 2 3} \delbar p \otimes \del p$.
Using \eqref{eq:evaluations-calculi} and \eqref{eq:S-right-module} we get
\[
\opS_{1 2 3} \delbar p \otimes \del p = \EV_{4 5} \opS_{1 2 3} \delbar p \otimes \del p p = q^{(\alpha_s, \alpha_s) - (\omega_s, \omega_s)} \EV_{4 5} \opS_{1 2 3} \opS_{3 4 5} \delbar p p \otimes \del p.
\]
Next, we use $\EV_{4 5} = \EV_{2 3} \opSt_{2 3 4} \opS_{3 4 5}$ from \eqref{eq:S-evaluation2} and the "braid equation" for $\opS$ \eqref{eq:S-braid-equation}. Then
\[
\begin{split}
\opS_{1 2 3} \delbar p \otimes \del p
& = q^{(\alpha_s, \alpha_s) - (\omega_s, \omega_s)} \EV_{2 3} \opSt_{2 3 4} \opS_{3 4 5} \opS_{1 2 3} \opS_{3 4 5} \delbar p p \otimes \del p \\
& = q^{(\alpha_s, \alpha_s) - (\omega_s, \omega_s)} \EV_{2 3} \opSt_{2 3 4} \opS_{1 2 3} \opS_{3 4 5} \opS_{1 2 3} \delbar p p \otimes \del p.
\end{split}
\]
Now we can use \eqref{eq:S-action-right} and \eqref{eq:S-right-module} to get
\[
\begin{split}
\opS_{1 2 3} \delbar p \otimes \del p
& = q^{(\alpha_s, \alpha_s)} \EV_{2 3} \opSt_{2 3 4} \opS_{1 2 3} \opS_{3 4 5} \delbar p p \otimes \del p
= q^{(\omega_s, \omega_s)} \EV_{2 3} \opSt_{2 3 4} \opS_{1 2 3} \delbar p \otimes \del p p \\
& = q^{(\omega_s, \omega_s)} \EV_{2 3} \opT_{1 2 3 4} \delbar p \otimes \del p p.
\end{split}
\]
Therefore we obtain
\[
\opT_{1 2 3 4} \delbar p \otimes \del p = q^{(\alpha_s, \alpha_s) - 2 (\omega_s, \omega_s)} \opS_{1 2 3} \opSt_{2 3 4}^{-1} \delbar p \otimes \del p + (1 - q^{(\alpha_s, \alpha_s)}) \EV_{2 3} \opT_{1 2 3 4} \delbar p \otimes \del p p.
\]
Finally plugging this into $\sigmaPM$ gives the expression
\[
\sigmaPM = q^{2 (\alpha_s, \alpha_s) - 2 (\omega_s, \omega_s)} \opS_{1 2 3} \opSt_{2 3 4}^{-1} \delbar p \otimes \del p - (q^{(\alpha_s, \alpha_s)} - 1) q^{(\alpha_s, \alpha_s)} q^{-(\omega_s, 2 \rho)} p \metMP p.
\]
This gives the result as claimed.
\end{proof}

The computations above suggest that the terms $\sigma_{a b}$ with $a, b \in \{ +, - \}$ might correspond to $\subalg$-bimodule maps $\calc_a \otimes_\subalg \calc_b \to \calc_b \otimes_\subalg \calc_a$.
This is indeed the case, as we verify by lengthy computations in \cref{sec:bimodule-maps}. Then we obtain the following result.

\begin{proposition}
\label{prop:bimodule-connections}
We have that $\connP$ and $\connM$ are bimodule connections.
\end{proposition}

\begin{proof}
From \cref{prop:sigmaPP}, \cref{prop:sigmaMM}, \cref{prop:sigmaPM} and \cref{prop:sigmaMP} we have four $\subalg$-bimodule maps $\calc_a \otimes_\subalg \calc_b \to \calc_b \otimes_\subalg \calc_a$ with appropriate $a, b \in \{ +, - \}$ given by
\[
\begin{split}
\sigmaPP(\del p \otimes \del p) & = q^{(\alpha_s, \alpha_s)} \opT_{1 2 3 4} \del p \otimes \del p, \\
\sigmaMM(\delbar p \otimes \delbar p) & = q^{-(\alpha_s, \alpha_s)} \opT_{1 2 3 4} \delbar p \otimes \delbar p, \\
\sigmaPM(\del p \otimes \delbar p) & = q^{2 (\alpha_s, \alpha_s) - 2 (\omega_s, \omega_s)} \opS_{1 2 3} \opSt_{2 3 4}^{-1} \delbar p \otimes \del p - (q^{(\alpha_s, \alpha_s)} - 1) q^{(\alpha_s, \alpha_s)} q^{-(\omega_s, 2 \rho)} p \metMP p, \\
\sigmaMP(\delbar p \otimes \del p) & = q^{2 (\omega_s, \omega_s) - 2 (\alpha_s, \alpha_s)} \opS_{1 2 3}^{-1} \opSt_{2 3 4} \del p \otimes \delbar p - (q^{-(\alpha_s, \alpha_s)} - 1) q^{-(\omega_s, 2 \rho)} p \metPM p.
\end{split}
\]
They can be assembled into a $\subalg$-bimodule map $\sigma: \calc \otimes_\subalg \calc \to \calc \otimes_\subalg \calc$.
Then using the expressions from \cref{lem:connM-bimodule} and \cref{lem:connP-bimodule} we observe that
\[
\begin{split}
\connM(\delbar p p) & = \sigma(\delbar p \otimes \diff p) + \connM(\delbar p) p, \\
\connP(\del p p) & = \sigma(\del p \otimes \diff p) + \connP(\del p) p,
\end{split}
\]
which shows that $\connM$ and $\connP$ are bimodule connection with generalized braiding $\sigma$.
\end{proof}

\subsection{Metric compatibility}

We now investigate whether the connection $\conn$ is quantum metric compatible, in the sense that $\conn \met = 0$.
Here the action on $\conn$ on $\met$ is given by
\[
\conn g = (\conn \otimes \id) \met + (\sigma \otimes \id) (\id \otimes \conn) \met.
\]

\begin{theorem}
\label{thm:strong-levi-civita}
We have $\conn \met = 0$. Hence $\conn$ is a quantum Levi-Civita connection.
\end{theorem}

\begin{proof}
We compute separately the action of $\conn$ on the two legs of $\metPM$ and $\metMP$.

$\bullet$ \textbf{The term $(\conn \otimes \id) \metMP$}.
We have
\[
\begin{split}
(\conn \otimes \id) \metMP
& = \EVp_{1 2} \EV_{2 3} \conn(\delbar p) \otimes \del p \\
& = \EVp_{1 2} \EV_{2 3} \EV_{2 3} \del p \otimes \delbar p \otimes \del p - q^{-(\omega_s, 2 \rho)} \EVp_{1 2} \EV_{2 3} p \metPM \otimes \del p.
\end{split}
\]
The second term vanishes, since $p \metPM = \metPM p$ and $\EV_{2 3} p \del p = 0$.
The first term vanishes due to \cref{lem:vanishing-third}, which is related to the Kähler property of the metric.

$\bullet$ \textbf{The term $(\conn \otimes \id) \metPM$}.
We have
\[
\begin{split}
(\conn \otimes \id) \metPM
& = \EVp_{1 2} \EV_{2 3} \conn(\del p) \otimes \delbar p \\
& = q^{(\alpha_s, \alpha_s)} \EVp_{1 2} \EV_{2 3} \EV_{2 3} \opT_{1 2 3 4} \delbar p \otimes \del p \otimes \delbar p - q^{(\alpha_s, \alpha_s)} q^{-(\omega_s, 2 \rho)} \EVp_{1 2} \EV_{2 3} p \metMP \otimes \delbar p.
\end{split}
\]
The second term vanishes due to $\EVp_{1 2} \EV_{2 3} p \delbar p = \EVp_{1 2} \delbar p = 0$.

Now write $A = \EV_{2 3} \EV_{2 3} \opT_{1 2 3 4} \delbar p \otimes \del p \otimes \delbar p$. Then
\[
\begin{split}
A & = \EV_{2 3} \EV_{2 3} \opT_{1 2 3 4} \EV_{6 7} \delbar p \otimes \del p \otimes p \delbar p \\
& = \EV_{2 3} \EV_{2 3} \EV_{6 7} \opSt_{2 3 4} \opS_{1 2 3} \delbar p \otimes \del p \otimes p \delbar p \\
& = \EV_{2 3} \EV_{4 5} \EV_{2 3} \opSt_{2 3 4} \opS_{1 2 3} \delbar p \otimes \del p p \otimes \delbar p.
\end{split}
\]
We proceed by using \eqref{eq:S-action-right} and \eqref{eq:S-right-module}. We obtain
\[
\begin{split}
A & = q^{(\alpha_s, \alpha_s) - (\omega_s, \omega_s)} \EV_{2 3} \EV_{4 5} \EV_{2 3} \opSt_{2 3 4} \opS_{1 2 3} \opS_{3 4 5} \delbar p p \otimes \del p \otimes \delbar p \\
& = q^{(\alpha_s, \alpha_s) - 2 (\omega_s, \omega_s)} \EV_{2 3} \EV_{4 5} \EV_{2 3} \opSt_{2 3 4} \opS_{1 2 3} \opS_{3 4 5} \opS_{1 2 3} \delbar p p \otimes \del p \otimes \delbar p.
\end{split}
\]
Now we use the "braid equation" $\opS_{1 2 3} \opS_{3 4 5} \opS_{1 2 3} = \opS_{3 4 5} \opS_{1 2 3} \opS_{3 4 5}$ from \eqref{eq:S-braid-equation} and the identity $\EV_{2 3} \opSt_{2 3 4} \opS_{3 4 5} = \EV_{4 5}$ from \eqref{eq:S-evaluation2}. We get
\[
\begin{split}
A & = q^{(\alpha_s, \alpha_s) - 2 (\omega_s, \omega_s)} \EV_{2 3} \EV_{4 5} \EV_{2 3} \opSt_{2 3 4} \opS_{3 4 5} \opS_{1 2 3} \opS_{3 4 5} \delbar p p \otimes \del p \otimes \delbar p \\
& = q^{(\alpha_s, \alpha_s) - 2 (\omega_s, \omega_s)} \EV_{2 3} \EV_{4 5} \EV_{4 5} \opS_{1 2 3} \opS_{3 4 5} \delbar p p \otimes \del p \otimes \delbar p \\
& = q^{(\alpha_s, \alpha_s) - 2 (\omega_s, \omega_s)} \EV_{2 3} \EV_{2 3} \opS_{1 2 3} \EV_{4 5} \opS_{3 4 5} \delbar p p \otimes \del p \otimes \delbar p.
\end{split}
\]

Next consider the identity $\opS_{1 2 3} = q^{2 (\omega_s, \omega_s) - (\alpha_s, \alpha_s)} \opS_{1 2 3}^{-1} + q^{(\omega_s, \omega_s)} (1 - q^{-(\alpha_s, \alpha_s)})$ from \eqref{eq:S-quadratic1}, valid in the quadratic case.
Using $\EV_{2 3} \opS_{1 2 3}^{-1} = q^{-(\omega_s, \omega_s + 2 \rho)} \EVp_{1 2}$ from \eqref{eq:S-evaluation} we get
\begin{equation}
\label{eq:identity-compatibility}
\EV_{2 3} \opS_{1 2 3} = q^{-(\omega_s, 2 \rho)} q^{(\omega_s, \omega_s) - (\alpha_s, \alpha_s)} \EVp_{1 2} + q^{(\omega_s, \omega_s)} (1 - q^{-(\alpha_s, \alpha_s)}) \EV_{2 3}.
\end{equation}
Plugging this into the previous identity for $A$ we get
\[
\begin{split}
A & = q^{-(\omega_s, 2 \rho)} q^{-(\omega_s, \omega_s)} \EV_{2 3} \EV_{2 3} \opS_{1 2 3} \EVp_{3 4} \delbar p p \otimes \del p \otimes \delbar p \\
& + (1 - q^{-(\alpha_s, \alpha_s)}) q^{(\alpha_s, \alpha_s) - (\omega_s, \omega_s)} \EV_{2 3} \EV_{2 3} \opS_{1 2 3} \EV_{4 5} \delbar p \otimes p \del p \otimes \delbar p \\
& = q^{-(\omega_s, \omega_s)} \EV_{2 3} \EV_{2 3} \opS_{1 2 3} \delbar p \otimes \del p \otimes \delbar p.
\end{split}
\]
Above we used $\EVp_{1 2} p = q^{(\omega_s, 2 \rho)}$ and $\EV_{2 3} p \del p = 0$.
Using \eqref{eq:identity-compatibility} once more we get
\[
\begin{split}
A & = q^{-(\omega_s, 2 \rho)} q^{- (\alpha_s, \alpha_s)} \EV_{2 3} \EVp_{1 2} \delbar p \otimes \del p \otimes \delbar p + (1 - q^{-(\alpha_s, \alpha_s)}) \EV_{2 3} \EV_{2 3} \delbar p \otimes \del p \otimes \delbar p \\
& = (1 - q^{-(\alpha_s, \alpha_s)}) \EV_{2 3} \EV_{2 3} \delbar p \otimes \del p \otimes \delbar p.
\end{split}
\]

Plugging this expression into $(\conn \otimes \id) \metPM$ we obtain
\[
(\conn \otimes \id) \metPM = (1 - q^{-(\alpha_s, \alpha_s)}) q^{(\alpha_s, \alpha_s)} \EVp_{1 2} \EV_{2 3} \EV_{2 3} \delbar p \otimes \del p \otimes \delbar p.
\]
But then this term vanishes due to \cref{lem:vanishing-third}.

$\bullet$ \textbf{The term $(\id \otimes \conn) \metPM$}.
We have
\[
\begin{split}
(\id \otimes \conn) \metPM
& = \EVp_{1 2} \EV_{2 3} \del p \otimes \conn(\delbar p) \\
& = \EVp_{1 2} \EV_{2 3} \EV_{4 5} \del p \otimes \del p \otimes \delbar p - q^{-(\omega_s, 2 \rho)} \EVp_{1 2} \EV_{2 3} \del p \otimes p \metPM.
\end{split}
\]
The first term vanishes by $\EV_{2 3} \del p \otimes \del p = 0$ and the second term by $\EVp_{1 2} \EV_{2 3} \del p p = 0$.

$\bullet$ \textbf{The term $(\id \otimes \conn) \metMP$}.
We have
\[
\begin{split}
(\id \otimes \conn) \metMP
& = \EVp_{1 2} \EV_{2 3} \delbar p \otimes \conn(\del p) \\
& = q^{(\alpha_s, \alpha_s)} \EVp_{1 2} \EV_{2 3} \EV_{4 5} \opT_{3 4 5 6} \delbar p \otimes \delbar p \otimes \del p - q^{(\alpha_s, \alpha_s)} q^{-(\omega_s, 2 \rho)} \EVp_{1 2} \EV_{2 3} \delbar p \otimes p \metMP.
\end{split}
\]
The second term vanishes due to $\EV_{2 3} \delbar p p = 0$.
Now we apply $\sigma \otimes \id$ to this expression.
Since $\sigma(\delbar p \otimes \delbar p) = q^{-(\alpha_s, \alpha_s)} \opT_{1 2 3 4} \delbar p \otimes \delbar p$ we get
\[
\begin{split}
(\sigma \otimes \id) (\id \otimes \conn) \metMP
& = q^{(\alpha_s, \alpha_s)} \EVp_{1 2} \EV_{2 3} \EV_{4 5} \opT_{3 4 5 6} \sigma(\delbar p \otimes \delbar p) \otimes \del p \\
& = \EVp_{1 2} \EV_{2 3} \EV_{4 5} \opT_{3 4 5 6} \opT_{1 2 3 4} \delbar p \otimes \delbar p \otimes \del p.
\end{split}
\]
Next we use $\EV_{2 3} \opT_{3 4 5 6} \opT_{1 2 3 4} = \opT_{1 2 3 4} \EV_{4 5}$ from \eqref{eq:identity-ETT}. We obtain
\[
\begin{split}
(\sigma \otimes \id) (\id \otimes \conn) \metMP
& = \EVp_{1 2} \EV_{2 3} \EV_{2 3} \opT_{3 4 5 6} \opT_{1 2 3 4} \delbar p \otimes \delbar p \otimes \del p \\
& = \EVp_{1 2} \EV_{2 3} \opT_{1 2 3 4} \EV_{4 5} \delbar p \otimes \delbar p \otimes \del p.
\end{split}
\]
Using the identity $\EVp_{1 2} \EV_{2 3} \opT_{1 2 3 4} = \EVp_{1 2} \EV_{2 3}$ from \eqref{lem:identity-EpET} we get
\[
\begin{split}
(\sigma \otimes \id) (\id \otimes \conn) \metMP
& = \EVp_{1 2} \EV_{2 3} \EV_{4 5} \delbar p \otimes \delbar p \otimes \del p \\
& = \EVp_{1 2} \EV_{2 3} \EV_{2 3} \delbar p \otimes \delbar p \otimes \del p = 0.
\end{split}
\]
In the last step we have used $\EV_{2 3} \delbar p \otimes \delbar p = 0$.
\end{proof}

The quantum metric compatibility of $\conn$ shows that this connection has all the properties one would expect from the Levi-Civita connection on quantum projective spaces.

\appendix

\section{Results on projective spaces}
\label{sec:projective-spaces}

In this appendix we recall some results on (classical) projective spaces, to facilitate the comparison between the classical and the quantum descriptions.

From the point of view of differential geometry, the complex projective space $\bbC P^N$ can be identified with $\bbC^{N + 1} \backslash \{0\}$ modulo the relation $(Z^1, \cdots, Z^{N + 1}) \sim \lambda (Z^1, \cdots, Z^{N + 1})$ with $\lambda \neq 0$. Here $\{Z^1, \cdots, Z^{N + 1}\}$ are the global coordinates of $\bbC^{N + 1}$.
Define the functions
\[
p^{i j} = \frac{Z^i \bar{Z}^j}{\|Z\|^2}, \quad
i, j = 1, \cdots, N + 1.
\]
These descend to $\bbC P^N$, as they are invariant under the equivalence relation.

Consider the coordinate patch with $Z^{N + 1} \neq 0$ and denote by $z^i = Z^i / Z^{N + 1}$ the corresponding homogeneous coordinates (the discussion is similar for the other patches).
Then with respect to these local coordinates the \emph{Fubini-Study metric} takes the form
\[
g = \sum_{i, j = 1}^N g_{i \bar{j}} \diff z^i \odot \diff \bar{z}^j
= \sum_{i, j = 1}^N \frac{(1 + \|z\|^2) \delta_{i j} - \bar{z}^i z^j}{(1 + \|z\|^2)^2} \diff z^i \odot \diff \bar{z}^j,
\]
where we write $\diff z^i \odot \diff \bar{z}^j = \diff z^i \otimes \diff \bar{z}^j + \diff \bar{z}^j \otimes \diff z^i$ for the symmetric product.
The inverse metric can be seen to have components $g^{\bar{i} j} = (\delta_{i j} + \bar{z}^i z^j) (1 + \|z\|^2)$.

The metric can be rewritten in terms of the functions $p^{i j}$, which can be seen as the entries of a projection of rank one.
An explicit computation shows that
\begin{equation}
\label{eq:classical-metric}
g = \sum_{i, j = 1}^{N + 1} (\del p^{i j} \otimes \delbar p^{j i} + \delbar p^{i j} \otimes \del p^{j i}).
\end{equation}
Similarly, the inverse metric can be seen as a map $(\cdot, \cdot)$ on the cotangent bundle satisfying
\begin{equation}
\label{eq:inverse-metric}
\begin{gathered}
(\del p^{i j}, \delbar p^{k l}) = \delta_{i l} p^{k j} - p^{i j} p^{k l}, \quad
(\delbar p^{i j}, \del p^{k l}) = \delta_{k j} p^{i l} - p^{i j} p^{k l}, \\
(\del p^{i j}, \del p^{k l}) = 0, \quad
(\delbar p^{i j}, \delbar p^{k l}) = 0.
\end{gathered}
\end{equation}

Next, we describe the \emph{Levi-Civita connection} on the cotangent bundle, defined with respect to the Fubini-Study metric.
We have the formulae
\begin{equation}
\label{eq:classical-connection}
\begin{split}
\nabla(\del p^{i j}) & = \sum_{k = 1}^{N + 1} \delbar p^{k j} \otimes \del p^{i k} - p^{i j} g_{- +}, \\
\nabla(\delbar p^{i j}) & = \sum_{k = 1}^{N + 1} \del p^{i k} \otimes \delbar p^{k j} - p^{i j} g_{+ -}.
\end{split}
\end{equation}
Here $g_{+ -} = \sum_{i, j} \del p^{i j} \otimes \delbar p^{j i}$ and $g_{- +} = \sum_{i, j} \delbar p^{i j} \otimes \del p^{j i}$.

The classical analogue of the relations \eqref{eq:calcP-quadratic} and \eqref{eq:calcM-quadratic} are
\[
\begin{gathered}
p^{i j} \del p^{k l} = p^{i l} \del p^{k j}, \quad
p^{i j} \delbar p^{k l} = p^{k j} \delbar p^{i l}, \\
\sum_{i = 1}^{N + 1} \del p^{i i} = 0, \quad
\sum_{i = 1}^{N + 1} \delbar p^{i i} = 0.
\end{gathered}
\]
Using these relations, one can prove directly that the Levi-Civita connection is given by \eqref{eq:classical-connection}.
Indeed, one checks that $\nabla$ is torsion free and metric compatible, that is $\nabla g = 0$.

\section{The maps \texorpdfstring{$\opS$}{S} and \texorpdfstring{$\opSt$}{St}}
\label{sec:properties-S}

In this appendix we prove various properties satisfied by the maps
\[
\begin{split}
\opS_{1 2 3} & = (\braid_{V, V^*})_{2 3} (\braid_{V, V})_{1 2} (\braid^{-1}_{V, V^*})_{2 3}, \\
\opSt_{2 3 4} & = (\braid_{V, V^*})_{2 3} (\braid^{-1}_{V^*, V^*})_{3 4} (\braid^{-1}_{V, V^*})_{2 3}.
\end{split}
\]
These were introduced in \eqref{eq:S-maps} to rewrite some of the relations of the differential calculus.

The most important properties are the commutation relations among them. 

\begin{proposition}
\label{prop:S-properties-proof}
The maps $\opS$ and $\opSt$ satisfy the following properties.

\begin{enumerate}
\item We have the commutation relations
\[
\opS_{1 2 3} \opSt_{2 3 4} = \opSt_{2 3 4} \opS_{1 2 3}, \quad
\opSt_{2 3 4} \opS_{3 4 5} = \opS_{3 4 5} \opSt_{2 3 4}.
\]
\item We have the "braid equations"
\[
\opS_{1 2 3} \opS_{3 4 5} \opS_{1 2 3} = \opS_{3 4 5} \opS_{1 2 3} \opS_{3 4 5}, \quad
\opSt_{2 3 4} \opSt_{4 5 6} \opSt_{2 3 4} = \opSt_{4 5 6} \opSt_{2 3 4} \opSt_{4 5 6}.
\]
\end{enumerate}
\end{proposition}

\begin{proof}
These identities are valid more generally within the braid group with generators $\{ \sigma_i \}_i$.
Under the obvious identifications, we can consider the elements
\[
\opS_{1 2 3} = \sigma_2 \sigma_1 \sigma_2^{-1}, \quad
\opSt_{2 3 4} = \sigma_2 \sigma_3^{-1} \sigma_2^{-1}.
\]

(1) The first identity is straightforward. For the second we compute
\[
\begin{split}
\opSt_{2 3 4} \opS_{3 4 5}
& = (\sigma_2 \sigma_3^{-1} \sigma_2^{-1}) (\sigma_4 \sigma_3 \sigma_4^{-1}) = (\sigma_3^{-1} \sigma_2^{-1} \sigma_3) (\sigma_3^{-1} \sigma_4 \sigma_3) \\
& = \sigma_3^{-1} \sigma_2^{-1} \sigma_4 \sigma_3 = (\sigma_3^{-1} \sigma_4 \sigma_3) (\sigma_3^{-1} \sigma_2^{-1} \sigma_3) \\
& = (\sigma_4 \sigma_3 \sigma_4^{-1}) (\sigma_2 \sigma_3^{-1} \sigma_2^{-1}) = \opS_{3 4 5} \opSt_{2 3 4}.
\end{split}
\]

(2) We consider the first identity. We compute
\[
\begin{split}
\opS_{1 2 3} \opS_{3 4 5} \opS_{1 2 3}
& = (\sigma_2 \sigma_1 \sigma_2^{-1}) (\sigma_4 \sigma_3 \sigma_4^{-1}) (\sigma_2 \sigma_1 \sigma_2^{-1}) = \sigma_4 (\sigma_2 \sigma_1 \sigma_2^{-1}) \sigma_3 (\sigma_2 \sigma_1 \sigma_2^{-1}) \sigma_4^{-1} \\
& = \sigma_4 (\sigma_1^{-1} \sigma_2 \sigma_1) \sigma_3 (\sigma_1^{-1} \sigma_2 \sigma_1) \sigma_4^{-1} = \sigma_4 \sigma_1^{-1} (\sigma_2 \sigma_3 \sigma_2) \sigma_1 \sigma_4^{-1} \\
& = \sigma_4 \sigma_1^{-1} (\sigma_3 \sigma_2 \sigma_3) \sigma_1 \sigma_4^{-1} = \sigma_4 \sigma_3 (\sigma_1^{-1} \sigma_2 \sigma_1) \sigma_3 \sigma_4^{-1} \\
& = \sigma_4 \sigma_3 (\sigma_2 \sigma_1 \sigma_2^{-1}) \sigma_3 \sigma_4^{-1} = (\sigma_4 \sigma_3 \sigma_4^{-1}) (\sigma_2 \sigma_1 \sigma_2^{-1}) (\sigma_4 \sigma_3 \sigma_4^{-1}) \\
& = \opS_{3 4 5} \opS_{1 2 3} \opS_{3 4 5}.
\end{split}
\]
The second identity can be proven in a similar way.
\end{proof}

Now we consider the case when $\braid_{V, V}$ satisfies the quadratic relation
\[
\braidP_{V, V} \braidQ_{V, V} = (\braid_{V, V} - q^{(\omega_s, \omega_s)}) (\braid_{V, V} + q^{(\omega_s, \omega_s) - (\alpha_s, \alpha_s)}) = 0.
\]
Then an analogous relation also holds for the braiding $\braid_{V^*, V^*}$.

\begin{lemma}
\label{lem:S-quadratic}
In the quadratic case we have the relations
\[
\begin{split}
\opS_{1 2 3} & = q^{2 (\omega_s, \omega_s) - (\alpha_s, \alpha_s)} \opS_{1 2 3}^{-1} + q^{(\omega_s, \omega_s)} (1 - q^{- (\alpha_s, \alpha_s)}), \\
\opSt_{2 3 4} & = q^{(\alpha_s, \alpha_s) - 2 (\omega_s, \omega_s)} \opSt_{2 3 4}^{-1} + q^{- (\omega_s, \omega_s)} (1 - q^{(\alpha_s, \alpha_s)}).
\end{split}
\]
\end{lemma}

\begin{proof}
The quadratic relation $\braidP_{V, V} \braidQ_{V, V} = 0$ can be rewritten as
\[
\braid_{V, V} = q^{2 (\omega_s, \omega_s) - (\alpha_s, \alpha_s)} \braid_{V, V}^{-1} + q^{(\omega_s, \omega_s)} (1 - q^{- (\alpha_s, \alpha_s)}).
\]
Using this identity we can rewrite $\opS$ in the following way
\[
\begin{split}
\opS_{1 2 3} & = (\braid_{V, V^*})_{2 3} (\braid_{V, V})_{1 2} (\braid^{-1}_{V, V^*})_{2 3} \\
& = q^{2 (\omega_s, \omega_s) - (\alpha_s, \alpha_s)} (\braid_{V, V^*})_{2 3} (\braid_{V, V}^{-1})_{1 2} (\braid^{-1}_{V, V^*})_{2 3} + q^{(\omega_s, \omega_s)} (1 - q^{- (\alpha_s, \alpha_s)}) (\braid_{V, V^*})_{2 3} (\braid^{-1}_{V, V^*})_{2 3} \\
& = q^{2 (\omega_s, \omega_s) - (\alpha_s, \alpha_s)} \opS_{1 2 3}^{-1} + q^{(\omega_s, \omega_s)} (1 - q^{- (\alpha_s, \alpha_s)}).
\end{split}
\]
The other identity is proven similarly.
\end{proof}

\section{Various identities}
\label{sec:identities}

In this appendix we provide the proofs of various identities that have been used in the main text.
We divide them into two groups, those that only depend on the (rigid) braided monoidal structure, and those that also involve the differential calculus $\calc$.

\subsection{Categorical identities}

The first identity we consider appears in the proof of \cite[Proposition 3.11]{heko} in slightly different terms.
We reprove it here for convenience.

\begin{lemma}
\label{lem:identity-ETT}
We have the identity
\begin{equation}
\label{eq:identity-ETT}
\EV_{2 3} \opT_{3 4 5 6} \opT_{1 2 3 4} = \opT_{1 2 3 4} \EV_{4 5}.
\end{equation}
\end{lemma}

\begin{proof}
Using the definition of $\opT_{1 2 3 4}$ we can write
\[
\opT_{3 4 5 6} \opT_{1 2 3 4} = (\braid_{V, V^*})_{4 5} (\braid_{V, V})_{3 4} (\braid_{V, V^*})_{2 3} (\braid_{V, V})_{1 2} (\braid^{-1}_{V^*, V^*})_{5 6} (\braid^{-1}_{V, V^*})_{4 5} (\braid^{-1}_{V^*, V^*})_{3 4} (\braid^{-1}_{V, V^*})_{2 3}.
\]
We have $\EV_{2 3} (\braid_{V, V^*})_{4 5} = (\braid_{V, V^*})_{2 3} \EV_{2 3}$.
Moreover $\EV_{2 3} (\braid_{V, V})_{3 4} (\braid_{V, V^*})_{2 3} =  \EV_{3 4}$ by naturality of the braiding, as in \eqref{eq:evaluations}.
Then we obtain
\[
\EV_{2 3} (\braid_{V, V^*})_{4 5} (\braid_{V, V})_{3 4} (\braid_{V, V^*})_{2 3} (\braid_{V, V})_{1 2}
= (\braid_{V, V^*})_{2 3} (\braid_{V, V})_{1 2} \EV_{3 4}.
\]
Hence we can write
\[
\EV_{2 3} \opT_{3 4 5 6} \opT_{1 2 3 4}
= (\braid_{V, V^*})_{2 3} (\braid_{V, V})_{1 2} \EV_{3 4} (\braid^{-1}_{V^*, V^*})_{5 6} (\braid^{-1}_{V, V^*})_{4 5} (\braid^{-1}_{V^*, V^*})_{3 4} (\braid^{-1}_{V, V^*})_{2 3}.
\]
As above, using $\EV_{3 4} (\braid^{-1}_{V, V^*})_{4 5} (\braid^{-1}_{V^*, V^*})_{3 4} = \EV_{4 5}$ due to \eqref{eq:evaluations} we obtain
\[
\EV_{3 4} (\braid^{-1}_{V^*, V^*})_{5 6} (\braid^{-1}_{V, V^*})_{4 5} (\braid^{-1}_{V^*, V^*})_{3 4} (\braid^{-1}_{V, V^*})_{2 3}
= (\braid^{-1}_{V^*, V^*})_{3 4} (\braid^{-1}_{V, V^*})_{2 3} \EV_{4 5}.
\]
We conclude that
\[
\EV_{2 3} \opT_{3 4 5 6} \opT_{1 2 3 4}
= (\braid_{V, V^*})_{2 3} (\braid_{V, V})_{1 2} (\braid^{-1}_{V^*, V^*})_{3 4} (\braid^{-1}_{V, V^*})_{2 3} \EV_{4 5} = \opT_{1 2 3 4} \EV_{4 5}. \qedhere
\]
\end{proof}

In the following we suppose that $V$ is a simple module.
First we obtain some relations for the evaluations $\EV$ and $\EVp$ applied to the maps $\opS$ and $\opSt$.

\begin{lemma}
\label{lem:S-evaluation}
Let $V = V(\lambda)$ be a simple module. Then we have
\begin{equation}
\label{eq:S-evaluation}
\EVp_{1 2} \opS_{1 2 3} = q^{(\lambda, \lambda + 2 \rho)} \EV_{2 3}, \quad
\EVp_{3 4} \opSt_{2 3 4}^{-1} = q^{(\lambda, \lambda + 2 \rho)} \EV_{2 3}.
\end{equation}
From these identities we also obtain
\begin{equation}
\label{eq:S-evaluation2}
\EVp_{1 2} \opS_{1 2 3} \opSt_{2 3 4} = \EVp_{3 4}, \quad
\EV_{2 3} \opSt_{2 3 4} \opS_{3 4 5} = \EV_{4 5}.
\end{equation}
\end{lemma}

\begin{proof}
Consider the first identity. Using \eqref{eq:evaluations} we compute
\[
\begin{split}
\EVp_{1 2} \opS_{1 2 3}
& = \EVp_{1 2} (\braid_{V, V^*})_{2 3} (\braid_{V, V})_{1 2} (\braid_{V, V^*}^{-1})_{2 3}
= \EVp_{2 3} (\braid_{V, V}^{-1})_{1 2} (\braid_{V, V})_{1 2} (\braid_{V, V^*}^{-1})_{2 3} \\
& = \EVp_{2 3} (\braid_{V, V^*}^{-1})_{2 3} = q^{(\lambda, \lambda + 2 \rho)} \EV_{2 3}.
\end{split}
\]
In the last step we have used \eqref{eq:identity-EEp}. Similarly, for the second identity we have
\[
\begin{split}
\EVp_{3 4} \opSt_{2 3 4}^{-1}
& = \EVp_{3 4} (\braid_{V, V^*})_{2 3} (\braid_{V^*, V^*})_{3 4} (\braid_{V, V^*}^{-1})_{2 3}
= \EVp_{2 3} (\braid_{V^*, V^*}^{-1})_{3 4} (\braid_{V^*, V^*})_{3 4} (\braid_{V, V^*}^{-1})_{2 3} \\
& = \EVp_{2 3} (\braid_{V, V^*}^{-1})_{2 3} = q^{(\lambda, \lambda + 2 \rho)} \EV_{2 3}.
\end{split}
\]
The other identities easily follow from these.
\end{proof}

The next result, again in the case of a simple module $V$, shows that we can get rid of the term $\opT_{1 2 3 4}$ when performing the evaluations $\EVp_{1 2} \EV_{2 3}$.

\begin{lemma}
\label{lem:identity-EpET}
Let $V = V(\lambda)$ be a simple module. Then we have
\begin{equation}
\label{eq:identity-EpET}
\EVp_{1 2} \EV_{2 3} \opT_{1 2 3 4} = \EVp_{1 2} \EV_{2 3}.
\end{equation}
\end{lemma}

\begin{proof}
We have $\EV_{1 2} (\braid_{V, V^*})_{1 2} = q^{-(\lambda, \lambda + 2 \rho)} \EVp_{1 2}$ from \eqref{eq:identity-EEp}. Then
\[
\EVp_{1 2} \EV_{2 3} \opT_{1 2 3 4} = q^{-(\lambda, \lambda + 2 \rho)} \EVp_{1 2} \EVp_{2 3} (\braid^{-1}_{V^*, V^*})_{3 4} (\braid_{V, V})_{1 2} (\braid^{-1}_{V, V^*})_{2 3}.
\]
By \eqref{eq:evaluations} we have $\EVp_{2 3} (\braid^{-1}_{V^*, V^*})_{3 4} = \EVp_{3 4} (\braid_{V, V^*})_{2 3}$. We obtain
\[
\begin{split}
\EVp_{1 2} \EV_{2 3} \opT_{1 2 3 4} & = q^{-(\lambda, \lambda + 2 \rho)} \EVp_{1 2} \EVp_{3 4} (\braid_{V, V^*})_{2 3} (\braid_{V, V})_{1 2} (\braid^{-1}_{V, V^*})_{2 3} \\
& = q^{-(\lambda, \lambda + 2 \rho)} \EVp_{1 2} \EVp_{1 2} (\braid_{V, V^*})_{2 3} (\braid_{V, V})_{1 2} (\braid^{-1}_{V, V^*})_{2 3}.
\end{split}
\]
Again by \eqref{eq:evaluations} we have $\EVp_{1 2} (\braid_{V, V^*})_{2 3} = \EVp_{2 3} (\braid^{-1}_{V, V})_{1 2}$. Then
\[
\begin{split}
\EVp_{1 2} \EV_{2 3} \opT_{1 2 3 4} & = q^{-(\lambda, \lambda + 2 \rho)} \EVp_{1 2} \EVp_{2 3} (\braid^{-1}_{V, V})_{1 2} (\braid_{V, V})_{1 2} (\braid^{-1}_{V, V^*})_{2 3} \\
& = q^{-(\lambda, \lambda + 2 \rho)} \EVp_{1 2} \EVp_{2 3} (\braid^{-1}_{V, V^*})_{2 3} = \EVp_{1 2} \EV_{2 3}.
\end{split}
\]
In the last step we have used again \eqref{eq:identity-EEp}.
\end{proof}

Finally we need the following identity in the proof of \cref{prop:connection-plus}.

\begin{lemma}
\label{lem:identity-StESt}
Let $V = V(\lambda)$ be a simple module. Then we have
\begin{equation}
\label{eq:identity-StESt}
\opSt_{2 3 4} \EV_{4 5} \opSt_{4 5 6} = \EV_{4 5} \opSt_{4 5 6} \opSt_{2 3 4}.
\end{equation}
\end{lemma}

\begin{proof}
Taking into account that $\opT_{1 2 3 4} = \opSt_{2 3 4} \opS_{1 2 3}$ we write
\[
\opSt_{2 3 4} \EV_{4 5} \opSt_{4 5 6}
= \opSt_{2 3 4} \opS_{1 2 3} \opS_{1 2 3}^{-1} \EV_{4 5} \opSt_{4 5 6}
= \opT_{1 2 3 4} \EV_{4 5} \opS_{1 2 3}^{-1} \opSt_{4 5 6}.
\]
Using $\opT_{1 2 3 4} \EV_{4 5} = \EV_{2 3} \opT_{3 4 5 6} \opT_{1 2 3 4}$ from \eqref{eq:identity-ETT} this becomes
\[
\opSt_{2 3 4} \EV_{4 5} \opSt_{4 5 6}
= \EV_{2 3} \opT_{3 4 5 6} \opT_{1 2 3 4} \opS_{1 2 3}^{-1} \opSt_{4 5 6} 
= \EV_{2 3} \opS_{3 4 5} \opSt_{4 5 6} \opSt_{2 3 4} \opSt_{4 5 6}.
\]
Now we use the "braid equation" for $\opSt$ from \eqref{eq:S-braid-equation}. We get
\[
\opSt_{2 3 4} \EV_{4 5} \opSt_{4 5 6}
= \EV_{2 3} \opS_{3 4 5} \opSt_{2 3 4} \opSt_{4 5 6} \opSt_{2 3 4}.
\]
Finally we have $\EV_{2 3} \opS_{3 4 5} \opSt_{2 3 4} = \EV_{4 5}$ from \eqref{eq:S-evaluation2}, which gives the result.
\end{proof}

\subsection{Differential calculus identities}

We now derive various identities involving some elements of the differential calculus $\calc$.
In the following $V$ always denotes the simple module $V(\omega_s)$.
We begin with some identities involving the metric.

\begin{lemma}
\label{lem:identity-metric}
We have the following identities for the metric $\met$.
\begin{enumerate}
\item For $\metPM$ we have
\begin{equation}
\label{eq:identity-metPM}
\begin{split}
\metPM p & = q^{(\omega_s, \omega_s) - (\alpha_s, \alpha_s)} \EVp_{1 2} \opSt_{2 3 4} \del p \otimes \delbar p \\
& = q^{(\omega_s, \omega_s) - (\alpha_s, \alpha_s)} \EVp_{3 4} \opS_{1 2 3}^{-1} \del p \otimes \delbar p.
\end{split}
\end{equation}
\item For $\metMP$ we have
\begin{equation}
\label{eq:identity-metMP}
p \metMP = q^{(\omega_s, 2 \rho)} \EV_{2 3} \delbar p \otimes \del p.
\end{equation}
\end{enumerate}
\end{lemma}

\begin{proof}
(1) Using the right $\subalg$-module relations \eqref{eq:S-right-module} and \eqref{eq:evaluations-calculi} we compute
\[
\begin{split}
\EV_{2 3} \del p \otimes \delbar p p
& = q^{(\omega_s, \omega_s) - (\alpha_s, \alpha_s)} \EV_{2 3} \opSt_{4 5 6} \del p p \otimes \delbar p \\
& = q^{(\omega_s, \omega_s) - (\alpha_s, \alpha_s)} \opSt_{2 3 4} \EV_{2 3} \del p p \otimes \delbar p \\
& = q^{(\omega_s, \omega_s) - (\alpha_s, \alpha_s)} \opSt_{2 3 4} \del p \otimes \delbar p.
\end{split}
\]
Applying $\EVp_{1 2}$ we get $\metPM p = q^{(\omega_s, \omega_s) - (\alpha_s, \alpha_s)} \EVp_{1 2} \opSt_{2 3 4} \del p \otimes \delbar p$.
The second expression can be obtained from the first one, since using \eqref{eq:S-evaluation} we can rewrite
\[
\EVp_{1 2} \opSt_{2 3 4}
= q^{(\omega_s, \omega_s + 2 \rho)} \EV_{2 3} \opS_{1 2 3}^{-1} \opSt_{2 3 4}
= q^{(\omega_s, \omega_s + 2 \rho)} \EV_{2 3} \opSt_{2 3 4} \opS_{1 2 3}^{-1}
= \EVp_{3 4} \opS_{1 2 3}^{-1}.
\]

(2) Using $\EVp_{3 4} = q^{(\omega_s, \omega_s + 2 \rho)} \EV_{2 3} \opSt_{2 3 4}$ from \eqref{eq:S-evaluation} we write
\[
p \metMP = \EVp_{3 4} \EV_{4 5} p \delbar p \otimes \del p = q^{(\omega_s, \omega_s + 2 \rho)} \EV_{2 3} \opSt_{2 3 4} \EV_{4 5} p \delbar p \otimes \del p.
\]
Since $\opT_{1 2 3 4} = \opSt_{2 3 4} \opS_{1 2 3}$ and $\opS^{-1}_{1 2 3} p \delbar p = q^{-(\omega_s, \omega_s)} p \delbar p$ from \eqref{eq:S-action-left}, we get
\[
\begin{split}
p \metMP & = q^{(\omega_s, \omega_s + 2 \rho)} \EV_{2 3} \opT_{1 2 3 4} \opS_{1 2 3}^{-1} \EV_{4 5} p \delbar p \otimes \del p \\
& = q^{(\omega_s, \omega_s + 2 \rho)} \EV_{2 3} \opT_{1 2 3 4} \EV_{4 5} \opS_{1 2 3}^{-1} p \delbar p \otimes \del p \\
& = q^{(\omega_s, 2 \rho)} \EV_{2 3} \opT_{1 2 3 4} \EV_{4 5} p \delbar p \otimes \del p.
\end{split}
\]
Next, we can use the identity $\opT_{1 2 3 4} \EV_{4 5} = \EV_{2 3} \opT_{3 4 5 6} \opT_{1 2 3 4}$ from \eqref{eq:identity-ETT}.
Finally, taking into account the right $\subalg$-module relations \eqref{eq:right-module} and $\EV_{2 3} \del p p = \del p$, we get
\[
\begin{split}
p \metMP & = q^{(\omega_s, 2 \rho)} \EV_{2 3} \EV_{2 3} \opT_{3 4 5 6} \opT_{1 2 3 4} p \delbar p \otimes \del p = q^{(\omega_s, 2 \rho)} \EV_{2 3} \EV_{2 3} \delbar p \otimes \del p p \\
& = q^{(\omega_s, 2 \rho)} \EV_{2 3} \EV_{4 5} \delbar p \otimes \del p p = q^{(\omega_s, 2 \rho)} \EV_{2 3} \delbar p \otimes \del p. \qedhere
\end{split}
\]
\end{proof}

The next two identities we discuss appear in the proof of \cite[Proposition 3.11]{heko}.
The first one lets us rewrite $\del \delbar p$ as a product of one-forms.

\begin{lemma}
\label{lem:deldelbar-formula}
We have
\begin{equation}
\label{eq:deldelbarp}
\del \delbar p = \EV_{2 3} (\del p \wedge \delbar p + \delbar p \wedge \del p).
\end{equation}
Moreover we have
\[
\EV_{2 3} p \del \delbar p = \EV_{2 3} \del \delbar p p = \EV_{2 3} \delbar p \wedge \del p.
\]
\end{lemma}

\begin{proof}
Using the identity $\delbar p = \EV_{2 3} p \delbar p$ from \eqref{eq:evaluations-calculi} and the Leibniz rule we have
\[
\del \delbar p = \EV_{2 3} \del(p \delbar p) = \EV_{2 3} (\del p \wedge \delbar p + p \del \delbar p).
\]
Now we apply $\delbar$ to $\EV_{2 3} p \del p = 0$.
Using $\del \delbar = - \delbar \del$ we get $\EV_{2 3} p \del \delbar p = \EV_{2 3} \delbar p \wedge \del p$.
Plugging this into the expression above we get the result for $\del \delbar p$.

For the other identities, using the formula above we compute
\[
\begin{split}
\EV_{2 3} p \del \delbar p & = \EV_{2 3} (E_V)_{4 5} (p \del p \wedge \delbar p + p \delbar p \wedge \del p) \\
& = \EV_{2 3} \EV_{2 3} (p \del p \wedge \delbar p + p \delbar p \wedge \del p) \\
& = \EV_{2 3} \delbar p \wedge \del p.
\end{split}
\]
In the last step we have used \eqref{eq:evaluations-calculi}.
The identity $\EV_{2 3} \del \delbar p p = \EV_{2 3} \delbar p \wedge \del p$ follows similarly.
\end{proof}

The second identity from \cite[Proposition 3.11]{heko} is a right $\subalg$-module relation for $\del \delbar p$.

\begin{lemma}
\label{lem:deldelbar-pp}
We have $\del \delbar p p = \opT_{1 2 3 4} p \del \delbar p$.
\end{lemma}

\begin{proof}
Using the identity for $\del \delbar p$ from \eqref{eq:deldelbarp} and the bimodule relations \eqref{eq:right-module} we get
\[
\del \delbar p p = \EV_{2 3} \opT_{3 4 5 6} \opT_{1 2 3 4} p (\del p \wedge \delbar p + \delbar p \wedge \del p).
\]
Then, using $\EV_{2 3} \opT_{3 4 5 6} \opT_{1 2 3 4} = \opT_{1 2 3 4} \EV_{4 5}$ from \eqref{eq:identity-ETT}, we obtain
\[
\del \delbar p p = \opT_{1 2 3 4} \EV_{4 5} p (\del p \wedge \delbar p + \delbar p \wedge \del p) = \opT_{1 2 3 4} p \del \delbar p. \qedhere
\]
\end{proof}

Finally, the next identity lets us rewrite $\del p \wedge \delbar p$ in terms of $\delbar p \wedge \del p$ under the evaluation $\EV_{2 3}$.
We have used it in the computation of the torsion in \cref{prop:torsion}.

\begin{lemma}
\label{lem:identity-torsion}
We have
\[
q^{(\alpha_s, \alpha_s)} \EV_{2 3} \opT_{1 2 3 4} \delbar p \wedge \del p = - \EV_{2 3} \del p \wedge \delbar p + (q^{(\alpha_s, \alpha_s)} - 1) \EV_{2 3} \delbar p \wedge \del p.
\]
\end{lemma}

\begin{proof}
Applying $\delbar$ to $\del p p = q^{(\alpha_s, \alpha_s)} \opT_{1 2 3 4} p \del p$ gives
\[
\delbar \del p p - \del p \wedge \delbar p = q^{(\alpha_s, \alpha_s)} \opT_{1 2 3 4} \delbar p \wedge \del p + q^{(\alpha_s, \alpha_s)} \opT_{1 2 3 4} p \delbar \del p.
\]
Using \cref{lem:deldelbar-pp} we rewrite this as
\[
q^{(\alpha_s, \alpha_s)} \opT_{1 2 3 4} \delbar p \wedge \del p = - \del p \wedge \delbar p + (1 - q^{(\alpha_s, \alpha_s)}) \delbar \del p p.
\]
Now we use the identity $\EV_{2 3} \del \delbar p p = \EV_{2 3} \delbar p \wedge \del p$ from \cref{lem:deldelbar-formula}, taking into account that $\del \delbar = - \delbar \del$.
This gives the result.
\end{proof}

\section{Bimodule maps}
\label{sec:bimodule-maps}

In this appendix we introduce various bimodule maps, which in the main text were used to define the inverse metric and check the bimodule property of the connections.

\subsection{Inverse metric}

First we introduce certain $\alg$-bimodule maps involving the FODCs $\overcalcP$ and $\overcalcM$ over $\alg$.
We assume that we are in the \emph{quadratic case}, which means that the only relations are as in \eqref{eq:overcalc-quadratic}.
For this part the tensor products are taken over $\bbC$, unless specified.
We denote by $\FovercalcP$ and $\FovercalcM$ the \emph{free} left $\alg$-modules generated by $\del f$ and $\delbar v$ respectively, with $\alg$-bimodule structures given by \eqref{eq:overcalcP-right} and \eqref{eq:overcalcM-right}.

\begin{lemma}
Define the $\alg$-bimodule maps
\[
\PhiPM: \FovercalcP \otimes \FovercalcM \to \alg, \quad
\PhiMP: \FovercalcM \otimes \FovercalcP \to \alg,
\]
by the formulae
\[
\PhiPM(\del f \otimes \delbar v) = \CV_1, \quad
\PhiMP(\delbar v \otimes \del f) = \CVp_1.
\]
Then they descend to maps on the tensor product over $\subalg$.
\end{lemma}

\begin{proof}
To prove that $\PhiPM$ descends to a map $\FovercalcP \otimes_\subalg \FovercalcM \to \alg$ we need to show the equality $\PhiPM(\del f p \otimes \delbar v) = \PhiPM(\del f \otimes p \delbar v)$, where $p = f v \in \subalg$ are the generators of $\subalg$.

Using the relations from \eqref{eq:overcalcP-right} and \eqref{eq:overcalcM-right} we compute
\[
\PhiPM(\del f f \otimes \delbar v)
= q^{(\alpha_s, \alpha_s) - (\omega_s, \omega_s)} (\braid_{V, V})_{1 2} \PhiPM(f \del f \otimes \delbar v)
= q^{(\alpha_s, \alpha_s) - (\omega_s, \omega_s)} (\braid_{V, V})_{1 2} \CV_2 f.
\]
We have $(\braid_{V, V})_{1 2} \CV_2 = (\braid^{-1}_{V, V^*})_{2 3} \CV_1$ from \eqref{eq:coevaluations}. Then
\[
\begin{split}
\PhiPM(\del f f \otimes \delbar v)
& = q^{(\alpha_s, \alpha_s) - (\omega_s, \omega_s)} (\braid^{-1}_{V, V^*})_{2 3} \CV_1 f
= q^{(\alpha_s, \alpha_s) - (\omega_s, \omega_s)} (\braid^{-1}_{V, V^*})_{2 3} \PhiPM(\del f \otimes \delbar v f) \\
& = q^{(\alpha_s, \alpha_s)} \PhiPM(\del f \otimes f \delbar v).
\end{split}
\]
Similarly we compute
\[
\PhiPM(\del f v \otimes \delbar v)
= q^{-(\omega_s, \omega_s)} (\braid^{-1}_{V, V^*})_{1 2} \PhiPM(v \del f \otimes \delbar v)
= q^{-(\omega_s, \omega_s)} (\braid^{-1}_{V, V^*})_{1 2} \CV_2 v.
\]
We have $(\braid^{-1}_{V, V^*})_{1 2} \CV_2 = (\braid_{V^*, V^*})_{2 3} \CV_1$ from \eqref{eq:coevaluations}. Then
\[
\begin{split}
\PhiPM(\del f v \otimes \delbar v)
& = q^{-(\omega_s, \omega_s)} (\braid_{V^*, V^*})_{2 3} \CV_1 v
= q^{-(\omega_s, \omega_s)} (\braid_{V^*, V^*})_{2 3} \PhiPM(\del f \otimes \delbar v v) \\
& = q^{-(\alpha_s, \alpha_s)} \PhiPM(\del f \otimes v \delbar v).
\end{split}
\]
Using these identities and $p = f v$, we obtain that $\PhiPM(\del f p \otimes \delbar v) = \PhiPM(\del f \otimes p \delbar v)$.
The computations for the map $\PhiMP$ are very similar and we omit the details.
\end{proof}

We proceed similarly for the following "multiplication" maps.

\begin{lemma}
Define the $\alg$-bimodule maps
\[
\PsiPM: \FovercalcP \otimes \FovercalcM \to \alg, \quad
\PsiMP: \FovercalcM \otimes \FovercalcP \to \alg,
\]
by the formulae
\[
\PsiPM(\del f \otimes \delbar v) = f v, \quad
\PsiMP(\delbar v \otimes \del f) = v f.
\]
Then they descend to maps on the tensor product over $\subalg$.
\end{lemma}

\begin{proof}
Consider the map $\PsiPM$. Taking into account the relations \eqref{eq:relationsA} we compute
\[
\begin{split}
\PsiPM(\del f f \otimes \delbar v)
& = q^{(\alpha_s, \alpha_s) - (\omega_s, \omega_s)} (\braid_{V, V})_{1 2} \PsiMP(f \del f \otimes \delbar v)
= q^{(\alpha_s, \alpha_s) - (\omega_s, \omega_s)} (\braid_{V, V})_{1 2} f f v \\
& = q^{(\alpha_s, \alpha_s)} f f v
= q^{(\alpha_s, \alpha_s) - (\omega_s, \omega_s)} (\braid^{-1}_{V, V^*})_{2 3} f v f \\
& = q^{(\alpha_s, \alpha_s) - (\omega_s, \omega_s)} (\braid^{-1}_{V, V^*})_{2 3} \PsiPM(\del f \otimes \delbar v f) \\
& = q^{(\alpha_s, \alpha_s)} \PsiPM(\del f \otimes f \delbar v).
\end{split}
\]
Similarly we compute
\[
\begin{split}
\PsiPM(\del f v \otimes \delbar v)
& = q^{- (\omega_s, \omega_s)} (\braid^{-1}_{V, V^*})_{1 2} \PsiMP(v \del f \otimes \delbar v)
= q^{- (\omega_s, \omega_s)} (\braid^{-1}_{V, V^*})_{1 2} v f v \\
& = f v v
= q^{- (\omega_s, \omega_s)} (\braid_{V^*, V^*})_{2 3} f v v
= q^{- (\omega_s, \omega_s)} (\braid_{V^*, V^*})_{2 3} \PsiPM(\del f \otimes \delbar v v) \\
& = q^{-(\alpha_s, \alpha_s)} \PsiPM(\del f \otimes v \delbar v).
\end{split}
\]
Since $p = f v$, these identities show that $\PsiPM(\del f p \otimes \delbar v) = \PsiPM(\del f \otimes p \delbar v)$.
The computations for the map $\PsiMP$ are completely analogous and we omit them.
\end{proof}

Finally, we show that certain linear combinations of the maps $\Phi$ and $\Psi$ defined above descend to the FODCs $\overcalcP$ and $\overcalcM$.

\begin{proposition}
\label{prop:bimodule-metric}
Let $\Phi$ and $\Psi$ be the maps defined above.
\begin{enumerate}
\item The map $\PhiPM - \PsiPM$ descends to a map $\overcalcP \otimes_\subalg \overcalcM \to \alg$.
\item The map $\PhiMP - q^{-(\omega_s, 2 \rho)} \PsiPM$ descends to a map $\overcalcM \otimes_\subalg \overcalcP \to \alg$.
\end{enumerate}
\end{proposition}

\begin{proof}
(1) We need to check that the relations of $\overcalcP$ and $\overcalcM$ are preserved under these maps.
According to \eqref{eq:overcalc-quadratic} we need to consider $\EV_{1 2} v \del f = 0$ and $\EVp_{1 2} f \delbar v = 0$.

Using the duality relations \eqref{eq:duality} and the relations of $\alg$ from \eqref{eq:relationsA} we compute
\[
\EV_{1 2} (\PhiPM - \PsiPM)(v \del f \otimes \delbar v) = \EV_{1 2} \CV_2 v - \EV_{1 2} v f v = v - v = 0.
\]
Similarly, using the same relations together with \eqref{eq:overcalcM-right} and \eqref{eq:identity-EEp}, we compute
\[
\begin{split}
\EVp_{2 3} (\PhiPM - \PsiPM)(\del f \otimes f \delbar v)
& = q^{-(\omega_s, \omega_s)} \EVp_{2 3} (\braid_{V, V^*}^{-1})_{2 3} (\PhiPM - \PsiPM)(\del f \otimes \delbar v f) \\
& = q^{(\omega_s, 2 \rho)} \EV_{2 3} (\PhiPM - \PsiPM)(\del f \otimes \delbar v f) \\
& = q^{(\omega_s, 2 \rho)} \EV_{2 3} \CV_1 f - q^{(\omega_s, 2 \rho)} \EV_{2 3} f v f \\
& = q^{(\omega_s, 2 \rho)} f - q^{(\omega_s, 2 \rho)} f = 0.
\end{split}
\]
This shows that $\PhiPM - \PsiPM$ descends to a map $\overcalcP \otimes_\subalg \overcalcM \to \alg$.

(2) The computations for $\PhiMP - q^{-(\omega_s, 2 \rho)} \PsiPM$ are very similar and we omit them.
\end{proof}

\subsection{Bimodule connections}

Consider the terms $\sigma_{a b}$ with $a, b \in \{ +, - \}$ appearing in \cref{lem:connM-bimodule} and \cref{lem:connP-bimodule}.
Our goal is to show that they correspond to $\subalg$-bimodule maps $\calc_a \otimes_\subalg \calc_b \to \calc_b \otimes_\subalg \calc_a$ defined by the same expressions.

\begin{proposition}
\label{prop:sigmaPP}
We have a $\subalg$-bimodule map $\sigmaPP: \calcP \otimes_\subalg \calcP \to \calcP \otimes_\subalg \calcP$ given by
\[
\sigmaPP(\del p \otimes \del p) = q^{(\alpha_s, \alpha_s)} \opT_{1 2 3 4} \del p \otimes \del p.
\]
\end{proposition}

\begin{proof}
First we check that $\sigmaPP$ is well-defined as a map $\calcP \otimes \calcP \to \calcP \otimes \calcP$.
We treat the first factor as a left $\subalg$-module using \eqref{eq:S-action-left} and the second factor as a right $\subalg$-module using \eqref{eq:S-action-right}.
Using this description we need to check the relations
\[
\begin{gathered}
\EVp_{1 2} \sigmaPP(\del p \otimes \del p) = 0, \quad
\EVp_{3 4} \sigmaPP(\del p \otimes \del p) = 0, \\
(\opSt_{2 3 4} - q^{-(\omega_s, \omega_s)}) \sigmaPP(p \del p \otimes \del p) = 0, \quad
(\opSt_{4 5 6} - q^{-(\omega_s, \omega_s)}) \sigmaPP(\del p \otimes \del p p) = 0.
\end{gathered}
\]
Once this is done, we check that $\sigmaPP$ descends to a map $\calcP \otimes_\subalg \calcP \to \calcP \otimes_\subalg \calcP$.

It is convenient to rewrite $\sigmaPP$ in a slightly different form.
Using $\opSt_{2 3 4} p \del p = q^{-(\omega_s, \omega_s)} p \del p$ from \eqref{eq:S-action-left} allows us to obtain the identity
\[
\begin{split}
\del p \otimes \del p
& = \EV_{2 3} \del p p \otimes \del p = \EV_{2 3} \del p \otimes p \del p
= q^{(\omega_s, \omega_s)} \EV_{2 3} \opSt_{4 5 6} \del p \otimes p \del p \\
& = q^{(\omega_s, \omega_s)} \opSt_{2 3 4} \EV_{2 3} \del p p \otimes \del p
= q^{(\omega_s, \omega_s)} \opSt_{2 3 4} \del p \otimes \del p.
\end{split}
\]
Using this identity we can rewrite
\[
\sigmaPP(\del p \otimes \del p) = q^{(\alpha_s, \alpha_s)} \opS_{1 2 3} \opSt_{2 3 4} \del p \otimes \del p
= q^{(\alpha_s, \alpha_s) - (\omega_s, \omega_s)} \opS_{1 2 3} \del p \otimes \del p.
\]
Now we proceed with the verifications.

$\bullet$ \textbf{Action of $\EVp_{1 2}$}.
Consider the expression $\sigmaPP(\del p \otimes \del p) = q^{(\alpha_s, \alpha_s)} \opT_{1 2 3 4} \del p \otimes \del p$.
We have the identity $\EVp_{1 2} \opS_{1 2 3} \opSt_{2 3 4} = \EVp_{3 4}$ from \eqref{eq:S-evaluation2}. Then
\[
\EVp_{1 2} \sigmaPP(\del p \otimes \del p) = q^{(\alpha_s, \alpha_s)} \EVp_{3 4} \del p \otimes \del p = 0.
\]

$\bullet$ \textbf{Action of $\EVp_{3 4}$}.
Consider the expression $\sigmaPP(\del p \otimes \del p) = q^{(\alpha_s, \alpha_s) - (\omega_s, \omega_s)} \opS_{1 2 3} \del p \otimes \del p$.
Using $\del p \otimes \del p = q^{(\omega_s, \omega_s)} \opSt_{2 3 4} \del p \otimes \del p$ we rewrite this as
\[
\sigmaPP(\del p \otimes \del p) = q^{(\alpha_s, \alpha_s) - 2 (\omega_s, \omega_s)} \opSt_{2 3 4}^{-1} \opS_{1 2 3} \del p \otimes \del p.
\]
Next, using $\EVp_{3 4} \opSt_{2 3 4}^{-1} = q^{(\omega_s, \omega_s + 2 \rho)} \EV_{2 3}$ from \eqref{eq:S-evaluation} we get
\[
\EVp_{3 4} \sigmaPP(\del p \otimes \del p) = q^{(\omega_s, 2 \rho)} q^{(\alpha_s, \alpha_s) - (\omega_s, \omega_s)} \EV_{2 3} \opS_{1 2 3} \del p \otimes \del p.
\]
Now we use the quadratic condition $\opS_{1 2 3} = q^{2 (\omega_s, \omega_s) - (\alpha_s, \alpha_s)} \opS_{1 2 3}^{-1} + q^{(\omega_s, \omega_s)} (1 - q^{- (\alpha_s, \alpha_s)})$ from \eqref{eq:S-quadratic1}. Then we obtain
\[
\begin{split}
\EVp_{3 4} \sigmaPP(\del p \otimes \del p)
& = q^{(\omega_s, 2 \rho)} q^{(\omega_s, \omega_s)} \EV_{2 3} \opS_{1 2 3}^{-1} \del p \otimes \del p \\
& + q^{(\omega_s, 2 \rho)} q^{(\alpha_s, \alpha_s)} (1 - q^{- (\alpha_s, \alpha_s)}) \EV_{2 3} \del p \otimes \del p.
\end{split}
\]
The second term vanishes due to $\EV_{2 3} \del p \otimes \del p = 0$.
Finally using $\EV_{2 3} \opS_{1 2 3}^{-1} = q^{-(\omega_s, \omega_s + 2 \rho)} \EVp_{1 2}$ from \eqref{eq:S-evaluation} we conclude that
\[
\EVp_{3 4} \sigmaPP(\del p \otimes \del p) = \EVp_{1 2} \del p \otimes \del p = 0.
\]

$\bullet$ \textbf{Action of $\opSt_{2 3 4}$}.
Using the expression $\sigmaPP(\del p \otimes \del p) = q^{(\alpha_s, \alpha_s) - (\omega_s, \omega_s)} \opS_{1 2 3} \del p \otimes \del p$ and the identity $\opSt_{2 3 4} \del p \otimes \del p = q^{-(\omega_s, \omega_s)} \del p \otimes \del p$ we easily get
\[
\opSt_{2 3 4} \sigmaPP(p \del p \otimes \del p) = q^{-(\omega_s, \omega_s)} \sigmaPP(p \del p \otimes \del p).
\]

$\bullet$ \textbf{Action of $\opSt_{4 5 6}$}.
As for the case of $\opSt_{2 3 4}$ we immediately get
\[
\opSt_{4 5 6} \sigmaPP(\del p \otimes \del p p) = q^{-(\omega_s, \omega_s)} \sigmaPP(\del p \otimes \del p p).
\]

$\bullet$ \textbf{Tensor product}.
Consider the expression $\sigmaPP(\del p \otimes \del p) = q^{(\alpha_s, \alpha_s) - (\omega_s, \omega_s)} \opS_{1 2 3} \del p \otimes \del p$.
Then using $\del p p = q^{(\alpha_s, \alpha_s) - (\omega_s, \omega_s)} \opS_{1 2 3} p \del p$ from \eqref{eq:S-right-module} we compute
\[
\begin{split}
\sigmaPP(\del p p \otimes \del p) & = q^{(\alpha_s, \alpha_s) - (\omega_s, \omega_s)} \opS_{1 2 3} \sigmaPP(p \del p \otimes \del p)
= q^{2 (\alpha_s, \alpha_s) - 2 (\omega_s, \omega_s)} \opS_{1 2 3} \opS_{3 4 5} p \del p \otimes \del p \\
& = \opS_{1 2 3} \opS_{3 4 5} \opS_{1 2 3}^{-1} \opS_{3 4 5}^{-1} \del p \otimes \del p p.
\end{split}
\]
Using the "braid equation" for $\opS$ from \eqref{eq:S-braid-equation} we obtain
\[
\begin{split}
\sigmaPP(\del p p \otimes \del p)
& = \opS_{3 4 5}^{-1} \opS_{1 2 3} \opS_{3 4 5} \opS_{3 4 5}^{-1} \del p \otimes \del p p
= q^{(\omega_s, \omega_s) - (\alpha_s, \alpha_s)} \opS_{3 4 5}^{-1} \sigmaPP(\del p \otimes \del p p) \\
& = \sigmaPP(\del p \otimes p \del p).
\end{split}
\]
This shows that $\sigmaPP$ descends to a map $\calcP \otimes_\subalg \calcP \to \calcP \otimes_\subalg \calcP$.
\end{proof}

The case of $\sigmaMM$ is fairly similar.

\begin{proposition}
\label{prop:sigmaMM}
We have a $\subalg$-bimodule map $\sigmaMM: \calcM \otimes_\subalg \calcM \to \calcM \otimes_\subalg \calcM$ given by
\[
\sigmaMM(\delbar p \otimes \delbar p) = q^{-(\alpha_s, \alpha_s)} \opT_{1 2 3 4} \delbar p \otimes \delbar p.
\]
\end{proposition}

\begin{proof}
We skip the computations, as they are quite similar to the case of $\sigmaPP$.
To verify that $\sigmaMM$ is well-defined as a map $\calcM \otimes \calcM \to \calcM \otimes \calcM$ we have to check the relations
\[
\begin{gathered}
\EVp_{1 2} \sigmaMM(\delbar p \otimes \delbar p) = 0, \quad
\EVp_{3 4} \sigmaMM(\delbar p \otimes \delbar p) = 0, \\
(\opS_{1 2 3} - q^{(\omega_s, \omega_s)}) \sigmaMM(p \delbar p \otimes \delbar p) = 0, \quad
(\opS_{3 4 5} - q^{(\omega_s, \omega_s)}) \sigmaMM(\delbar p \otimes \delbar p p) = 0.
\end{gathered}
\]
One can show the identity $\opS_{1 2 3} \delbar p \otimes \delbar p = q^{(\omega_s, \omega_s)} \delbar p \otimes \delbar p$.
This allows us to rewrite
\[
\sigmaMM(\delbar p \otimes \delbar p) = q^{(\omega_s, \omega_s) - (\alpha_s, \alpha_s)} \opSt_{2 3 4} \delbar p \otimes \delbar p.
\]
The various verifications are then performed as in \cref{prop:sigmaPP}.
\end{proof}

Next we consider the term $\sigmaPM$, which is more involved.

\begin{proposition}
\label{prop:sigmaPM}
We have a $\subalg$-bimodule map $\sigmaPM: \calcP \otimes_\subalg \calcM \to \calcM \otimes_\subalg \calcP$ given by
\[
\sigmaPM(\del p \otimes \delbar p) = q^{2 (\alpha_s, \alpha_s) - 2 (\omega_s, \omega_s)} \opS_{1 2 3} \opSt_{2 3 4}^{-1} \delbar p \otimes \del p - (q^{(\alpha_s, \alpha_s)} - 1) q^{(\alpha_s, \alpha_s)} q^{-(\omega_s, 2 \rho)} p \metMP p.
\]
\end{proposition}

\begin{proof}
First we check that $\sigmaPM$ is well-defined as a $\subalg$-bimodule map $\calcP \otimes \calcM \to \calcM \otimes \calcP$.
In other words, we need to check that the following identities hold
\[
\begin{gathered}
\EVp_{1 2} \sigmaPM(\del p \otimes \delbar p) = 0, \quad
\EVp_{3 4} \sigmaPM(\del p \otimes \delbar p) = 0, \\
(\opSt_{2 3 4} - q^{-(\omega_s, \omega_s)}) \sigmaPM(p \del p \otimes \delbar p) = 0, \quad
(\opS_{3 4 5} - q^{(\omega_s, \omega_s)}) \sigmaPM(\del p \otimes \delbar p p) = 0.
\end{gathered}
\]
Then we check that $\sigmaPM$ descends to a map $\calcP \otimes_\subalg \calcM \to \calcM \otimes_\subalg \calcP$.

$\bullet$ \textbf{Action of $\EVp_{1 2}$}.
We use $\EVp_{1 2} \opS_{1 2 3} = q^{(\omega_s, \omega_s + 2 \rho)} \EV_{2 3}$ from \eqref{eq:S-evaluation} and the identity $\opSt_{2 3 4}^{-1} = q^{2 (\omega_s, \omega_s) - (\alpha_s, \alpha_s)} \opSt_{2 3 4} + (1 - q^{-(\alpha_s, \alpha_s)}) q^{(\omega_s, \omega_s)}$ from \eqref{eq:S-quadratic2}. We get
\[
\begin{split}
\EVp_{1 2} \opS_{1 2 3} \opSt_{2 3 4}^{-1} \delbar p \otimes \del p
& = q^{(\omega_s, \omega_s + 2 \rho)} \EV_{2 3} \opSt_{2 3 4}^{-1} \delbar p \otimes \del p \\
& = q^{2 (\omega_s, \omega_s) - (\alpha_s, \alpha_s)} q^{(\omega_s, \omega_s + 2 \rho)} \EV_{2 3} \opSt_{2 3 4} \delbar p \otimes \del p \\
& + (1 - q^{-(\alpha_s, \alpha_s)}) q^{(\omega_s, \omega_s)} q^{(\omega_s, \omega_s + 2 \rho)} \EV_{2 3} \delbar p \otimes \del p.
\end{split}
\]
The first term vanishes since, $\EV_{2 3} \opSt_{2 3 4} = q^{-(\omega_s, \omega_s + 2 \rho)} \EVp_{3 4}$ by \eqref{eq:S-evaluation} and $\EVp_{1 2} \del p = 0$.
For the second term we use $\EV_{2 3} \delbar p \otimes \del p = q^{-(\omega_s, 2 \rho)} p \metMP$ from \eqref{eq:identity-metMP}. Then
\[
\EVp_{1 2} \opS_{1 2 3} \opSt_{2 3 4}^{-1} \delbar p \otimes \del p = (1 - q^{-(\alpha_s, \alpha_s)}) q^{2 (\omega_s, \omega_s)} p \metMP.
\]
Finally using $\EVp_{1 2} p = q^{(\omega_s, 2 \rho)}$ we obtain
\[
\EVp_{1 2} \sigmaPM(\del p \otimes \delbar p) = q^{2 (\alpha_s, \alpha_s)} (1 - q^{-(\alpha_s, \alpha_s)}) p \metMP - (q^{(\alpha_s, \alpha_s)} - 1) q^{(\alpha_s, \alpha_s)} \metMP p = 0.
\]

$\bullet$ \textbf{Action of $\EVp_{3 4}$}.
We use $\EVp_{3 4} \opSt_{2 3 4}^{-1} = q^{(\omega_s, \omega_s + 2 \rho)} \EV_{2 3}$ from \eqref{eq:S-evaluation} and the identity $\opS_{1 2 3} = q^{2 (\omega_s, \omega_s) - (\alpha_s, \alpha_s)} \opS_{1 2 3}^{-1} + (1 - q^{-(\alpha_s, \alpha_s)}) q^{(\omega_s, \omega_s)}$ from \eqref{eq:S-quadratic1}. We get
\[
\begin{split}
\EVp_{3 4} \opSt_{2 3 4}^{-1} \opS_{1 2 3} \delbar p \otimes \del p
& = q^{(\omega_s, \omega_s + 2 \rho)} \EV_{2 3} \opS_{1 2 3} \delbar p \otimes \del p \\
& = q^{2 (\omega_s, \omega_s) - (\alpha_s, \alpha_s)} q^{(\omega_s, \omega_s + 2 \rho)} \EV_{2 3} \opS_{1 2 3}^{-1} \delbar p \otimes \del p \\
& + (1 - q^{-(\alpha_s, \alpha_s)}) q^{(\omega_s, \omega_s)} q^{(\omega_s, \omega_s + 2 \rho)} \EV_{2 3} \delbar p \otimes \del p.
\end{split}
\]
The first term vanishes, since $\EV_{2 3} \opS_{1 2 3}^{-1} = q^{-(\omega_s, \omega_s + 2 \rho)} \EVp_{1 2}$ and $\EVp_{1 2} \delbar p = 0$. Then
\[
\EVp_{3 4} \opSt_{2 3 4}^{-1} \opS_{1 2 3} \delbar p \otimes \del p = (1 - q^{-(\alpha_s, \alpha_s)}) q^{2 (\omega_s, \omega_s)} p \metMP.
\]
Therefore we obtain
\[
\EVp_{3 4} \sigmaPM(\del p \otimes \delbar p) = q^{2 (\alpha_s, \alpha_s)} (1 - q^{-(\alpha_s, \alpha_s)}) p \metMP - (q^{(\alpha_s, \alpha_s)} - 1) q^{(\alpha_s, \alpha_s)} p \metMP = 0.
\]

$\bullet$ \textbf{Action of $\opSt_{2 3 4}$}.
Let us write $\elemMPL = \opSt_{4 5 6}^{-1} p \delbar p \otimes \del p$.
Using the relations \eqref{eq:S-action-left}, \eqref{eq:S-right-module} and the "braid equation" for $\opSt$ \eqref{eq:S-braid-equation} we compute
\[
\begin{split}
\opSt_{2 3 4}^{-1} \elemMPL
& = \opSt_{2 3 4}^{-1} \opSt_{4 5 6}^{-1} p \delbar p \otimes \del p
= q^{(\alpha_s, \alpha_s) - (\omega_s, \omega_s)} \opSt_{2 3 4}^{-1} \opSt_{4 5 6}^{-1} \opSt_{2 3 4}^{-1} \delbar p \otimes p \del p \\
& = q^{(\alpha_s, \alpha_s) - (\omega_s, \omega_s)} \opSt_{4 5 6}^{-1} \opSt_{2 3 4}^{-1} \opSt_{4 5 6}^{-1} \delbar p \otimes p \del p
= q^{(\alpha_s, \alpha_s)} \opSt_{4 5 6}^{-1} \opSt_{2 3 4}^{-1} \delbar p p \otimes \del p \\
& = q^{(\omega_s, \omega_s)} \opSt_{4 5 6}^{-1} p \delbar p \otimes \del p = q^{(\omega_s, \omega_s)} \elemMPL.
\end{split}
\]
Now we rewrite $\sigmaPM(p \del p \otimes \delbar p)$ in terms of $\elemMPL$ as
\[
\sigmaPM(p \del p \otimes \delbar p)
= q^{2 (\alpha_s, \alpha_s) - 2 (\omega_s, \omega_s)} \opS_{3 4 5} \elemMPL - (q^{(\alpha_s, \alpha_s)} - 1) q^{(\alpha_s, \alpha_s)} q^{-(\omega_s, 2 \rho)} p p \metMP p.
\]
Then using $\opSt_{2 3 4} \elemMPL = q^{-(\omega_s, \omega_s)} \elemMPL$ and $\opSt_{2 3 4} p p = q^{-(\omega_s, \omega_s)} p p$ we obtain
\[
(\opSt_{2 3 4} - q^{-(\omega_s, \omega_s)}) \sigmaPM(p \del p \otimes \delbar p) = 0.
\]

$\bullet$ \textbf{Action of $\opS_{3 4 5}$}.
Let us write $\elemMPR = \opS_{1 2 3} \delbar p \otimes \del p p$. As above we compute
\[
\begin{split}
\opS_{3 4 5} \elemMPR
& = \opS_{3 4 5} \opS_{1 2 3} \delbar p \otimes \del p p
= q^{(\alpha_s, \alpha_s) - (\omega_s, \omega_s)} \opS_{3 4 5} \opS_{1 2 3} \opS_{3 4 5} \delbar p p \otimes \del p \\
& = q^{(\alpha_s, \alpha_s) - (\omega_s, \omega_s)} \opS_{1 2 3} \opS_{3 4 5} \opS_{1 2 3} \delbar p p \otimes \del p
= q^{(\alpha_s, \alpha_s)} \opS_{1 2 3} \opS_{3 4 5} \delbar p \otimes p \del p \\
& = q^{(\omega_s, \omega_s)} \opS_{1 2 3} \delbar p \otimes \del p p = q^{(\omega_s, \omega_s)} \elemMPR.
\end{split}
\]
Now we rewrite $\sigmaPM(\del p \otimes \delbar p p)$ in terms of $\elemMPR$ as
\[
\sigmaPM(\del p \otimes \delbar p p)
= q^{2 (\alpha_s, \alpha_s) - 2 (\omega_s, \omega_s)} \opSt_{2 3 4}^{-1} \elemMPR - (q^{(\alpha_s, \alpha_s)} - 1) q^{(\alpha_s, \alpha_s)} q^{-(\omega_s, 2 \rho)} p \metMP p p.
\]
Then using $\opS_{3 4 5} \elemMPR = q^{(\omega_s, \omega_s)} \elemMPR$ and $\opS_{1 2 3} p p = q^{(\omega_s, \omega_s)} p p$ we obtain
\[
(\opS_{3 4 5} - q^{(\omega_s, \omega_s)}) \sigmaPM(\del p \otimes \delbar p p) = 0.
\]

$\bullet$ \textbf{Tensor product}.
We want to show that $\sigmaPM(\del p p \otimes \delbar p) = \sigmaPM(\del p \otimes p \delbar p)$.
Using the right $\subalg$-module relations \eqref{eq:S-right-module} we compute
\[
\begin{split}
\sigmaPM(\del p p \otimes \delbar p)
& = q^{(\alpha_s, \alpha_s) - (\omega_s, \omega_s)} \opS_{1 2 3} \sigmaPM(p \del p \otimes \delbar p) \\
& = q^{3 (\alpha_s, \alpha_s) - 3 (\omega_s, \omega_s)} \opS_{1 2 3} \opS_{3 4 5} \opSt_{4 5 6}^{-1} p \delbar p \otimes \del p \\
& - q^{(\alpha_s, \alpha_s) - (\omega_s, \omega_s)} (q^{(\alpha_s, \alpha_s)} - 1) q^{(\alpha_s, \alpha_s)} q^{-(\omega_s, 2 \rho)} \opS_{1 2 3} p p \metMP p.
\end{split}
\]
We focus on the first term. Using again \eqref{eq:S-right-module} we have
\[
\begin{split}
\opS_{1 2 3} \opS_{3 4 5} \opSt_{4 5 6}^{-1} p \delbar p \otimes \del p
& = \opS_{1 2 3} \opS_{3 4 5} \opSt_{4 5 6}^{-1} \opSt_{2 3 4}^{-1} \opS_{3 4 5}^{-1} \delbar p \otimes \del p p \\
& = \opSt_{4 5 6}^{-1} \opS_{1 2 3} \opSt_{2 3 4}^{-1} \delbar p \otimes \del p p.
\end{split}
\]
We also used that the terms $\opS$ and $\opSt$ commute.
For the second term we have
\[
\opS_{1 2 3} p p \metMP p = q^{(\omega_s, \omega_s)} p p \metMP p = q^{(\omega_s, \omega_s)} p \metMP p p = \opSt_{4 5 6}^{-1} p \metMP p p.
\]
Using these identities we obtain
\[
\sigmaPM(\del p p \otimes \delbar p) = q^{(\alpha_s, \alpha_s) - (\omega_s, \omega_s)} \opSt_{4 5 6}^{-1} \sigmaPM(\del p \otimes \delbar p p)  = \sigmaPM(\del p \otimes p \delbar p). \qedhere
\]
\end{proof}

Finally we consider $\sigmaMP$, which is similar to $\sigmaPM$.
However, since the details are fairly involved, we include the necessary steps also in this case.

\begin{proposition}
\label{prop:sigmaMP}
We have a $\subalg$-bimodule map $\sigmaMP: \calcM \otimes_\subalg \calcP \to \calcP \otimes_\subalg \calcM$ given by
\[
\sigmaMP(\delbar p \otimes \del p) = q^{2 (\omega_s, \omega_s) - 2 (\alpha_s, \alpha_s)} \opS_{1 2 3}^{-1} \opSt_{2 3 4} \del p \otimes \delbar p - (q^{-(\alpha_s, \alpha_s)} - 1) q^{-(\omega_s, 2 \rho)} p \metPM p.
\]
\end{proposition}

\begin{proof}
We need to verify the following relations
\[
\begin{gathered}
\EVp_{1 2} \sigmaMP(\delbar p \otimes \del p) = 0, \quad
\EVp_{3 4} \sigmaMP(\delbar p \otimes \del p) = 0, \\
(\opS_{1 2 3} - q^{(\omega_s, \omega_s)}) \sigmaMP(p \delbar p \otimes \del p) = 0, \quad
(\opSt_{4 5 6} - q^{-(\omega_s, \omega_s)}) \sigmaMP(\delbar p \otimes \del p p) = 0.
\end{gathered}
\]
Then we check that $\sigmaMP$ descends to a map $\calcM \otimes_\subalg \calcP \to \calcP \otimes_\subalg \calcM$.

$\bullet$ \textbf{Action of $\EVp_{1 2}$}.
We have the identity $\opS_{1 2 3}^{-1} = q^{(\alpha_s, \alpha_s) - 2 (\omega_s, \omega_s)} \opS_{1 2 3} + (1 - q^{(\alpha_s, \alpha_s)}) q^{-(\omega_s, \omega_s)}$ from \eqref{eq:S-quadratic1}. Then we can rewrite
\[
\begin{split}
\EVp_{1 2} \opS_{1 2 3}^{-1} \opSt_{2 3 4} \del p \otimes \delbar p
& = q^{(\alpha_s, \alpha_s) - 2 (\omega_s, \omega_s)} \EVp_{1 2} \opS_{1 2 3} \opSt_{2 3 4} \del p \otimes \delbar p \\
& + (1 - q^{(\alpha_s, \alpha_s)}) q^{-(\omega_s, \omega_s)} \EVp_{1 2} \opSt_{2 3 4} \del p \otimes \delbar p.
\end{split}
\]
The first term vanishes, since $\EVp_{1 2} \opS_{1 2 3} \opSt_{2 3 4} = \EVp_{3 4}$ from \eqref{eq:S-evaluation2} and $\EVp_{1 2} \delbar p = 0$.
For the second term we use $\metPM p = q^{(\omega_s, \omega_s) - (\alpha_s, \alpha_s)} \EVp_{1 2} \opSt_{2 3 4} \del p \otimes \delbar p$ from \eqref{eq:identity-metMP}. Then
\[
\EVp_{1 2} \opS_{1 2 3}^{-1} \opSt_{2 3 4} \del p \otimes \delbar p = (q^{-(\alpha_s, \alpha_s)} - 1) q^{2 (\alpha_s, \alpha_s) - 2 (\omega_s, \omega_s)} \metPM p.
\]
Using this identity we conclude that
\[
\EVp_{1 2} \sigmaMP(\delbar p \otimes \del p) = (q^{-(\alpha_s, \alpha_s)} - 1) \metPM p - (q^{-(\alpha_s, \alpha_s)} - 1) \metPM p = 0.
\]

$\bullet$ \textbf{Action of $\EVp_{3 4}$}.
We have the identity $\opSt_{2 3 4} = q^{(\alpha_s, \alpha_s) - 2 (\omega_s, \omega_s)} \opSt_{2 3 4}^{-1} + (1 - q^{(\alpha_s, \alpha_s)}) q^{-(\omega_s, \omega_s)}$ from \eqref{eq:S-quadratic2}. Then we can rewrite
\[
\begin{split}
\EVp_{3 4} \opSt_{2 3 4} \opS_{1 2 3}^{-1} \del p \otimes \delbar p
& = q^{(\alpha_s, \alpha_s) - 2 (\omega_s, \omega_s)} \EVp_{3 4} \opSt_{2 3 4}^{-1} \opS_{1 2 3}^{-1} \del p \otimes \delbar p \\
& + (1 - q^{(\alpha_s, \alpha_s)}) q^{-(\omega_s, \omega_s)} \EVp_{3 4} \opS_{1 2 3}^{-1} \del p \otimes \delbar p.
\end{split}
\]
The first term vanishes, since $\EVp_{3 4} \opSt_{2 3 4}^{-1} \opS_{1 2 3}^{-1} = \EVp_{1 2}$ from \eqref{eq:S-evaluation2} and $\EVp_{1 2} \del p = 0$.
For the second term we use $p \metPM = q^{(\omega_s, \omega_s) - (\alpha_s, \alpha_s)} \EVp_{3 4} \opS_{1 2 3}^{-1} \del p \otimes \delbar p$ from \eqref{eq:identity-metPM}. Then
\[
\EVp_{3 4} \opSt_{2 3 4} \opS_{1 2 3}^{-1} \del p \otimes \delbar p = (q^{-(\alpha_s, \alpha_s)} - 1) q^{2 (\alpha_s, \alpha_s) - 2(\omega_s, \omega_s)} p \metPM.
\]
Using this identity we conclude that
\[
\EVp_{3 4} \sigmaMP(\delbar p \otimes \del p) = (q^{-(\alpha_s, \alpha_s)} - 1) p \metPM - (q^{-(\alpha_s, \alpha_s)} - 1) p \metPM = 0.
\]

$\bullet$ \textbf{Action of $\opS_{1 2 3}$}.
Let us write $\elemPML = \opS_{3 4 5}^{-1} p \del p \otimes \delbar p$.
Then using the relations \eqref{eq:S-action-left}, \eqref{eq:S-right-module} and the "braid equation" for $\opS$ \eqref{eq:S-braid-equation} we compute
\[
\begin{split}
\opS_{1 2 3}^{-1} \elemPML & = \opS_{1 2 3}^{-1} \opS_{3 4 5}^{-1} p \del p \otimes \delbar p
= q^{(\omega_s, \omega_s) - (\alpha_s, \alpha_s)} \opS_{1 2 3}^{-1} \opS_{3 4 5}^{-1} \opS_{1 2 3}^{-1} \del p \otimes p \delbar p \\
& = q^{(\omega_s, \omega_s) - (\alpha_s, \alpha_s)} \opS_{3 4 5}^{-1} \opS_{1 2 3}^{-1} \opS_{3 4 5}^{-1} \del p \otimes p \delbar p
= q^{- (\alpha_s, \alpha_s)} \opS_{3 4 5}^{-1} \opS_{1 2 3}^{-1} \del p \otimes p \delbar p \\
& = q^{-(\omega_s, \omega_s)} \opS_{3 4 5}^{-1} \del p \otimes p \delbar p = q^{-(\omega_s, \omega_s)} \elemPML.
\end{split}
\]
Now we rewrite $\sigmaMP(p \delbar p \otimes \del p)$ in terms of $\elemPML$ as
\[
\sigmaMP(p \delbar p \otimes \del p)
= q^{2 (\omega_s, \omega_s) - 2 (\alpha_s, \alpha_s)} \opSt_{4 5 6} \elemPML - (q^{-(\alpha_s, \alpha_s)} - 1) q^{-(\omega_s, 2 \rho)} p p \metPM p.
\]
Then using $\opS_{1 2 3} \elemPML = q^{(\omega_s, \omega_s)} \elemPML$ and $\opS_{1 2 3} p p = q^{(\omega_s, \omega_s)} p p$ we obtain
\[
(\opS_{1 2 3} - q^{(\omega_s, \omega_s)}) \sigmaMP(p \delbar p \otimes \del p) = 0.
\]

$\bullet$ \textbf{Action of $\opSt_{4 5 6}$}.
Let us write $\elemPMR = \opSt_{2 3 4} \del p \otimes \delbar p p$. As above we compute
\[
\begin{split}
\opSt_{4 5 6} \elemPMR
& = \opSt_{4 5 6} \opSt_{2 3 4} \del p \otimes \delbar p p
= q^{(\omega_s, \omega_s) - (\alpha_s, \alpha_s)} \opSt_{4 5 6} \opSt_{2 3 4} \opSt_{4 5 6} \del p p \otimes \delbar p \\
& = q^{(\omega_s, \omega_s) - (\alpha_s, \alpha_s)} \opSt_{2 3 4} \opSt_{4 5 6} \opSt_{2 3 4} \del p p \otimes \delbar p
= q^{-(\alpha_s, \alpha_s)} \opSt_{2 3 4} \opSt_{4 5 6} \del p \otimes p \delbar p \\
& = q^{-(\omega_s, \omega_s)} \opSt_{2 3 4} \del p \otimes \delbar p p = q^{-(\omega_s, \omega_s)} \elemPMR.
\end{split}
\]
Now we rewrite $\sigmaMP(\delbar p \otimes \del p p)$ in terms of $\elemPMR$ as
\[
\sigmaMP(\delbar p \otimes \del p p)
= q^{2 (\omega_s, \omega_s) - 2 (\alpha_s, \alpha_s)} \opS_{1 2 3}^{-1} \elemPMR - (q^{-(\alpha_s, \alpha_s)} - 1) q^{-(\omega_s, 2 \rho)} p \metPM p p.
\]
Then using $\opSt_{4 5 6} \elemPMR = q^{-(\omega_s, \omega_s)} \elemPMR$ and $\opSt_{2 3 4} p p = q^{-(\omega_s, \omega_s)} p p$ we obtain
\[
(\opSt_{4 5 6} - q^{-(\omega_s, \omega_s)}) \sigmaMP(\delbar p \otimes \del p p) = 0.
\]

$\bullet$ \textbf{Tensor product}.
We want to show that $\sigmaMP(\delbar p p \otimes \del p) = \sigmaMP(\delbar p \otimes p \del p)$.
Using the right $\subalg$-module relations \eqref{eq:S-right-module} we compute
\[
\begin{split}
\sigmaMP(\delbar p p \otimes \del p)
& = q^{(\omega_s, \omega_s) - (\alpha_s, \alpha_s)} \opSt_{2 3 4} \sigmaMP(p \delbar p \otimes \del p) \\
& = q^{3 (\omega_s, \omega_s) - 3 (\alpha_s, \alpha_s)} \opSt_{2 3 4} \opS_{3 4 5}^{-1} \opSt_{4 5 6} p \del p \otimes \delbar p \\ 
& - q^{(\omega_s, \omega_s) - (\alpha_s, \alpha_s)} (q^{-(\alpha_s, \alpha_s)} - 1) q^{-(\omega_s, 2 \rho)} \opSt_{2 3 4} p p \metPM p.
\end{split}
\]
We focus on the first term. Using again \eqref{eq:S-right-module} we have
\[
\begin{split}
\opSt_{2 3 4} \opS_{3 4 5}^{-1} \opSt_{4 5 6} p \del p \otimes \delbar p
& = \opSt_{2 3 4} \opS_{3 4 5}^{-1} \opSt_{4 5 6} \opS_{1 2 3}^{-1} \opSt_{4 5 6}^{-1} \del p \otimes \delbar p p \\
& = \opS_{3 4 5}^{-1} \opS_{1 2 3}^{-1} \opSt_{2 3 4} \del p \otimes \delbar p p.
\end{split}
\]
For the second term we have
\[
\opSt_{2 3 4} p p \metMP p
= q^{-(\omega_s, \omega_s)} p p \metMP p = q^{-(\omega_s, \omega_s)} p \metMP p p = \opS_{3 4 5}^{-1} p \metMP p p.
\]
Using these identities we obtain
\[
\sigmaMP(\delbar p p \otimes \del p)
= q^{(\omega_s, \omega_s) - (\alpha_s, \alpha_s)} \opS_{3 4 5}^{-1} \sigmaMP(\delbar p \otimes \del p p) = \sigmaMP(\delbar p \otimes p \del p). \qedhere
\]
\end{proof}

\end{document}